\documentclass[a4paper, 12pt]{amsart}

\makeatletter
\def\@tocline#1#2#3#4#5#6#7{\relax
  \ifnum #1>\c@tocdepth 
  \else
    \par \addpenalty\@secpenalty\addvspace{#2}%
    \begingroup \hyphenpenalty\@M
    \@ifempty{#4}{%
      \@tempdima\csname r@tocindent\number#1\endcsname\relax
    }{%
      \@tempdima#4\relax
    }%
    \parindent\z@ \leftskip#3\relax \advance\leftskip\@tempdima\relax
    \rightskip\@pnumwidth plus4em \parfillskip-\@pnumwidth
    #5\leavevmode\hskip-\@tempdima
      \ifcase #1
       \or\or \hskip 1em \or \hskip 2em \else \hskip 3em \fi%
      #6\nobreak\relax
    \hfill\hbox to\@pnumwidth{\@tocpagenum{#7}}\par
    \nobreak
    \endgroup
  \fi}
\makeatother

\usepackage{amsmath,amssymb,amsthm,url}
\usepackage[utf8]{inputenc}
\usepackage[T1]{fontenc}
\usepackage{quiver}
\usepackage{libertine}
\usepackage[libertine,cmintegrals,cmbraces]{newtxmath}
\usepackage[shortlabels]{enumitem}
\usepackage{mathtools}
\usepackage{adjustbox}
\usepackage{xcolor}
\usepackage{tikz}
\usepackage{tikz-cd}
\usepackage{nicefrac}
\usepackage{dsfont}
\usepackage{todonotes}
\usepackage{a4wide}

\AtBeginDocument{%
   \def\MR#1{}
} 
\usepackage{euscript}
\usepackage[all]{xy}

\usepackage[pagebackref]{hyperref}
  
\hypersetup{%
  bookmarksnumbered=true,
  colorlinks=true,%
  linkcolor=blue,%
  citecolor=blue,%
  filecolor=blue,%
  menucolor=blue,%
  urlcolor=blue,%
  pdfnewwindow=true,%
  pdfstartview=FitBH}

\usepackage{xcolor}
\usepackage{amsmath}
\definecolor{seagreen}{RGB}{46,139,87}
\definecolor{maroon}{RGB}{128,0,0}
\definecolor{darkviolet}{RGB}{148,0,211}
\definecolor{twelve}{RGB}{100,100,170}
\definecolor{thirteen}{RGB}{100,150,50}
\definecolor{fourteen}{RGB}{200,0,0}
\definecolor{fifteen}{RGB}{0,200,0}
\definecolor{sixteen}{RGB}{0,0,200}
\definecolor{seventeen}{RGB}{200,0,200}
\definecolor{eighteen}{RGB}{0,200,200}

\newcommand{\bb}[1]{\mathbb{#1}}

\renewcommand{\mathrm}[1]{{\mathsf{#1}}}

\newcommand{\s}{\mathrm{Sp}}
\DeclareMathOperator{\T}{\mathrm{S}_\ast}

\DeclareMathOperator{\nat}{\mathsf{nat}}

  \newcommand{\adjunction}[4]{
\xymatrix{
#1:#2 \ar@<.5ex>[r] &
\ar@<.5ex>[l] #3:#4
}}

\newtheorem{thm}{Theorem}[subsection]
\newtheorem{prop}[thm]{Proposition}

\newtheorem{lem}[thm]{Lemma}
\newtheorem{cor}[thm]{Corollary}
\newtheorem{conj}[thm]{Conjecture}

\newtheorem*{thm*}{Theorem}

\theoremstyle{definition}
\newtheorem{example}[thm]{Example}
\newtheorem{remark}[thm]{Remark}
\newtheorem{definition}[thm]{Definition}
\newtheorem{question}[thm]{Question}
\newtheorem{ex}[thm]{Example}

\title{Unstable $1$-semiadditivity as classifying Goodwillie towers}
\author{Connor Malin} \thanks{
Max Planck Institute for Mathematics, Bonn. \texttt{malin@mpim-bonn.mpg.de}}

\begin{document}

\maketitle

\begin{abstract}
  A stable $\infty$-category is $1$-semiadditive if the norms for all finite group actions are equivalences. In the presence of $1$-semiadditivity, Goodwillie calculus simplifies drastically.  We introduce two variants of $1$-semiadditivity for an $\infty$-category $C$ and study their relation to the Goodwillie calculus of functors $C \rightarrow \s(C)$. We demonstrate that these variations of $1$-semiadditivity are complete obstructions to the problem of endowing $\partial_\ast F$ with either a right module or a divided power right module structure which completely classifies the Goodwillie tower of $F$.  We find applications to algebraic localizations of spaces, the Morita theory of operads, and bar-cobar duality of algebras. Along the way, we address several milestones in these areas including: Lie structures in the Goodwillie calculus of spaces, spectral Lie algebra models of $v_h$-periodic homotopy theory, and the Poincaré/Koszul duality of $E_d$-algebras.
\end{abstract}

\tableofcontents

\section{Introduction}
A stable $\infty$-category $C$ is called $1$-semiadditive if for all finite groups $G$ and $G$-actions $X \in C^{BG}$ the Tate construction, defined as the cofiber of the norm
\[X^{tG}:=\mathrm{cofiber}(X_{hG} \rightarrow X^{hG}),\]
is contractible. Suppose $C=L\s$ is a compactly generated Bousfield localization of spectra $\s$, the unit $L\bb{S}$ not necessarily compact. In this case, the condition for $1$-semiadditivity can be weakened to either equivalent condition that should hold for all $n$, $X \in L\s^{B\Sigma_n}$, and $Z \in L\s^\omega$:
 \[(X \wedge Z^{\wedge n})^{t\Sigma_n} \simeq \ast,\]
 \[(X \wedge \partial_n (L\Sigma^\infty \mathrm{Map}(Z,-)  )^\vee)^{t\Sigma_n} \simeq \ast.\]
Here $C^\omega$ denotes the subcategory of compact objects, $\partial_n$ denotes the $n$th Goodwillie derivative in the sense of Goodwillie calculus, and $(-)^\vee$ denotes the local Spanier--Whitehead dual. Suppose now that we drop the assumption that $C:=L\s$ and instead only assume that it is pointed, compactly generated, and equipped with an identification $\s(C)\simeq L\s$. We say $C$ is \textit{infinitesimally $1$-semiadditive} if for all $n$, $X\in L\s^{B\Sigma_n}$, and compact $c \in C^\omega$
\[(X \wedge (\Sigma^\infty_C c)^{\wedge n})^{t\Sigma_n} \simeq \ast.\]
We say $C$ is \textit{differentially $1$-semiadditive} if for all $n$, $X\in L\s^{B\Sigma_n}$, and compact $c \in C^\omega$ that
\[(X \wedge \partial_n (L\Sigma^\infty \mathrm{Map}(c,-)  )^\vee)^{t\Sigma_n} \simeq \ast.\]
We are motivated by the following examples:
\begin{enumerate}
    \item If $O$ is a reduced operad in $L\s$, then $\mathrm{Alg}_O$ is infinitesimally $1$-semiadditive, if and only if $\mathrm{Alg}_O$ is differentially $1$-semiadditive, if and only if $L\s$ is $1$-semiadditive.
    \item The $\infty$-category of spaces $\T$ is differentially $1$-semiadditive, but not infinitesimally $1$-semiadditive.
    \item The $\infty$-category of $v_h$-periodic spaces \cite{telescopichomotopybous} is both infinitesimally $1$-semiadditive and differentially $1$-semiadditive.
\end{enumerate}
We will understand unstable $1$-semiadditivity through the use of operadic right module structures on the Goodwillie derivatives of finitary functors $F \in \mathrm{Fun}^\omega(C,L\s)$. A functor is finitary if it preserves  filtered colimits.

\subsection{Main results}
In stating our results, we freely use terminology associated to operads, their algebras, and their right modules. We will also reference their various theories of Koszul duality. Definitions may be found in Section \ref{section: right module stuff}, Section \ref{section: right module primitive stuff}, and Section \ref{section: goodwillie calculus algebras}

We will assume $C$ is a compactly generated, differentiable\footnote{A differentiable $\infty$-category $C$ is a pointed presentable category in which filtered colimits commute with finite limits.} $\infty$-category. Furthermore, we assume $C$ is equipped with an equivalence $\s(C) \simeq L\s$ to a Bousfield localization of $\s$ and has the property that the symmetric sequence of derivatives $\partial_\ast \mathrm{Id}_C \in \mathrm{SymSeq}(L\s)$ is levelwise dualizable. We call such $\infty$-categories \textit{differentially dualizable}.

If $(D,\otimes)$ is a stable symmetric monoidal $\infty$-category with $L$-local mapping spectra $D(-,-)$, the coendomorphism operad of $d \in (D,\otimes)$ is the $L$-local spectral operad defined via
\[\mathrm{CoEnd}(d)(n):= D(d,d^{\otimes n}),\]
and similarly the endomorphism operad is the $L$-local spectral operad defined via
\[\mathrm{End}(d)(n):= D(d^{\otimes n},d).\]

\begin{thm}
   If $C$ is differentially dualizable, the Goodwillie derivatives lift to a symmetric monoidal functor with respect to the pointwise smash product of functors and Day convolution of right modules
   \[\partial_\ast: (\mathrm{Fun}^\omega(C,L\s),\wedge) \rightarrow (\mathrm{RMod}_{K(\mathrm{CoEnd}(\Sigma^\infty_C))},\circledast).\]
Given $F\in \mathrm{Fun}^\omega(C,L\s)$ such that $\partial_\ast F$ is levelwise dualizable, there is an equivalence of right modules
   \[\partial_\ast F \simeq K(\nat(F,(\Sigma^\infty_C)^{\wedge \ast}))\]
   where $\nat(F,(\Sigma^\infty_C)^{\wedge \ast})$ is a right module over the operad $\mathrm{CoEnd}(\Sigma^\infty_C))$ by postcomposition of natural transformation spectra.
\end{thm}

A natural example to consider is the $\infty$-category $\mathrm{Alg}_O$ where $O$ is a reduced, levelwise dualizable operad in $L\s$. In this case, the stabilization functor is given by the functor of $O$-algebra indecomposables written $\mathrm{Indecom}_O$. Goodwillie calculus allows us to show that the Koszul dual operad $K(O)$ supports the universal operadic coaction on $\mathrm{Indecom}_O(A)$ which is natural in $A \in \mathrm{Alg}_O$.

\begin{thm}\label{thm: first theorem}
    If $O\in \mathrm{Operad}(L\s)$ is a reduced levelwise dualizable operad, then the map of operads \[K(O) \rightarrow \mathrm{CoEnd}(\mathrm{Indecom}_O)\] induced by the coaction of $K(O)$ on $\mathrm{Indecom}_O(A)$ is an equivalence.
\end{thm}

This fact has an immediate application to the Morita theory of operads. Recall that the \textit{Picard group} $\mathrm{Pic}(D,\otimes)$ of a symmetric monoidal $\infty$-category $(D,\otimes)$  is the group of $\otimes$-invertible objects, up to equivalence.

\begin{thm}
If $O,P$ are reduced and levelwise dualizable operads in $L\s$, then the following are equivalent:
   \begin{itemize}
       \item $\mathrm{Alg}_O$ is equivalent to $\mathrm{Alg}_P$.
       \item There is $X \in \mathrm{Pic}(L\s,\wedge)$ and an equivalence of operads
       \[O \wedge \mathrm{End}(X) \simeq P.\]
   \end{itemize}
\end{thm}

When $C= \T$, the $\infty$-category of spaces, Theorem \ref{thm: first theorem} recovers the original right action of the spectral Lie operad $\mathrm{lie}:=K(\mathrm{com})$ on $\partial_\ast F$ discovered by Arone-Ching \cite{ACOperads}. However, there is more structure hidden in the Goodwillie derivatives in the case of $\T$. Namely the derivatives have the structure of a \textit{divided power right module} over $\mathrm{lie}$. Divided powers encode lifts of right module structure maps along $\Sigma_n$-norms, as in Definition \ref{definition: dp rmod}. If we let $\epsilon$ denote the augmentation ideal of an augmented right $\mathrm{com}$-comodule, we have the following description of the derivatives of polynomial functors on $\T$.

\begin{prop}[cf. {\cite[Proposition 6.4]{ACClassification}, \cite[Theorem 3.82]{crosseffectsclassification}}]

        Given $F\in \mathrm{Poly}^\omega(\T,\s)$, there is an equivalence of right $\mathrm{lie}$-modules
    \[\partial_\ast F \simeq B(\epsilon (F|_{\mathrm{FinSet}_\ast})),\]
   and the righthand side naturally admits the structure of a divided power right $\mathrm{lie}$-module which classifies the Goodwillie tower of $F$.
\end{prop}

Before understanding these divided powers, we first consider the Goodwillie calculus of $v_h$-periodic pointed spaces $\T_{v_h}$ which is simpler from a formal perspective. The $\infty$-category $\T_{v_h}$ has the property that its stabilization $\s(\T_{v_h})\simeq \s_{T(h)}$ is $1$-semiadditive \cite{heutsannals,kuhn_2004}. Here $T(h)$ refers to the telescopic spectrum associated to a type $h$ complex in the sense of chromatic homotopy theory. In the presence of $1$-semiadditivity, the categories of divided power right modules and right modules coincide as a consequence of the fact the norm is an equivalence. In this setting, we extend Heuts's recent Lie algebra characterization of $v_h$-periodic spaces \cite{heutsannals} to polynomial functors. Recall that $\Phi: \T_{v_h} \rightarrow \s_{T(h)}$ is the Bousfield-Kuhn functor \cite{telescopichomotopybous} which has a section given by $\Omega^\infty_{T(h)}: \s_{T(h)} \rightarrow \T_{v_h}$.

\begin{thm}
    There is an equivalence of operads in $\s_{T(h)}$
\[\mathrm{lie}_{T(h)} \simeq \mathrm{End}(\Phi)\]
under which Heuts's spectral Lie algebra models are given by the action
\begin{center}
   \[\begin{tikzcd}
	{\T_{v_h}} & {\mathrm{Alg}_{\mathrm{End}(\Phi)}} \\
	X & {\Phi(X)}
	\arrow[from=1-1, to=1-2]
	\arrow[maps to, from=2-1, to=2-2]
	\arrow["{\mathrm{End}(\Phi)}"{pos=0.4}, shift left=4, from=2-2, to=2-2, loop, in=325, out=35, distance=10mm]
\end{tikzcd}\]
\end{center}

The assignment
\[R \mapsto B(R,\mathrm{lie},\Phi(-))\]
induces an equivalence of categories 
    \[  \mathrm{RMod}_{\mathrm{lie}_{T(h)}}^{<\infty}\simeq\mathrm{Poly}^\omega(\T_{v_h},\s_{T(h)}).\]
      
\end{thm}

Instead of fixing $C$ and classifying its Goodwillie towers, we now fix a classification scheme and ask what properties $C$ must satisfy for such a classification to hold. If one asks when the right module structure alone classifies Goodwillie towers, as above, one is led to consider \textit{infinitesimally $1$-semiadditive} categories. Recall this means that for all $n$, $X\in L\s^{B\Sigma_n}$, and $c \in C^\omega$ that
\[(X \wedge (\Sigma^\infty_C c)^{\wedge n})^{t\Sigma_n} \simeq \ast.\]

\begin{thm}
     The following are equivalent for differentially dualizable $C$:  
    \begin{enumerate}
        \item $C$ is infinitesimally $1$-semiadditive. 
        \item $\s(C)$ is $1$-semiadditive.
        \item There is a symmetric monoidal equivalence $(\mathrm{Poly}^\omega(C,L\s),\wedge) \simeq (\mathrm{RMod}^{< \infty}_O,\circledast)$ for some operad $O$.  
        \item The derivative map $\partial_\ast:\mathrm{Fun}^\omega(C,L\s) \rightarrow \mathrm{RMod}_{K(\mathrm{CoEnd}(\Sigma^\infty_C))}$ classifies Goodwillie towers.

    \end{enumerate}
\end{thm}
The question of when divided power right module structures on the derivatives classify Goodwillie towers in $C$ is much more subtle, as the existence of divided powers depends on $C$ and, even if they exist, they need not classify Goodwillie towers.

Recall $C$ is \textit{differentially $1$-semiadditive} if for all $n$, $X\in L\s^{B\Sigma_n}$, and $c \in C^\omega$
\[(X \wedge \partial_n (L\Sigma^\infty \mathrm{Map}(c,-)  )^\vee)^{t\Sigma_n} \simeq \ast.\]
The vanishing of this class of Tate constructions leads to a rich theory of divided power structures on derivatives. We denote the category of divided power right $O$-modules by $\mathrm{RMod}^\mathrm{dp}_O$.
\begin{thm}
         
    The following are equivalent for differentially dualizable $C$:  
    \begin{enumerate}
        \item $C$ is differentially $1$-semiadditive.  
        \item There is a symmetric monoidal equivalence \[(\mathrm{Poly}^\omega(C,L\s),\wedge) \simeq (\mathrm{RMod}^\mathrm{< \infty,dp}_O,\circledast).\]   
        \item The derivatives have a contractible space of lifts \[\partial_\ast:\mathrm{Fun}(C^\omega,L\s) \rightarrow \mathrm{RMod}^\mathrm{dp}_{K(\mathrm{CoEnd}(\Sigma^\infty_C))},\]
        and these classify Goodwillie towers. 
        
    \end{enumerate}
\end{thm}

Unlike the category $\T$ which is differentially $1$-semiadditive, the $\infty$-category $\mathrm{Alg}_O$ associated to $O\in \mathrm{Operad}(L\s)$ resists differential $1$-semiadditivity in a strong way. $\mathrm{Alg}_O$ is differentially $1$-semiadditive, if and only if $L\s$ is $1$-semiadditive. We say $C$ is \textit{differentially algebraic} if for some $O$ there is a symmetric monoidal equivalence
\[(\mathrm{Poly}^\omega(C,L\s),\wedge) \simeq (\mathrm{Poly}^\omega(\mathrm{Alg}_O,L\s),\wedge).\]
This is a partially linearized analogue of the property that $C$ admits a spectral algebraic model $C \simeq \mathrm{Alg}_O$. The property of being differentially algebraic is weaker and is preserved under passage to $n$-connected objects for reasonable notions of connectivity. For example, the category of rational spaces $(\T)_\mathbb{Q}$ is differentially algebraic even though it doesn't embed into $\mathrm{Alg}_{\mathrm{lie}_\mathbb{Q}}$ prior to taking simply connected objects.

We end with addressing whether certain sublocalizations of $\T$ called \textit{differential sublocalizations} (Definition \ref{definition: sublocalization}) are differentially algebraic. For $D \subset C$, a differential sublocalization $D[W^{-1}]$ of $C$ has the property that there is a canonical map
\[\partial_\ast \mathrm{Id}_C \rightarrow \partial_\ast \mathrm{Id}_{D[W^{-1}]}\]
which is identified with the unit of some Bousfield localization of $\s(C)$
\[\partial_\ast \mathrm{Id}_C \rightarrow L'\partial_\ast \mathrm{Id}_C.\]
We demonstrate that differential $1$-semiadditivity is preserved under differential localizations. As a consequence of the theory built up to this point, if we sublocalize $\T$ to force differential algebraicity then, it forces $1$-semiadditivity on the stabilization. 
\begin{cor}
    A differential sublocalization of $\T$ is differentially algebraic, if and only if it is infinitesimally $1$-semiadditive.
\end{cor}

\subsection{Acknowledgments}
 I am grateful to Mark Behrens, Gijs Heuts, Yuqing Shi, and Thomas Blom for many helpful discussions surrounding these topics. I thank Max Blans for formative conversations about $v_h$-periodic Goodwillie calculus at the 2022 EAST Autumn school in Nijmegen. This work grew out of the author's desire to more deeply understand Arone-Ching's remarkable program for studying Goodwillie calculus \cite{ACOperads,ACClassification,crosseffectsclassification,AroneChingManifolds}.

\subsection{Conventions}
By ``category'' we mean $\infty$-category. We write $\T$ for the category of pointed spaces. We write $\s$ for the category of spectra. We write $\mathrm{Map}(-,-)$ to denote the mapping spaces in a category. We consider only Borel $G$-spectra, i.e. spectra with a group action of $G$. We write $L\s^{BG}$ for this category. When defining adjunctions, we always write the left adjoint on the left and the right adjoint on the right.

When we make reference to enriched categories, it is based off Hinich's model of enriched categories \cite{HinichYoneda}. We use this model because an enriched category $C$ admits an enriched Yoneda embedding which is symmetric monoidal with respect to Day convolution when $C$ is symmetric monoidal \cite[8.4.3]{HinichColimits}. By a symmetric monoidal $L\s$-enriched category, we mean a commutative algebra object in $L\s$-enriched categories with the cartesian product. We write $\mathrm{Fun}_{L\s}(-,-)$ for enriched functor categories between $L\s$-enriched categories.

By a Bousfield localization $L:\s \rightarrow L\s$, we always mean a symmetric monoidal, stable, reflective localization. Every enriched category we consider is $L\s$-enriched in the following way, up to passing to a full enriched subcategory: start with a stable category, extract the canonical spectral enrichment, and observe the morphism spectra are $L$-local. By pushing forward the $\s$-enriched category along the unit $\s \rightarrow L\s$, we obtain an $L\s$-enriched category with the same morphism spectra as the stable category we started with. We denote full subcategories by $\langle c_i \rangle_{i \in I}$ where $\{c_i\}_{i \in I}$ is a set of objects in the category.

If $C$ is an enriched category, we will write $C(-,-)$ for the enriched morphism objects. In particular, in $L\s$ we write $L\s(-,-)$ for the internal $L$-local morphism spectra. Note that although $L\s(-,-)$ has the potential to be mistaken for $L(\s(-,-))$, this is no issue since the spectrum of morphisms into an $L$-local spectra is always $L$-local. We write $\nat(-,-)$ for the natural transformation spectrum (when it is exists)  between functors with values in a stable category.

We only consider operads in $(L\s,\wedge)$, and we abbreviate this category as $\mathrm{Operad}(L\s)$. For the most part, we restrict our attention to reduced operads $O$, i.e. those which satisfy $O(0)\simeq \ast,O(1)=L\bb{S}$. (co)Algebras over such operads are automatically non(co)unital.  Associated to an operad $O$ are many categories of relevant algebraic objects, for instance, there are categories of right modules and right comodules over an operad. Both categories have free and cofree objects. In order to avoid confusion, we emphasize in which category the relevant construction lands by writing its abbreviation in the upper right corner. We will emphasize the relevant operad by writing it in the lower right corner, e.g. if $O$ is a reduced operad, $\mathrm{Triv}_O^\mathrm{RComod}(-)$ denotes the functor which takes in a symmetric sequence and outputs the associated right comodule with a trivial coaction by $O$. We only treat operads in $L\s$, and we will only consider algebras, right modules, etc. defined inside $L\s$. Because of this we drop the category from our notation, e.g. we abbreviate $\mathrm{Alg}_O(L\s,\wedge)$ by $\mathrm{Alg}_O$. If, for instance, we wanted to consider the commutative algebra objects of $(L\s,\wedge)$ we would consider $\mathrm{Alg}_{L(\mathrm{com})}$ rather than the equivalent $\mathrm{Alg}_\mathrm{com}(L\s,\wedge)$. Although we do not use the cooperad structure on $B(1,O,1)$, we occasionally refer to the underlying symmetric sequence. We abbreviate this $B(O)$. 

In general, by ``truncation'' we mean an operation which on the underlying symmetric sequence of an object $R$ is defined by $R^{\leq n}(m)=R(m)$ if $m \leq n$ and $R^{\leq n}(m)=\ast$ otherwise. When such operations exist as endofunctors of the category, they give rise to a complete ``truncation filtration'' via maps $R \rightarrow R^{\leq n}$. We often consider full subcategories on objects which are $n$-truncated for a fixed $n$ or for an arbitrary $n$ which we call bounded. We write these with either $\leq n$ or $<\infty$ in the upper right corner, e.g. the bounded divided power right $O$-modules are $\mathrm{RMod}^\mathrm{<\infty,dp}_O$

\section{Right (co)modules over an operad} \label{section: modules and comodules}
This section relays some of the preliminary information on right (co)modules and Koszul duality necessary for the rest of the paper. Section \ref{section: right module stuff} is relevant to Section \ref{section: right modules in goodwillie calculus} about the role of right modules in Goodwillie calculus. Much of this background material is already known, though we also make reference to work in progress \cite{envBMT} which collects this information in the $\infty$-categorical setting; see also \cite{brantnerHeutsUniversal}. Section \ref{section: right module primitive stuff} is relevant to Section \ref{section: classifications} about the role of divided power right modules in Goodwillie calculus. Here we extend the bar-cobar duality introduced in \cite{AroneChingManifolds} to include divided power structures on the indecomposables of right $O$-comodules.

\subsection{Right modules over an operad} \label{section: right module stuff}
We consider operads in $L\s$, a Bousfield localization of $\s$. Spectral operads are defined in terms of the category of symmetric sequences of spectra \[\mathrm{SymSeq}(L\s):= \mathrm{Fun}(\bigsqcup_{n \in \mathbb{N}} B\Sigma_n, L\s ).\] Its composition product makes it into a monoidal category $(\mathrm{SymSeq}(\s),\circ)$, and algebras for the composition product form the category $\mathrm{Operad}(L\s)$.

More precisely, for $\s$ we use the definition studied in \cite{blans2024chainrulegoodwilliecalculus,brantnerthesis}. The Bousfield localization $L$ induces a monoidal localization of $(\mathrm{SymSeq}(\s),\circ)$  which we define to be $(\mathrm{SymSeq}(L\s),\circ))$. We denote the unit of $(\mathrm{SymSeq}(L\s),\circ)$ by $1$. It is given by the symmetric sequence which is contractible in every degree except the first where it is equal to $L\bb{S}$. 

Suppose $O$ is an operad in $L\s$, then we define the category $\mathrm{RMod}_O$ as the category of right module objects of $O$, i.e. symmetric sequences with a right action of $O$. This is a stable category, and there is an adjunction associated to the unit map $1 \rightarrow O$
\[
\adjunction{\mathrm{Free}^\mathrm{RMod}_O}{\mathrm{SymSeq}(L\s)}{\mathrm{RMod}_O}{\mathrm{Forget}_O}.
\]
 The right adjoint is restriction of right $1$-modules, i.e. symmetric sequences, along the unit. The left adjoint is induction along the unit.
If $O$ is further supplied with an augmentation $O \rightarrow 1$, there is a second adjunction
\[
\adjunction{\mathrm{Indecom}_O}{\mathrm{RMod}_O}{\mathrm{SymSeq}(L\s)}{\mathrm{Triv}^\mathrm{RMod}_O}.
\]
 The right adjoint is given by restricting a right $O$-module to a symmetric sequence, and the left adjoint is given by inducing along the augmentation, otherwise known as forming the indecomposables. The indecomposables are a homotopically coherent way of quotienting out the image of the operad composites, modulo what arises from composition with the unit $L\bb{S} \rightarrow O(1)$.
 
Koszul duality is a device which interchanges these two adjunctions. In order to define the Koszul dual of an augmented operad, we recall the existence of the Day convolution of right modules \cite{envBMT,brantnerHeutsUniversal} which has an underlying formula
\[(R \circledast S) (k) \simeq \bigvee_{i+j=k} \mathrm{ind}^{\Sigma_k}_{\Sigma_i \times \Sigma_j}R(i) \wedge S(j).\]
If $C$ is a stable symmetric monoidal $\infty$-category with $L$-local mapping spectra, the coendomorphism operad $\mathrm{CoEnd}(c)$ of $c \in (C,\otimes)$ is the endomorphism object of $c \in C$ with respect to the enrichment of $C$ in $(\mathrm{SymSeq}(L\s),\circ)$ \cite{envBMT} informally given by
\[C^{\Sigma}(c,d):= \{C(c,d^{\otimes n})\}_{n \in \mathbb{N}}.\]

\begin{definition}
    Let $O \rightarrow 1$ be an augmented operad in $L\s$. The Koszul dual of $O$ is the coendomorphism operad of $1\in(\mathrm{RMod}_O,\circledast)$
    \[K(O):= \mathrm{CoEnd}(1).\]
\end{definition}

It is an easy calculation that \[1^{\circledast n}\simeq \mathrm{Triv}_O(L \Sigma^\infty_+ \Sigma_n),\] and so understanding $K(O)$ is largely reduced to understanding the mapping spectra between trivial right $O$-modules, which can in turn be related to computing the indecomposables of trivial right $O$-modules. When $O$ is reduced, i.e. $O(0)=\ast,O(1)=L\bb{S}$, there is a unique augmentation of $O$. Koszul duality behaves rather well for reduced operads and their right modules. Given $X \in L\s$ we write $X^\vee$ for the $L$-local Spanier-Whitehead dual. We say a spectrum $X \in L\s$ is \textit{dualizable} if for all $Y \in L\s$ the natural map 
\[X^\vee \wedge Y \rightarrow L\s(X,Y)\]
is an equivalence.

The following is one of the main results of \cite{envBMT}.

\begin{prop}\label{prop: koszul duality right modules}
   For a reduced operad $O$ there is a symmetric monoidal adjunction
  \[ \adjunction{K}{(\mathsf{RMod}_{O},\circledast)}{(\mathsf{RMod}_{K(O)}^{\mathsf{op}},\circledast)}{K^{-1}}\]
  \[R \mapsto \mathrm{RMod}^\Sigma_O(R, 1)\]
  \[\mathrm{RMod}^\Sigma_{K(O)}(S, 1) \mapsfrom S\]
where either adjoint lifts $\mathrm{Indecom}_{(-)}(-)^\vee$.
   If $O$ is reduced and levelwise dualizable, this adjunction restricts to an equivalence between the subcategories of levelwise dualizable right modules, and there are equivalences of such levelwise dualizable operads and right modules, respectively,
   \[O \simeq K(K(O)),\]
   \[R \simeq K(K(R)).\]
\end{prop}

The quality of an operad being reduced is useful because it allows one to truncate right modules. More precisely, if $\mathrm{RMod}_O^{\leq n}$ denotes the full subcategory of $\mathrm{RMod}_O$ on right modules $R$ for which $R(m)\simeq \ast$ for $m>n$, then there is an adjunction
  \[ \adjunction{(-)^{\leq n}}{\mathsf{RMod}_{O}}{\mathsf{RMod}_{O}^{\leq n}}{\mathrm{inclusion}}\]
where $R^{\leq n}(m)=R(m)$ if $m \leq n$ and $R^{\leq n}(m)=\ast$ otherwise. We will use such notation throughout the paper to denote operations performed by truncation of the underlying symmetric sequence.

If $O$ is an operad in $L\s$, the \textit{envelope} $(\mathrm{Env}(O),\otimes)$ is an $(L\s,\wedge)$-enriched symmetric monoidal category associated to $O$ which completely encodes the operad $O$ \cite{envBMT}. The category $\mathrm{Env}(O)$ satisfies \[\mathrm{Env}(O)(n,1) = O(n),\] and the property that for objects $x_1,\dots, x_n$ and $y_1,\dots,y_m$ the natural map obtained using the symmetric monoidal structure 
\[\bigvee_{ \phi \in \mathrm{FinSet}(n,m)} \bigwedge_{i=1}^m \mathrm{Env}(O)(\otimes_{j \in \phi^{-1}(i)} x_j,y_i) \rightarrow \mathrm{Env}(O)(\otimes_{j=1}^n x_j, \otimes_{i=1}^m y_i)\]
is an equivalence.

The category $(\mathrm{Env}(O),\otimes)$ can be constructed \cite{envBMT} as the full symmetric monoidal subcategory
\[ \langle O^{\circledast n}\rangle_{n \in \mathbb{N}} = \langle \mathrm{Free}_O^\mathrm{RMod}(L \Sigma^\infty_+ \Sigma_n) \rangle_{n \in \mathbb{N}} \subset (\mathrm{RMod}_O,\circledast).\]
Dually, one has that $\mathrm{Env}(K(O))^\mathrm{op}$ is equivalent \cite{envBMT} to the full symmetric monoidal subcategory
\[ \langle 1^{\circledast n}\rangle_{n \in \mathbb{N}} = \langle \mathrm{Triv}_O^\mathrm{RMod}(L \Sigma^\infty_+ \Sigma_n) \rangle_{n \in \mathbb{N}} \subset (\mathrm{RMod}_O,\circledast).\]

\begin{definition}
    The $\mathrm{Cofree}^\mathrm{RMod}_O$ functor is the right adjoint of the $(\mathrm{Forget}_O,\mathrm{Cofree}^\mathrm{RMod}_O)$-adjunction
    \[
\adjunction{\mathrm{Forget}_O}{\mathrm{RMod}_O}{\mathrm{SymSeq}(L\s)}{\mathrm{Cofree}^\mathrm{RMod}_O}.
\]
\end{definition}

\begin{prop}\label{prop: cofree right module}
    If $S \in \mathrm{SymSeq}(L\s)$ and $O$ is an operad in $L\s$, there is an equivalence of symmetric sequences
    \[\mathrm{Cofree}^\mathrm{RMod}_O(S)(r) \simeq \prod_n \left[ \prod_{n \twoheadrightarrow r} L\s(O(n_1) \wedge \dots \wedge O(n_r), S(n)) \right]^{h\Sigma_n}.\]
\end{prop}
\begin{proof}
    This description is dual to the formula for the free right $O$-module $S \circ O$ and can be obtained by observing
    \[\mathrm{Cofree}^\mathrm{RMod}_O(S)(r)\simeq \mathrm{RMod}_O(L \Sigma^\infty_+ \Sigma_r \circ O, \mathrm{Cofree}^\mathrm{RMod}_O(S))\]
    and using freeness to compute the righthand side.
\end{proof}

\begin{definition}
     The right $O$-module comonad $\mathrm{RM}_{O}$ is the comonad on $\mathrm{SymSeq}(L\s)$ associated to the $(\mathrm{Forget}_O,\mathrm{Cofree}^\mathrm{RMod}_O)$-adjunction.
\end{definition}

It was observed in \cite[Definition 6.9]{ACClassification} that the $(\mathrm{Forget}_O,\mathrm{Cofree}^\mathrm{RMod}_O)$-adjunction was comonadic.  As such, a right module structures on $S \in \mathrm{SymSeq}(L\s)$ is equivalent to the data of coherent structure maps
\[S(r) \rightarrow \prod_n \left[ \prod_{n \twoheadrightarrow r} L\s(O(n_1) \wedge \dots \wedge O(n_r), S(n)) \right]^{h\Sigma_n}.\]

We end this section with a result about the interaction of Koszul duality and Bousfield localization.

\begin{prop}\label{prop: localizing koszul dual}
    Suppose $O \in \mathrm{Operad}(L\s)$ and $L'L\s$ is a Bousfield localization of $L\s$. If both $O$ and $R \in \mathrm{RMod}_O$ are levelwise dualizable, then
    \[L'K(R) \simeq K(L'(R)).\]
\end{prop}
\begin{proof}
It suffices to show the natural map is an equivalence on the underlying symmetric sequences, i.e.
\[L'(\mathrm{Indecom}_O(R)^\vee)\simeq (\mathrm{Indecom}_{L'O}L'R)^\vee.\]
The indecomposables have a bar construction model $B(R,O,1)$ and since $L'$ is both symmetric monoidal and commutes with colimits\footnote{All Bousfield localizations commute with colimits since we always assume colimits of local spectra are computed in the category of local spectra.} 
\[L'(\mathrm{Indecom}_O(R))\simeq (\mathrm{Indecom}_{L'O}L'R).\]
Thus the question reduces to when $L'$-localization of the $L$-local dual of an $L$-local spectrum is equivalent to the $L'$-local dual of the $L'$-localization. This is the case for all dualizable $L$-local spectra. The hypotheses that $O$ is reduced and both $O$ and $R$ are levelwise dualizable imply that the geometric realization $B(R,O,1)(n)$ is a finite geometric realization, and so is a finite colimit of dualizable objects, hence is dualizable.
\end{proof}

\subsection{Divided power right modules over an operad} \label{section: right module primitive stuff}
If $G$ is a finite group, the norm \cite{klein_2001} is a certain natural map associated to a Borel $G$-spectrum $X \in L\s^{BG}$
\[X_{hG} \rightarrow X^{hG}.\]
The norm witnesses the homotopy orbits as the finitary approximation to homotopy fixed points as functors $L\s^{BG} \rightarrow L\s$ \cite{klein_2002}. In operad theory, the nomenclature ``divided powers'' is used to denote the extra data of a lift or an extension, depending on if the setting is algebraic or coalgebraic, along a $\Sigma_n$-norm.  This terminology is meant to invoke the emotions associated to the classical definition of a divided power polynomial ring
\[k[x,\frac{x^2}{2},\frac{x^3}{3!},\dots, \frac{x^n}{n!},\dots].\]
This name is justified since $1$-categorically the norm from $\Sigma_n$-orbits to $\Sigma_n$-fixed points is defined by
\[[x] \mapsto \sum_{g \in \Sigma_n} gx,\]
and to find a section of this requires division by $n!$. In the case the coefficient ring is a characteristic $0$ field, this is always possible, and one of many reasons why working over such fields simplifies life.

It is somewhat subtle to define divided powers for (co)algebraic objects over spectral operads because it is not obvious how to encode the necessary higher coherences. We will encode such coherences through a certain type of bar-cobar duality which is based on the indecomposables of operadic comodules. Bar-cobar dualities of this nature were originally studied in \cite{crosseffectsclassification,AroneChingManifolds}, but seem to be otherwise neglected in the literature. The connection between this bar-cobar duality and the one described in the previous section seems rather subtle, though can be understood in special cases through Goodwillie calculus, see Proposition \ref{prop: derivatives of polynomial} and also Lemma \ref{lem: representables in alg}.

\begin{definition}
        If $O$ is an operad in $L\s$, the category of right $O$-comodules is \[\mathrm{RComod}_O :=\mathrm{Fun}_{L\s}(\mathrm{Env}(O),L\s).\]
    \end{definition}

This definition is motivated by a formally dual characterization of right $O$-modules
\[\mathrm{RMod}_O \simeq\mathrm{Fun}_{L\s}(\mathrm{Env}(O)^\mathrm{op},L\s),\]
 which $1$-categorically \cite[Appendix A]{ACOperads} \cite[\S10.1]{MayZhangZou} is well known, and $\infty$-categorically will appear in \cite{envBMT}. If $O$ is augmented, then taking envelopes yields a symmetric monoidal functor
\[\mathrm{Env}(O) \rightarrow \bigsqcup_{n \in \mathbb{N}} L\Sigma^\infty_+ B{\Sigma_n} = \mathrm{Env}(1).\]

\begin{definition}
    If $O \rightarrow 1$ is an augmented operad, the $(\mathrm{Indecom}_O,\mathrm{Triv}^\mathrm{RComod}_O)$-adjunction
\[\adjunction{\mathrm{Indecom}_O}{\mathrm{RComod}_O}{\mathrm{SymSeq}(L\s)}{\mathrm{Triv}^\mathrm{RComod}_O}\]
    is given by induction and restriction along the augmentation of $O$.
\end{definition}

\begin{remark}
The indecomposables of a right $O$-comodule are a generalization of the indecomposables, or André-Quillen homology, of an $O$-algebra in the sense that there is a forgetful functor 
\[\mathrm{Alg}_O \rightarrow \mathrm{RComod}_O\]
\[A \mapsto \{A^{\wedge n}\}_{n\in \mathbb{N}}.\]
such that the diagram commutes:
\[\begin{tikzcd}
	{\mathrm{Alg}_O} & {L\s} \\
	{\mathrm{RComod}_O} & {\mathrm{SymSeq}(L\s)}
	\arrow["{\mathrm{Indecom}_O}"', from=1-1, to=1-2]
	\arrow["{(-)^{\wedge \ast}}", from=1-1, to=2-1]
	\arrow["{(-)^{\wedge \ast}}"', from=1-2, to=2-2]
	\arrow["{\mathrm{Indecom}_O}", from=2-1, to=2-2]
\end{tikzcd}\]

\end{remark}

For the remainder of this section we fix an operad $O$ in $L\s$ which is reduced and levelwise dualizable. The following definition of a divided power right module appears to have been first considered by Ching \cite[Remark 9.4]{BehrensRezk}.

\begin{definition}\label{definition: dp rmod}
    The divided power right $K(O)$-module comonad $\mathrm{RM}^\mathrm{dp}_{K(O)}$ is the comonad on $\mathrm{SymSeq}(L\s)$ associated to the $(\mathrm{Indecom}_O,\mathrm{Triv}^\mathrm{RComod}_O)$-adjunction. We denote its coalgebra category
    \[\mathrm{RMod}^\mathrm{dp}_{K(O)}:=\mathrm{Coalg}_{\mathrm{RM}^\mathrm{dp}_{K(O)}}. \]
\end{definition}
We denote the associated functor lifting $\mathrm{Indecom}_O$ to divided power right $K(O)$-modules by 
\[B:\mathrm{RComod}_O \rightarrow \mathrm{RMod}^\mathrm{dp}_{K(O)}.\]

In general, we write \[\mathrm{RM}^\mathrm{dp}_{O}:=\mathrm{RM}^\mathrm{dp}_{K(K(O))},\] and this is justified since for reduced, levelwise dualizable operads \[O \simeq K(K(O)).\]

The category $\mathrm{RComod}_O$ is defined as enriched presheaves on a symmetric monoidal enriched category, and so it inherits a Day convolution \cite[8.3.3]{HinichColimits}. Unlike the Day convolution of right modules, this depends on the right $O$-comodule structure rather than just the underlying symmetric sequence.

\begin{remark}
For formal reasons, $\mathrm{Indecom}_O$ is symmetric monoidal with respect to day convolution, and so its adjoint $\mathrm{Triv}_O^\mathrm{RComod}$  is lax symmetric monoidal. As a consequence of \cite[Theorem 1.10, Theorem 1.11]{heine2025monadicitytheoremhigheralgebraic}, $(\mathrm{Coalg}_{\mathrm{RM}^\mathrm{dp}_{K(O)}},\circledast)$ is a symmetric monoidal category
and the forgetful functor to symmetric sequences is symmetric monoidal. In other words it is computed via the Day convolution of the underlying symmetric sequences. We won't make substantial use of this symmetric monoidal structure, though it does appear in a later characterization of polynomial functors.
\end{remark}

It is expected that the $(\mathrm{Indecom}_O,\mathrm{Triv}^{\mathrm{RComod}}_O)$-adjunction is comonadic on bounded right $O$-comodules. This is proved by Arone-Ching in the case $O$ is $\Sigma$-finite \cite[Proposition 2.7]{AroneChingManifolds}, in which case divided power right modules are equivalent to right modules, though this fact is not yet obvious from our definition.

We now justify the use of the term ``divided powers''. As in \cite[Equation 2.6]{AroneChingManifolds}, one can supply the sequence of right $K(O)$-modules \[\mathrm{Tr}:=\{\mathrm{Triv}^{\mathrm{RMod}}_{K(O)}(L\Sigma^\infty_+\Sigma_n)\}_{n \in \mathbb{N}}\]
with the structure of a presheaf \footnote{This object is equivalently a right $O$-module in right $K(O)$-modules.}
\[\mathrm{Tr}:\mathrm{Env}(O)^\mathrm{op} \rightarrow \mathrm{RMod}_{K(O)}\]
\[\mathrm{Tr}(n) := K( \mathrm{Free}_O^{\mathrm{RMod}}(L\Sigma^\infty_+\Sigma_n) ).\]

If $C$ is an $O$-comodule, the symmetric sequence $\mathrm{Indecom}_O(C)$ may be explicitly computed as the coend 
\[\mathrm{Indecom}_O(C)\simeq \int^{n \in \mathrm{Env}(O)} \mathrm{Tr}(n) \wedge C(n),\]
and the presheaf structure on $T$ endows $\mathrm{Indecom}_O(C)$ with the structure of a right $K(O)$-module compatible with its coalgebra structure. We denote this functor\footnote{We have intentionally given this the same name as $B: \mathrm{RComod}_O \rightarrow \mathrm{RMod}^\mathrm{dp}_{K(O)}$ and distinguish them by context.}
\[B: \mathrm{RComod}_O \rightarrow \mathrm{RMod}_{K(O)}.\]

Classically, the comonad induced by an adjuction is constructed by witnessing it as a \textit{coendomorphism object} \cite[Section 4]{HA}, and so $\mathrm{RM}^\mathrm{dp}_{K(O)}$ has a universal property with respect to comonads coacting on $\mathrm{Indecom}_O$. Since we have produced a right $K(O)$-module structure on $\mathrm{Indecom}_O$, the universal property produces a map of comonads
\[\mathrm{RM}^\mathrm{dp}_{K(O)} \rightarrow \mathrm{RM}_{K(O)}.\]
As a consequence, a divided power right module has an underlying right module. The following appears as \cite[Proposition 3.67]{crosseffectsclassification} when $O =\Sigma^\infty_+ E_d$, and the given argument proves more generally:
\begin{prop}\label{prop: cofree divided power right modules}
   If $S$ is a bounded symmetric sequence, then
   \[\mathrm{Cofree}^{\mathrm{RMod}^\mathrm{dp}}_{O}(S)(r) \simeq \prod_n \left[ \prod_{n \twoheadrightarrow r} L\s(O(n_1) \wedge \dots \wedge O(n_r), S(n)) \right]_{h\Sigma_n},\]
   and the map of comonads
   \[\mathrm{RM}^\mathrm{dp}_{K(O)} \rightarrow \mathrm{RM}_{K(O)}\]
   is given on underlying symmetric sequences by the product of the norm maps.
\end{prop}

 In other words, a divided power right $O$-module structure is a lift, together with coherences, of a right $O$-module structure against the norm:
\begin{center}
\[\begin{tikzcd}
	& { \prod_n \left[ \prod_{n \twoheadrightarrow r} L\s(O(n_1) \wedge \dots \wedge O(n_r), R(n)) \right]_{h\Sigma_n}} \\
	{R(r)} & { \prod_n \left[ \prod_{n \twoheadrightarrow r} L\s(O(n_1) \wedge \dots \wedge O(n_r), R(n)) \right]^{h\Sigma_n}}
	\arrow[from=1-2, to=2-2]
	\arrow[dashed, from=2-1, to=1-2]
	\arrow[from=2-1, to=2-2]
\end{tikzcd}\]
\end{center}

In Arone--Ching's original work on the classification of polynomial functors \cite{ACClassification}, a different method of encoding such coherences was produced for $O=\mathrm{lie}$ which used Goodwillie calculus. It is a formal consequence of Theorem \ref{thm: diff semiadditive classification} that the Goodwillie calculus approach to divided power right $\mathrm{lie}$-modules agrees with the bar-cobar duality approach.
 
In any coalgebra category, one may consider the comonadic resolution of a coalgebra. Mapping into this yields a resolution of the mapping objects in the category.
Denote the $n$-fold composition of a functor $F$ by $F^n$. If $R,S$ are right $K(O)$-modules, this yields totalization formulas for mapping spectra
\[\mathrm{RMod}_O(R,S) \simeq \mathrm{Tot}(\mathrm{SymSeq}(R,\mathrm{RM}_O^\bullet(S))), \]
and if they are equipped with divided powers, then
\[\mathrm{RMod}^\mathrm{dp}_O(R,S) \simeq \mathrm{Tot}(\mathrm{SymSeq}(R,(\mathrm{RM}_O^\mathrm{dp})^\bullet(S))). \]

Generally, mapping objects in the category of right modules are much easier to understand than in the category of divided power right modules. From the perspective of these cosimplicial resolutions, this is due to the fact that $\mathrm{SymSeq}(-,-)$ is computed via $\Sigma_n$-homotopy fixed points, and it is more difficult to compute homotopy fixed points of homotopy orbits than it is to computed iterated homotopy fixed points. In particular, simply filtering the codomain by the truncation filtration $S^{\leq \ast}$ leads to a convergent filtration of $\mathrm{RMod}_O(R,S)$:
\begin{center}
\[\begin{tikzcd}
	& \vdots \\
	{L\s^{B\Sigma_{n+1}}(\mathrm{Indecom}_O(R)(n+1),S(n+1))} & {\mathrm{RMod}_O(R,S^{\leq n+ 1})} \\
	{L\s^{B\Sigma_n}(\mathrm{Indecom}_O(R)(n),S(n))} & {\mathrm{RMod}_O(R,S^{\leq n})} \\
	& \vdots
	\arrow[from=1-2, to=2-2]
	\arrow[from=2-1, to=2-2]
	\arrow[from=2-2, to=3-2]
	\arrow[from=3-1, to=3-2]
	\arrow[from=3-2, to=4-2]
\end{tikzcd}\]
\end{center}
Note that when restricted to levelwise dualizable operads and right modules, Koszul duality preserves the associated graded of this filtration
\[K(R)(n) \wedge^{h\Sigma_n} S(n) \simeq K(K(S))(n) \wedge^{h\Sigma_n} K(R)(n).\]

It is often useful to quantify the failure of the norm to be an equivalence through the \textit{Tate construction}
\[X^{tG}:=\mathrm{cofiber}(X_{hG} \rightarrow X^{hG}).\]
When $X^{tG}\simeq \ast$, we say $X$ has Tate vanishing. Tate vanishing can occur for a variety of reasons, not all of which interact well with the smash product. 
\begin{definition}
     We say that $X \in L\s^{BG}$ satisfies ideal Tate vanishing if for all $Y \in L\s^{BG}$, $X \wedge Y$ satisfies Tate vanishing.
\end{definition}

 The most common examples of Borel $G$-spectra with ideal Tate vanishing are the finite $G$-CW complexes. Another scenario in which ideal Tate vanishing occurs is when \textit{all} Tate constructions in $L\s^{BG}$ vanish. If this is the case for all finite groups $G$, the category is called $1$-semiadditive.

\begin{definition}
    Given an operad $O$ and a bounded symmetric sequence $S$, we say that $S$ satisfies $O$-ideal Tate vanishing if the map
    \[\mathrm{RM}_O^\mathrm{dp}(S) \rightarrow \mathrm{RM}_O(S)\]
    is an equivalence.
\end{definition}

\begin{lem}\label{lem: ideal implies o ideal}
    Suppose $O$ is an operad and $S$ is a bounded symmetric sequence and both are levelwise dualizable. If $S^\vee$ satisfies ideal Tate vanishing levelwise, then $S$ satisfies $O$-ideal Tate vanishing.
\end{lem}
\begin{proof}
    If $X,Y$ are both dualizable Borel $G$-spectra, then 
    \[(X^\vee \wedge Y^\vee)^{tG} \simeq \ast \iff (X \wedge Y)^{tG} \simeq \ast.\]
    This fact combined with the explicit description of the comonads of Proposition \ref{prop: cofree divided power right modules} implies the result since $O$ and $S$ are levelwise dualizable.
\end{proof}

\begin{prop}\label{prop: divided power mapping spectrum codomain tate vanish}
    Suppose $R,S$ are divided power right $O$-modules and that $S$ is bounded and has $O$-ideal Tate vanishing, then the natural map is an equivalence
    \[\mathrm{RMod}^\mathrm{dp}_O(R,S) \xrightarrow{\simeq} \mathrm{RMod}(R,S).\]
\end{prop}
\begin{proof}
    This is a direct consequence of the definition of $O$-ideal Tate vanishing and the cosimplicial formula for either mapping spectrum.
\end{proof}

There is a dual scenario in which we enforce Tate vanishing hypotheses on the domain. In that circumstance, the mapping spectrum in the category of divided power right modules behaves like an excisive approximation to the mapping spectrum of the underlying right modules.

\begin{prop}\label{prop: filtration when domain tate vanish}
    Suppose $R,S\in \mathrm{RMod}_O^\mathrm{dp}$, $R$ is levelwise dualizable, and $R^\vee$ has ideal Tate vanishing. The truncation filtration of $S$ induces a convergent filtration of $\mathrm{RMod}^\mathrm{dp}_O(R,S)$:
\[\begin{tikzcd}
	& \vdots \\
	{\mathrm{Indecom}_O(R)(n+1)^{\vee}\wedge_{h\Sigma_{n+1}}S(n+1)} & {\mathrm{RMod}^\mathrm{dp}_O(R,S^{\leq n+ 1})} \\
	{\mathrm{Indecom}_O(R)(n)^\vee\wedge_{h\Sigma_n}S(n)} & {\mathrm{RMod}^\mathrm{dp}_O(R,S^{\leq n})} \\
	& \vdots
	\arrow[from=1-2, to=2-2]
	\arrow[from=2-1, to=2-2]
	\arrow[from=2-2, to=3-2]
	\arrow[from=3-1, to=3-2]
	\arrow[from=3-2, to=4-2]
\end{tikzcd}\]
\end{prop}

\begin{proof}
    The convergence follows from the convergence of the truncation filtrations in $\mathrm{RMod}_O^\mathrm{dp}$. It then suffices to show the equivalence
    \[\mathrm{RMod}^\mathrm{dp}_O(R,S(n)) \simeq \mathrm{Indecom}_O(R)(n)^\vee \wedge_{h\Sigma_n} S(n) \]
    where $S(n) = \mathrm{fiber} (S^{\leq n} \rightarrow S^{\leq n-1})$ is concentrated in degree $n$. This is true when $S(n)$ is a free Borel $\Sigma_n$-spectrum by Proposition \ref{prop: divided power mapping spectrum codomain tate vanish} and Tate vanishing. Thus, it suffices to show that $\mathrm{RMod}^\mathrm{dp}_O(R,S(n))$ preserves colimits in $S(n)$. By the reducedness of $O$, the totalization computing $\mathrm{RMod}^\mathrm{dp}_O(R,S(n))$ is a finite totalization; the formula for $\mathrm{RM}_O^\mathrm{dp}(S(n))$ is written using finite products and homotopy orbits; the mapping spectra of symmetric sequences appearing in the totalization are computed with homotopy orbits by the Tate vanishing hypotheses on $R^\vee$, and so the resulting totalization commutes with colimits in the variable $S(n)$, finishing the proof.
\end{proof}

\begin{remark}
    We expect that in the above scenario the mapping spectrum can be written as a comonadic cobar construction $\Omega(R^\vee,\mathrm{RM}_O^\mathrm{dp},S)$, similar to the divided power coalgebraic cobar constructions which appear in \cite{generalizedPoincare}. 
\end{remark}
\begin{remark}

The conclusion of Proposition \ref{prop: filtration when domain tate vanish} can be used to show that for a general $O$ there are not typically functors \[\mathrm{Free}^{\mathrm{RMod}_O^\mathrm{dp}}:\mathrm{SymSeq}(L\s) \rightarrow \mathrm{RMod}_O^\mathrm{dp},\] 
 \[\mathrm{Indecom}_O: \mathrm{RMod}_O^\mathrm{dp} \rightarrow \mathrm{SymSeq}(L\s)\]
defined, respectively, as left adjoints to $\mathrm{Forget}_O,\mathrm{Triv}_O^{\mathrm{RMod}^\mathrm{dp}}$.
This is because limits in $\mathrm{RMod}_O^\mathrm{dp}$, even those of divided power right modules concentrated in a single degree, are not necessarily computed levelwise.
\end{remark}

Earlier in the section, we recalled Arone--Ching's construction of the left adjoint in an adjunction which is in some sense an extension of bar-cobar duality between $O$-algebras and $K(O)$-coalgebras:
\[\adjunction{B}{\mathrm{RComod}_O}{\mathrm{RMod}_{K(O)}}{\Omega}.\]
For a general operad $O$, this adjunction is far from an equivalence. However, \cite[Proposition 2.7]{AroneChingManifolds} demonstrates that when $O$ is $\Sigma$-finite, or has some more general Tate vanishing hypotheses, this restricts to an equivalence on bounded right (co)modules. The following is readily extracted from the proof of \cite[Proposition 2.7]{AroneChingManifolds} when one observes that the $\Sigma$-finiteness of $K(O)$ used to prove \cite[
Lemma 3.71]{crosseffectsclassification} can be replaced by the assumption that
\[\mathrm{RM}_O^\mathrm{dp}(S) \rightarrow \mathrm{RM}_O(S) \]
is an equivalence, for a fixed $S$.

\begin{lem}\label{lem: counit of bar cobar is equiv on sigma finite}
    If $S$ is a bounded right $K(O)$-module which satisfies $K(O)$-ideal Tate vanishing, then the counit of the bar-cobar duality adjunction is an equivalence
    \[B \Omega S \xrightarrow{\simeq} S.\]
\end{lem}

\section{Right modules in Goodwillie calculus}\label{section: right modules in goodwillie calculus}

\subsection{Right module structures on Goodwillie derivatives}\label{section: right module construction}
We fix a compactly generated, differentiable category $C$ together with an identification $\s(C) \simeq L\s$ to some Bousfield localization of $\s$. We denote the stabilization adjunction
\[\adjunction{\Sigma^\infty_C}{C}{L\s}{\Omega^\infty_C}.\]
Note that because $C$ is compactly generated $L\s= \s(C)$ must be as well.

We are motivated by three main examples: 

\begin{enumerate}
    \item  The category of pointed spaces
\[\adjunction{\Sigma^\infty}{\T}{\s}{\Omega^\infty},\]
\item The category of $v_h$-periodic pointed spaces $\T_{v_h}$, the object of study in unstable chromatic homotopy theory \cite{telescopichomotopybous},
\[\adjunction{\Sigma_{T(h)}^\infty}{\T_{v_h}}{\s_{T_{h}}}{\Omega^\infty_{T(h)}};\]
This category is compactly generated \cite[Proposition 3.14]{heutsannals}, and $\s(\T_{v_h})$ is identified with the localization of $\s$ at the telescopic spectrum $T(h)$ \cite[Remark 3.20]{heutsannals}.
\item The category of algebras over a reduced operad $O \in \mathrm{Operad}(L\s)$
\[\adjunction{\mathrm{Indecom}_O}{\mathrm{Alg}_O}{L\s}{\mathrm{Triv}^\mathrm{Alg}_O};\]
Here the stabilization functor is given by taking $O$-algebra indecomposables \cite{taqstabilization,Pereira_2013}. Its compact generation is witnessed by the free algebras on the compact generators of $L\s$. 
\end{enumerate}
One can produce additional interesting examples of $C$ such that $\s(C) \simeq L\s$ by applying Heuts's categorical Goodwillie approximations \cite{HeutsGoodwillieApprox} to any of the above examples.

Examples of compactly generated $L\s$ to keep in mind are $\s,(\s)_\mathbb{Q},\s_{T(h)},$ and $\s_{K(h)}$, where $K(h)$ is Morava $K$-theory \cite[Theorem 7.1]{Hovey1999MoravaKA}. If $L\s$ were not compactly generated, then the categories of algebras over a reduced $O \in \mathrm{Operad}(L\s)$ would not be compactly generated. Some notable examples of Bousfield localizations with no nontrivial compact objects at all are the localizations at $H\mathbb{F}_p$ and $\bigvee_{h \in \mathbb{N}} K(h)$ \cite[Corollary B.13.]{Hovey1999MoravaKA}.

Goodwillie calculus \cite{GoodCalcIII} was introduced to study categories of finitary  functors $\mathrm{Fun}^\mathrm{\omega}(C,D)$, i.e. those functors which preserves filtered colimits. Our requirement that $C$ is compactly generated is motivated entirely by the fact that restriction to the compact objects $C^\omega \subset C$ induces an equivalence
\[\mathrm{Fun}^\mathrm{\omega}(C,D) \simeq \mathrm{Fun}(C^\mathrm{\omega},D).\]
This equivalence allows for us to easily identify the right adjoints of colimit preserving functors out of $\mathrm{Fun}^\mathrm{\omega}(C,L\s)$ in terms of enriched natural transformations as a consequence of the Yoneda lemma. 

Goodwillie calculus attempts to approximate such functors by increasingly excisive\footnote{A functor is $n$-excisive if it takes strongly cocartesian $(n+1)$-cubes to cartesian $(n+1)$-cubes.} functors. If $\s(C)\simeq \s(D) \simeq L\s$, the Goodwillie tower of a finitary functor $G:C \rightarrow D$ takes the form

\adjustbox{scale=.9,center}{
\begin{tikzcd}
	&& G \\
	{P_\infty (G)} & \cdots & {P_n(G)} & \cdots & {P_1(G)} & {P_0(G)} \\
	&& {\Omega^\infty_D(\partial_nG \wedge_{h\Sigma_n} \Sigma^\infty_C(-)^{\wedge n})} && {\Omega^\infty_D(\partial_1G \wedge \Sigma^\infty_C(-))}
	\arrow[curve={height=6pt}, from=1-3, to=2-1]
	\arrow[from=1-3, to=2-3]
	\arrow[curve={height=-6pt}, from=1-3, to=2-5]
	\arrow[curve={height=-6pt}, from=1-3, to=2-6]
	\arrow[from=2-1, to=2-2]
	\arrow[from=2-2, to=2-3]
	\arrow[from=2-3, to=2-4]
	\arrow[from=2-4, to=2-5]
	\arrow[from=2-5, to=2-6]
	\arrow[from=3-3, to=2-3]
	\arrow[from=3-5, to=2-5]
\end{tikzcd}
}
The Borel $\Sigma_n$-spectra $\partial_n G \in L\s^{B\Sigma_n}$ are called the derivative spectra. The derivatives completely classify the functors which appear as the $n$th fiber of a Goodwillie tower. Such fibers are known as the $n$-homogeneous functors. It is useful to bundle all the derivatives together into a symmetric sequence \[\partial_\ast F :=\{\partial_n F\}_{n \in \mathbb{N}} \in \mathrm{SymSeq}(L\s).\] 

 The observation of Arone-Ching \cite{ACOperads}, proven universally by Blans-Blom \cite{blans2024chainrulegoodwilliecalculus}, is that if $F$ is a reduced functor\footnote{$F$ is reduced if $F(\ast)=\ast$.}, composable with $G$, then the symmetric sequences $\{\partial_n F\}_{n \in \mathbb{N}},\{\partial_n G\}_{n \in \mathbb{N}}$ can be equipped with certain operadic module structures which can be used to reconstruct the symmetric sequence $\partial_\ast (G \circ F)$ via the \textit{chain rule}.

Instead of using functor composition, our approach is loosely based on generalizing the duality between commutative models of rational spaces \cite{sullivanrational} and Lie algebra models of rational spaces \cite{quillenrational}. We associate to $c \in C^\omega$ an operadic coalgebra structure on $\Sigma^\infty_C c$ (the ``Sullivan part'') and through Koszul duality we extract information about the Goodwillie derivatives of the representable functor associated to $c$ (the ``Quillen part''). This approach uses Spanier-Whitehead duality in a fundamental way, and so various finiteness assumptions show up in our treatment.

We say that an object $c \in C$ is compact if $\mathrm{Map}(c,-)$ is finitary, i.e. commutes with filtered colimits. We say that a spectrum $X \in L\s$ is finite if it is built out of finitely many local cells, in particular this holds if $X$ is the localization of a finite complex. We write
\[X^\vee:= L\s(X,L\mathbb{S})\]
for the $L$-local Spanier-Whitehead dual of $X$. For $X,Y \in L\s$ there is a canonical map
\[X^\vee \wedge Y \rightarrow L\s(X,Y),\]
and we call $X$ dualizable if this map is an equivalence for all $Y$. It is important to note that in $L\s$ the following implications hold
\[\mathrm{compact} \implies \mathrm{finite} \implies \mathrm{dualizable}.\]
The first implication holds by a standard retraction argument, and the second follows from the fact that $L\bb{S}$ is dualizable. Furthermore, all of these qualities are closed under retracts.

\begin{definition}
    A category $C$ is differentially dualizable if $\partial_\ast (\Sigma^\infty_C\Omega^\infty_C)$ is levelwise dualizable.
\end{definition}

\begin{ex}
    If $O\in \mathrm{Operad}(L\s)$ is a reduced, levelwise dualizable operad, then $\mathrm{Alg}_O$ is differentially dualizable by a direct computation. This is because there is a splitting (see \cite[pg. 123]{chingHarper})
    \[\mathrm{Indecom}_O \circ \mathrm{Triv}_O^\mathrm{Alg} \simeq \bigvee_{i \in \mathbb{N}} B(O)(i) \wedge_{h\Sigma_i} \mathrm{Id}_{L\s}^{\wedge i} \]
    which implies \[\partial_\ast( \mathrm{Indecom}_O \circ \mathrm{Triv}_O^\mathrm{Alg}) \simeq B(O).\] 
On the other hand, \cite[Theorem 11.3]{Pereira_2013} asserts there is an equivalence of symmetric sequences \[\partial_\ast \mathrm{Id}_{\mathrm{Alg}_O} = O.\]
Hence for $C=\mathrm{Alg}_O$, bar-cobar duality implies $\partial_\ast (\Sigma^\infty_C \Omega^\infty_C)$ is levelwise dualizable, if and only if $\partial_\ast \mathrm{Id}_C$ is levelwise dualizable.  This relationship generalizes to arbitrary $C$ as a consequence of the chain rule \cite[Theorem 4.3.1]{blans2024chainrulegoodwilliecalculus}
\[\partial_\ast (\Sigma^\infty_C \Omega^\infty_C) \simeq B(1,\partial_\ast \mathrm{Id}_C,1).\]
Therefore, $C$ is differentially dualizable exactly when $\partial_\ast \mathrm{Id}_C$ is levelwise dualizable.
\end{ex}

Every stable presentably symmetric monoidal category with local morphism spectra has two canonical enrichments in $\mathrm{SymSeq}(L\s,\circ)$ \cite{envBMT} which are informally given by 
\[\mathrm{Hom}((-)^{\otimes n},(-)),\]
\[\mathrm{Hom}((-),(-)^{\otimes n}).\]
The endomorphism object with respect to the first enrichment define the endomorphism operad $\mathrm{End}(-)$ while the endomorphism object with respect to the second define the coendomorphism operad $\mathrm{CoEnd}(-)$. We primarily make use of the second enrichment, and we denote it $\mathrm{Hom}^{\Sigma}(-,-).$

Via adjunction, the natural transformation spectrum $\nat(F,G) \in L\s$ between pointed functors $F,G:C \rightarrow L\s$ may be computed as the enriched natural transformations between the enriched functors \[{L}\Sigma^\infty F,{L}\Sigma^\infty G:{L}\Sigma^\infty C \rightarrow L\s.\] Here the notation $L\Sigma^\infty (-)$ denotes the functor
    \[{L}\Sigma^\infty: \mathrm{Cat}_{\T} \rightarrow \mathrm{Cat}_{L\s}\]
    which takes a category enriched in pointed spaces and pushes it forward along the unit $\T \rightarrow L\s$. When $c$ is compact, the enriched Yoneda lemma \cite{HinichYoneda} allows us to compute natural transformations out of \[{L}\Sigma^\infty C(c,-):C^\omega \rightarrow L\s.\] We denote this stabilized representable by ${L}\Sigma^\infty R_c$. We note that when $C$ is compact $L\Sigma^\infty R_c: C \rightarrow L\s$ is finitary. 

\begin{lem}\label{lem:dualizable natural trans spec}
    There is an equivalence of functors \[\mathrm{Fun}^\omega(L\s,L\s) \rightarrow \mathrm{SymSeq}(L\s)\]
    \[\partial_n(-)^\vee \xrightarrow{\simeq} \nat(-,\mathrm{Id}_{L\s}^{\wedge n}).\] When restricted to the subcategory of $\mathrm{Fun}^\omega(L\s,L\s)$ consisting of functors with $n$th derivative dualizable, it implies an equivalence
    \[\partial_n(-) \simeq \nat(-,\mathrm{Id}_{L\s}^{\wedge n})^\vee.\]
    
\end{lem}

\begin{proof}
 It suffices to prove first equivalence, as it implies the second under the dualizability hypothesis.  For a general functor, the comparison map is defined by taking the $n$th derivative of such a natural transformation. This is an equivalence when restricted to $n$-homogeneous functors by the classification of homogeneous functors. It then suffices to show that when $k<n$ the natural transformations of a $k$-homogeneous functor into $\mathrm{Id}_{L\s}^{\wedge n}$ is contractible. Since any $\Sigma_{k}$-spectrum is a colimit of free $\Sigma_{k}$-spectra, it suffices to prove contractibility for homogeneous functors with free coefficients.
    
    For such homogeneous functors, the $\Sigma_{k}$-norm is an equivalence, and so they may be expressed with fixed points rather than orbits. Thus, such homogeneous functors are also cohomogeneous functors of degree $k$, in the sense of dual calculus \cite{dualcalc}. As the target is cohomogeneous of a higher degree, the spectrum of natural transformations is contractible by the universal property of dual Goodwillie approximations.
\end{proof}

It is quite tempting to apply Lemma \ref{lem:dualizable natural trans spec} to compute $\partial_\ast L\Sigma^\infty R_c$ for $c \in L\s^\omega$. These functors are finitary and have their natural transformation spectra accessible through the Yoneda lemma. However, in life and in Goodwillie calculus there is no free lunch. To apply the lemma to compute Goodwillie derivatives, one needs to show that the derivatives of these stabilized representables are levelwise dualizable. Such results can often be proven from first principles, but we will instead appeal to a weak version of the chain rule due to Ching. The \textit{weak stable chain rule}  \cite[Theorem 1.15]{chainruleforspec} asserts that given finitary, composable functors of the form
\[L\s \xrightarrow{G} L\s \xrightarrow{F} L\s\]
or of the form
\[\s \xrightarrow{G} L\s \xrightarrow{F} L\s\]
where $G$ is pointed, there is an equivalence of symmetric sequences in $\mathrm{SymSeq}(L\s)$
\[\partial_\ast(F \circ G) \simeq \partial_\ast F \circ \partial_\ast G.\]
The weak stable chain rule can be equivalently formulated as derivatives being monoidal with respect to composition \textit{at the level of homotopy categories}. In contrast, the stable chain rule proved in \cite{blans2024chainrulegoodwilliecalculus}, almost two decades later, asserts that the derivatives assemble into a monoidal functor with respect to composition \textit{before} taking homotopy categories.

Given a spectrum $X$, we have a symmetric sequence \[X^{\wedge \ast}:=\{ X^{\wedge n} \}_{n \in \mathbb{N}_{>0}} \]
where we use the convention that when a symmetric sequence is not specified in a certain arity, in this case $0$, then it is contractible in that degree.

\begin{cor}\label{cor:representables for spectra}
    If $X \in L\s^\omega$, then there is an equivalence
    \[\partial_\ast L\Sigma^\infty R_X \simeq (X^\vee)^{\wedge \ast}.\]
\end{cor}
\begin{proof}
   The fully faithful embedding
   \[\mathrm{Fun}^\omega(L\s,L\s) \subset \mathrm{Fun}^\omega(\s,L\s)\]
   \[F \mapsto F \circ L\]
   preserves derivative spectra, and so we will instead compute the derivatives of its image \[L\Sigma^\infty R_X(L(-)) \simeq L \Sigma^\infty \Omega^\infty\s(X,L(-)).\] If $Z$ is a finite spectrum, there is an equivalence \footnote{Typically, we would write $\s(X,L\mathbb{S})= L\s(X,L\mathbb{S})$ as $X^\vee$, given our tendency to work internally to $L\s$. In this case, it is an object of $\s$, and this notation is meant to emphasize that.}
   \[\s(X,LZ) \simeq \s(X,L\mathbb{S}) \wedge Z,\]
   since they agree on $Z = \bb{S}$ and both sides preserve finite colimits in $Z$.
   Hence, for finite $Z$ there is a decomposition
   \[L \Sigma^\infty \Omega^\infty \s(X,LZ)\simeq L \circ \Sigma^\infty \Omega^\infty \circ (\s(X,L\bb{S}) \wedge Z).\]
We may use the weak chain rule to see its derivatives are
\[L\bb{S} \circ \mathrm{com} \circ \s(X,L(\bb{S})) \in \mathrm{SymSeq}(L\s)\]
where the computation \[\partial_\ast \Sigma^\infty \Omega^\infty \simeq \mathrm{com}:= \{\bb{S}\}_{n \in \mathbb{N}}\]
is classical, following from the Snaith splitting \cite{snaithsplitting}.

By immediate computation, one has
\[\partial_\ast L\Sigma^\infty R_X \simeq (X^\vee)^{\wedge \ast} \in \mathrm{SymSeq}(L\s),\]
which we note agrees with what one expects from  Lemma \ref{lem:dualizable natural trans spec} in combination with the Yoneda lemma.
\end{proof}


\begin{lem}\label{lem: coend is retract of dual derivatives}
    For $C$ differentially dualizable, the symmetric sequence
    $\mathrm{CoEnd}(\Sigma^\infty_C)$  is a retract of $\partial_\ast (\Sigma^\infty_C\Omega^\infty_C)^\vee$. In particular, $\mathrm{CoEnd}(\Sigma^\infty_C)$ is levelwise dualizable.
\end{lem}

\begin{proof}

    The functors $\Sigma^\infty_C$ and $(\Sigma^\infty_C)^{\wedge n}$ are homogeneous and finitary, and so by \cite[Proposition B.4]{HeutsGoodwillieApprox} an extra codegeneracy argument implies that taking the comonadic $\Sigma^\infty_C \Omega^\infty_C$-resolution of these functors provides a retraction:
    \[\mathrm{RComod}_{\Sigma_C^\infty \Omega^\infty_C}(\Sigma^\infty_C \circ \Omega_C^\infty, (\Sigma^\infty_C)^{\wedge n} \circ \Omega_C^\infty  ) \rightarrow \nat(\Sigma^\infty_C,(\Sigma^\infty_C)^{\wedge n}).\]

    The right $\Sigma^\infty_C \Omega^\infty_C$-comodule $(\Sigma^\infty_C)^{\wedge n} \circ \Omega_C^\infty$ can be equivalently written as  $\mathrm{Id}_{L\s}^{\wedge n} \circ \Sigma^\infty_C \Omega_C^\infty$ which is the cofree right comodule on $\mathrm{Id}_{L\s}^{\wedge n}$. Hence,
    \[\mathrm{RComod}_{\Sigma_C^\infty \Omega^\infty_C}(\Sigma^\infty_C \circ \Omega_C^\infty, (\Sigma^\infty_C)^{\wedge n} \circ \Omega^\infty_C ) \simeq \nat(\Sigma_C^\infty \Omega^\infty_C,\mathrm{Id}_{L\s}^{\wedge n}).\]
Applying Lemma \ref{lem:dualizable natural trans spec} and the levelwise dualizability of $\partial_\ast (\Sigma^\infty_C \Omega^\infty_C)$ yields:
    \[\nat(\Sigma_C^\infty \Omega^\infty_C,\mathrm{Id}_{L\s}^{\wedge n})\simeq (\partial_n \Sigma^\infty_C \Omega^\infty_C)^\vee.\]
This demonstrates the claimed retraction.
\end{proof}
\begin{remark}
Proposition \ref{prop:product rule implies coend is koszul dual to identity} shows these are actually equivalent:
\[\mathrm{CoEnd}(\Sigma^\infty_C) \simeq \partial_\ast (\Sigma^\infty_C \Omega^\infty_C)^\vee.\]
\end{remark}

\begin{definition}
    The Koszul dual derivatives of $F\in \mathrm{Fun}^\omega (C,L\s)$ are
    \[\partial^\ast F := \mathrm{Hom}^\Sigma(F,\Sigma^\infty_C).\]
    These define a contravariant functor
    \[\partial^\ast:\mathrm{Fun}^\omega(C,L\s) \rightarrow \mathrm{RMod}^\mathrm{op}_{\mathrm{CoEnd}(\Sigma^\infty_C)}.\]
\end{definition}

More generally, $\partial^\ast (-)$, somewhat tautologically, is defined for nonfinitary functors as long as the given natural transformation objects exist. For instance, if $F$ is left Kan extended from a small subcategory, the natural transformation objects exist because they can be computed by restricting to a subcategory where size issues are not a concern. In particular, this includes all representable functors.

There are a few immediately accessible facts about Koszul dual derivatives. The operad $\mathrm{CoEnd}(\Sigma^\infty_C)$ is reduced by elementary properties of homogeneous functors. The Yoneda lemma \cite{HinichYoneda} allows one to compute the operad action on Koszul dual derivatives of representables.

\begin{prop}\label{prop: koszul dual derivatives of rep}
    If $c \in C$, there is a natural equivalence of right modules \[\partial^\ast L\Sigma^\infty R_C \simeq (\Sigma^\infty_C c)^{\wedge \ast}\] where $\Sigma^\infty_C c^{\wedge \ast}$ is equipped with the tautological action of $\mathrm{CoEnd}(\Sigma^\infty_C)$.
\end{prop}
Note this proposition holds even when $c\in C$ is not compact. In order to understand the Koszul dual derivatives of general finitary functors, we start by understanding them for homogeneous functors. 
\begin{prop}\label{prop: dual derivatives of homogeneous}
    Suppose $X \in L\s^{B\Sigma_n}$ is dualizable. There is an equivalence of right modules \[\partial^\ast (X \wedge_{h\Sigma_n}( \Sigma^\infty_C)^{\wedge n}) \simeq \mathrm{Free}^\mathrm{RMod}_{\mathrm{CoEnd}(\Sigma^\infty_C)}(X^\vee) . \]
\end{prop}

\begin{proof}
    We will compute $ \nat(X \wedge_{h\Sigma_n} (\Sigma^\infty_C)^{\wedge n},(\Sigma^\infty_C)^{\wedge m})$ by decomposing the domain and codomain. Given a function $\phi:m \rightarrow n$, there is a preferred identification \[(\Sigma^\infty_C)^{\wedge m}\simeq \bigwedge^n_{i=1} (\Sigma^\infty_C)^{\wedge \phi^{-1} (i)}.\] 
    
The symmetric monoidal structure on $\mathrm{Fun}^\omega (C,L\s)$ induces a map
  \[X^\vee \wedge_{h\Sigma_n} \bigvee_{ \phi \in \mathrm{FinSet}(m,n)} \bigwedge_{i=1}^n \nat(\Sigma^\infty_C,(\Sigma^\infty_C)^{\wedge{ \phi^{-1}(i)}}) \rightarrow \nat(X \wedge_{h\Sigma_n} (\Sigma^\infty_C)^{\wedge n},(\Sigma^\infty_C)^{\wedge m}).\]
  The left hand side is $\mathrm{Free}^\mathrm{RMod}_{\mathrm{CoEnd}(\Sigma^\infty_C)}(X^\vee)(m)$. As before, we resolve by the $\Sigma^\infty_C \Omega^\infty_C$-comonad and use cofreeness to instead check the following is an equivalence
   \[X^\vee \wedge_{h\Sigma_n} \bigvee_{ \phi \in \mathrm{FinSet}(m,n)} \bigwedge_{i=1}^n \nat(\Sigma^\infty_C \Omega^\infty_C,\mathrm{Id}_{L\s}^{\wedge ^{\wedge{\phi^{-1}(i)}}}) \rightarrow \nat(X \wedge_{h\Sigma_n} (\Sigma^\infty_C \Omega^\infty_C)^{\wedge n},\mathrm{Id}_{L\s}^{\wedge m}),\]
which suffices because the original map is a retract of this new map. This new map is an equivalence by the $L$-local version of the weak dual stable chain rule \cite[Lemma 2.5.9]{ACOperads} applied to \[(X \wedge_{h\Sigma_n} (-)^{\wedge n} ) \circ \Sigma^\infty_C \Omega^\infty_C.\]
\end{proof}

In our study of the Goodwillie tower of $F(c)$, two different finiteness criteria will arise. The first is the levelwise dualizability of $\partial_\ast F$. The second is the compactness of $c \in C$. These two hypotheses play dual roles, and we will define two different approximations of $P_\infty(F)$ which agree when $\partial_\ast(F)$ is levelwise dualizable and $c$ is compact, but not in general.

\begin{definition}\label{definition: upper tower}
    Suppose $F \in \mathrm{Fun}^\omega(C,L\s)$ is pointed. We  define the upper fake Goodwillie tower of $F$ as a filtered object of $\mathrm{Fun}(C,L\s)$ given by
    \[P_n^\mathrm{fake}(F)(c):=\mathrm{RMod}_{\mathrm{CoEnd}(\Sigma^\infty_C)}(\partial^{\ast} F ,(\Sigma^\infty_C(c)^{\wedge \ast})^{\leq n}).\]
\end{definition}

 There is a map $F(c) \rightarrow P_n^\mathrm{fake}(F)(c)$ which is constructed by via the Yoneda lemma and Proposition \ref{prop: koszul dual derivatives of rep}:
 \[F(c) \simeq \nat(L\Sigma^\infty R_c,F) \rightarrow \mathrm{RMod}_{\mathrm{CoEnd}(\Sigma^\infty_C)}(\partial^\ast F,(\Sigma^\infty_C(c)^{\wedge \ast})^{\leq n}) ).\]
 
 When $c$ is compact, this map may be constructed as the unit of the adjunction
\[
\adjunction{\partial^{\leq n}}{\mathrm{Fun}_\ast(C^{\omega},L\s)}{\mathrm{RMod}^{\leq n}_{\mathrm{CoEnd}(\Sigma^\infty_C)}}{\Theta^{\leq n}_K}.
\]

We define $D^\mathrm{fake}_n(F)$ as $n$th fiber of this filtered object. The map $F \rightarrow P^\mathrm{fake}_n(F)$ factors through $P_n(F)$, and we get an induced maps 
\[D_n(F) \rightarrow D_n^\mathrm{fake}(F).\]
The upper fake tower is often more accessible than the Goodwillie tower. By identifying the maps between the layers, one can pull back calculations in the upper fake Goodwillie tower to the Goodwillie tower. 

It is no surprise that we will soon introduce a ``lower fake Goodwillie tower''. One studies the elementary properties of the upper and lower fake Goodwillie towers in a rather formal manner, following \cite[Sections 6.2, 6.4]{MTOrthogonal}:
\begin{enumerate}
    \item Construct the upper (lower) fake tower as the unit of an adjunction between $\mathrm{Fun}(C^\omega,L\s)$ and $\mathrm{RMod}^\mathrm{op}_O$ $(\mathrm{RMod}_{K(O)})$. 
    \item For $X \in L\s^{B\Sigma_n}$, compute that the right module associated to $X \wedge_{h\Sigma_n} (\Sigma^\infty_C)^{\wedge n}$ is free (trivial).
    \item  Demonstrate that $D^\mathrm{fake}_n(F)$ $(D_\mathrm{fake}^n(F))$ depends only on $\partial_n F$.
    \item Compute that the map between the layers of the Goodwillie tower and either fake Goodwillie tower for the homogeneous functor $X \wedge_{h\Sigma_n} (\Sigma^\infty_C)^{\wedge n}$ is the norm map.
\end{enumerate}

Koszul duality enters the picture in the analysis of the fibers $D^\mathrm{fake}_n(F)$, and informs the types of finiteness assumptions required on $\partial_\ast F$ for the fake tower to be well behaved.

For $P^\mathrm{fake}_n$ we have already demonstrated the first two elements as Definition \ref{definition: upper tower} and Proposition \ref{prop: dual derivatives of homogeneous}. The latter two items are rather formal and appear several times in various forms throughout \cite{ACOperads,MTOrthogonal}; we will not reproduce such arguments here and rather state the final result.

\begin{prop}\label{prop: fake tower koszul dual derivatives}
    For a pointed finitary functor $F$ with levelwise dualizable derivatives, there is a natural equivalence
    \[\partial_\ast (F) \simeq K(\partial^\ast (F)).\]
For $c\in C$, the layers of the filtered map \[P_\ast(F)(c) \rightarrow P_\ast^\mathrm{fake}(F)(c)\]
    \begin{center}
\begin{tikzcd}
(\partial_nF \wedge (\Sigma^\infty_C c)^{\wedge n})_{h\Sigma_n}  \arrow[r] & (K(\partial^\ast F)(n) \wedge (\Sigma^\infty_C c)^{\wedge n})^{h\Sigma_n} \arrow[d,"\simeq"] \\
D_n(F)(c) \arrow[u, "\simeq"]                            & D_n^\mathrm{fake}(F)(c)                                                                  
\end{tikzcd}
\end{center}
    are equivalent to the $\Sigma_n$-norms
   of $D_n(F)(c)$.
\end{prop}

\begin{cor}\label{cor: derivatives of compact representable}
    If $c \in C^\omega$, then there is an equivalence of symmetric sequences
    \[\partial_\ast L\Sigma^\infty R_c \simeq K(\Sigma^\infty_C c^{\wedge \ast}).\]
\end{cor}
\begin{proof}
   The equivalence follows from Proposition \ref{prop: koszul dual derivatives of rep} as long as $\partial_\ast L\Sigma^\infty R_c$ is levelwise dualizable by Proposition \ref{prop: fake tower koszul dual derivatives}. Recall there is a totalization formula \cite[Proposition B.4]{HeutsGoodwillieApprox}
    \[\partial_n  F \simeq \mathrm{Tot}(\partial_n ( F \Omega_C^\infty \circ (\Sigma_C^\infty \Omega_C^\infty)^{\bullet} \circ \Sigma_C^\infty)). \]
    This implies an equivalence
    \[\partial_n(F) \simeq \mathrm{Tot}(\partial_n (F \Omega_C^\infty \circ (\Sigma_C^\infty \Omega_C^\infty)^\bullet))\]
    since precomposition of a functor $L\s \rightarrow L\s$ with $\Sigma^\infty_C$ does not change derivatives. After applying the weak stable chain rule, we see (1) this is a finite totalization and (2) if $\partial_\ast F \Omega^\infty_C$ is levelwise dualizable, then every term in the totalization is levelwise dualizable. The proof is finished by observing that $L\Sigma^\infty R_c \circ \Omega^\infty_C$ is represented by $\Sigma^\infty_C c \in L\s^\omega$, hence the derivatives of this are levelwise dualizable by Corollary \ref{cor:representables for spectra}.
\end{proof}

Our analysis up to this point has used the \textit{pointed} $L$-local spectral representables $L\Sigma^\infty R_c$. A posteori one is able to deduce the analogous results for the ``unpointed''
 $L$-local spectral representable $L\Sigma^\infty_+ R_c$ since these functors differ only by wedge sum with the constant functor $L\bb{S}$. The upper fake Goodwillie tower extends to not necessarily pointed functors in the expected way if one replaces $\Sigma^\infty_C X^{\wedge \ast }$ by $\Sigma^\infty_C X^{\wedge \ast } \vee \{L\bb{S}\}_{n =0}$.\footnote{It is elementary to show that every right module over a reduced operad admits its arity $0$ component as a direct summand with a trivial action.}

\begin{definition}\label{definition: right module derivatives}
    The derivatives of $F\in\mathrm{Fun}(C^\omega,L\s)$ are given by the coend
    \[\partial_\ast F:=\int^{X \in C^\omega} (K(\Sigma^\infty_C X^{\wedge \ast }) \vee \{L\bb{S}\}_{n =0}) \wedge F(X),\]
    and thus define a functor
    \[\partial_\ast :\mathrm{Fun}^\omega(C,L\s) \rightarrow \mathrm{RMod}_{K(\mathrm{CoEnd(\Sigma^\infty_C)})}.\]
    
\end{definition}

\begin{example}
    By the coYoneda Lemma $\partial_\ast L \Sigma^\infty_+ R_c \simeq K(\Sigma^\infty_C X^{\wedge \ast }) \vee \{L\bb{S}\}_{n =0}$.
\end{example}

\begin{prop}\label{prop: right module construction of derivatives}
    The above definition is a model of the Goodwillie derivatives.
\end{prop}
\begin{proof}
    This follows from both sides preserving colimits and Corollary \ref{cor: derivatives of compact representable}.
\end{proof}

One has adjunctions
\[
\adjunction{\partial_{\leq n}}{\mathrm{Fun}(C^{\omega},L\s)}{\mathrm{RMod}^{\leq n}_{K(\mathrm{CoEnd}(\Sigma^\infty_C))}}{\Theta_{\leq n}^{\mathrm{RMod}}},
\]
where the right adjoint now has the formula
\[R \mapsto \mathrm{RMod}_{K(\mathrm{CoEnd}(\Sigma^\infty_C))}(\partial_\ast L\Sigma^\infty_+ R_{(-)},R).\]
Note that no truncations appear in the formula since $R$ is already $n$-truncated. Similarly, there is an adjunction $(\partial_\ast,\Theta^{\mathrm{RMod}})$ where the right adjoint is defined on the entire category of right modules. We will consider several variants of $\Theta$ throughout this paper. All of them are defined as right adjoints to (truncated) Goodwillie derivatives, but considered as functors into different categories. We differentiate these adjoints by the upper right index which we use to denote what category their respective left adjoints, the derivatives, take value in.

\begin{definition}\label{definition: lower fake goodwillie tower}

 The lower fake Goodwillie tower of $F \in \mathrm{Fun}(C^\omega,L\s)$
 \[ F \rightarrow P^n_\mathrm{fake}(F)\]
 is the filtered object defined by the adjunction units \[F \rightarrow \Theta^{\mathrm{RMod}}_{\leq n} \circ \partial_{\leq n}.\]
\end{definition}

As is the case for the upper fake Goodwillie tower, the universal property of $P_n$ guarantees a factorization
\[F \rightarrow P_n(F) \rightarrow P_\mathrm{fake}^n(F),\]
and defines a map from the Goodwillie tower to the lower fake Goodwillie tower.

Using Koszul duality, it is straightforward to extend our results about the the upper fake Goodwillie tower to the lower fake Goodwillie tower. Because $\partial_\ast$ preserves colimits, rather than sending colimits to limits as $\partial^\ast$ does, one can completely remove levelwise dualizability assumptions on $\partial_\ast F$, though this comes at the cost of evaluating only on compact objects. 

\begin{prop} \label{prop: fake no finiteness}
 For $F \in \mathrm{Fun}^\omega(C,L\s)$ and $c\in C^\omega$, the layers of \[P_\ast(F)(c) \rightarrow P^\ast_\mathrm{fake}(F)(c)\]
    \begin{center}
\begin{tikzcd}
(\partial_nF \wedge (\Sigma^\infty_C c)^{\wedge n})_{h\Sigma_n}  \arrow[r] & (K(K(\Sigma^\infty_C c^{\wedge \ast}))(n) \wedge \partial_n F )^{h\Sigma_n} \arrow[d,"\simeq"] \\
D_n(F)(c) \arrow[u, "\simeq"]                            & D_n^\mathrm{fake}(F)(c)                                                                  
\end{tikzcd}
\end{center}
    are equivalent to the $\Sigma_n$-norms
   of $D_n(F)(c)$.
\end{prop}

\begin{definition}
The category of finitary $n$-polynomial functors
    $\mathrm{Poly}^\omega(C,L\s)$
    is the full subcategory of $\mathrm{Fun}^\omega(C,L\s)$ for which
    \[F \rightarrow P_n(F)\]
    is an equivalence.
    The category of all finitary polynomial functors
    \[\mathrm{Poly}^\omega(C,L\s)\]
    is the full subcategory of $\mathrm{Fun}^\omega(C,L\s)$ for which there exists an $N$ such that 
    \[F \rightarrow P_N(F)\]
    is an equivalence. Equivalently, it is the union of $\mathrm{Poly}_n^\omega(C,L\s)$ over all $n \in \mathbb{N}$.
\end{definition}

\begin{prop}\label{prop: unit and counit}

Let $V_n$ denote the full subcategory of $\mathrm{Poly}_n^\omega(C,L\s)$ on the functors $F$ which satisfy for all $i$
    \[(\partial_i F \wedge (\Sigma^\infty_C)^{\wedge i})^{t\Sigma_i}\simeq \ast.\]

Let $W_n$ denote the full subcategory of $\mathrm{RMod}^{\leq n}_{K(\mathrm{CoEnd}(\Sigma^\infty_C))}$ on right modules $R$
which satisfy for all $i$
\[(R(i) \wedge (\Sigma^\infty_C)^{\wedge i})^{t\Sigma_i}\simeq \ast.\]
Then $(\partial_{\leq n},\Theta^\mathrm{RMod}_{\leq n})$ restricts to an equivalence of categories
\[
\adjunction{\partial_{\leq n}}{V_n}{W_n}{\Theta_{\leq n}^{\mathrm{RMod}}},
\]

Similarly, if $V=\bigcup_{n\in \mathbb{N}}V_n$ and $W=\bigcup_{n\in \mathbb{N}}W_n$, then $(\partial_\ast,\Theta^\mathrm{RMod})$ restricts to an equivalence
\[
\adjunction{\partial_\ast}{V}{W}{\Theta^{\mathrm{RMod}}}.
\]
\end{prop}
\begin{proof}
That any given left adjoint restricts to the given categories is immediate from definitions. That the corresponding right adjoint restricts to the given categories follows from the Tate vanishing combined with the description of the layers of the truncation filtration of $\mathrm{RMod}_{K(\mathrm{CoEnd}(\Sigma^\infty_C))}(K(\Sigma^\infty_C c^{\wedge \ast }),R)$ of Section \ref{section: right module primitive stuff}.

    That the unit of any given adjunction is an equivalence follows from the Tate vanishing and Proposition \ref{prop: fake no finiteness}. In general, the counit is trickier to study because it involves computing derivatives of the homotopy fixed points which appear in the layers of the truncation filtration. However, if $R \in W_n$ the Tate vanishing assumptions imply the Goodwillie tower of $\Theta^{\mathrm{RMod}}_{\leq n}(R)$ is given by the truncation filtration on the right module mapping spectra. This implies the derivatives are $R$. The analysis extends to the union of all the $W_n$ since for a given polynomial functor, all the towers which appear in the analysis are finite towers.
\end{proof}

\begin{prop}\label{prop: approximation is sym mon}
    The $n$th polynomial approximation
    \[P_n: \mathrm{Fun}^\omega(C,L\s) \rightarrow \mathrm{Poly}^\omega_n(C,L\s)\]
    is a symmetric monoidal Bousfield localization of $(\mathrm{Fun}^\omega(C,L\s),\wedge)$.
\end{prop}
\begin{proof}
    By definition, $P_n$ is the Bousfield (i.e. reflective) localization of $\mathrm{Fun}^\omega(C,L\s)$ at $\mathrm{Poly}^\omega_n(C,L\s)$, and so we must only check the symmetric monoidality. By  \cite[Proposition 2.2.1.9]{HA}, it suffices to verify that if $P_n(F \rightarrow G)$ is an equivalence, then for any finitary functor $H$ the natural transformation $P_n (F \wedge H \rightarrow G \wedge H)$ is an equivalence. By \cite[Lemma 3.1.33 (ii)]{blans2024chainrulegoodwilliecalculus}, we see that \[ P_n((-)^{\wedge 2 } \circ (F \vee H) \rightarrow (-)^{\wedge 2 } \circ (G \vee H))\] is an equivalence. The result follows because $F \wedge H \rightarrow G \wedge H $ is a retract of this natural transformation.

\end{proof}

\begin{thm}[Product rule]\label{thm: product rule}
    The functors \[\partial_\ast: (\mathrm{Fun}(C^\omega,L\s),\wedge) \rightarrow (\mathrm{RMod}_{K(\mathrm{CoEnd}(\Sigma^\infty_C))},\circledast),\]
    \[\partial^\ast: (\mathrm{Fun}(C^\omega,L\s),\wedge) \rightarrow (\mathrm{RMod}^{\mathrm{op}}_{\mathrm{CoEnd}(\Sigma^\infty_C)},\circledast)\]
    are symmetric monoidal.
\end{thm}

\begin{proof}
Let $D$ be an $L\s$-enriched symmetric monoidal category with a chosen subset of objects $d_\alpha$ closed under the symmetric monoidal product. In general,  the contravariant restricted Yoneda embedding 
\[d \mapsto D(d,d_\alpha) \in \mathrm{Fun}_{L\s}(\langle d_\alpha\rangle,L\s)^\mathrm{op}\]
is lax symmetric monoidal with respect to Day convolution \cite{envBMT}, so to check if it is symmetric monoidal it suffices to check the lax comparison map is an equivalence. 

We will apply this to the restricted Yoneda embedding for the functors $(\Sigma^\infty_C)^{\wedge n}$. Proposition \ref{prop: dual derivatives of homogeneous} implies 
\[\mathrm{Env}(\mathrm{CoEnd}(\Sigma^\infty_C))^\mathrm{op} \simeq \langle(\Sigma^\infty_C)^{\wedge n} \rangle_{n \in \mathbb{N}},\]
and so, by the correspondence between right $\mathrm{CoEnd}(\Sigma^\infty_C)$-modules and enriched presheaves on the envelope, the Koszul dual derivatives $\partial^\ast$ are identified with the restricted Yoneda embedding. 

To prove that the lax comparison map is an equivalence, it suffices to show the lax comparison map is an equivalence after $N$-truncation for all $N$. The universal property of Goodwillie approximations and Proposition \ref{prop: approximation is sym mon} reduce this to checking the lax comparison map when the functors $F,G$ are polynomials. By induction on the first variable, we are reduced to checking on homogeneous $F$ and general polynomial $G$. Induction on the second variable allows us to assume $G$ is homogeneous. Since a general homogeneous functor is a colimit of one of the form $(\Sigma^\infty_C)^{\wedge n}$ for some single $n$, and the Koszul dual derivatives $\partial^\ast$ send colimits to limits, it suffices to check the lax comparison map for $(\Sigma^\infty_C) ^{\wedge n} \wedge (\Sigma^\infty_C) ^{\wedge m}$.

The homogenous functors of this form are what we perform the restricted Yoneda embedding against to define $\partial^\ast$, and so that the lax comparison maps are equivalences follows from the symmetric monoidality of the Yoneda embedding \cite[8.4.3]{HinichColimits}.

To show that $\partial_\ast$ admits a lax symmetric monoidal structure, by \cite[Lemma 2.59]{pdoperads}, it suffices to show that it is lax symmetric monoidal when restricted to the full symmetric monoidal subcategory generated by $L\Sigma^\infty_+ R_{c_k}$ for $c_k \in C^\omega$. This is because the referenced lemma asserts the left Kan extension is automatically lax symmetric monoidal, and the Yoneda lemma in $C^\omega \times C^\omega$ will imply that $\partial_\ast$ is actually fully symmetric monoidal on all of $(\mathrm{Fun}(C^\omega,L\s),\wedge)$ since both $\partial_\ast(F \wedge G)$ and $\partial_\ast (F) \wedge \partial_\ast(G)$ commute with colimits in either variable.  A slight variation of the argument of Corollary \ref{cor: derivatives of compact representable} shows that the derivatives \[\partial_\ast (L\Sigma_+^\infty R_{c_1} \wedge \dots \wedge L\Sigma_+^\infty R_{c_{k'}})\]
are levelwise dualizable, and so when restricted to this subcategory there is an equivalence of functors
\[\partial_\ast \simeq K \circ \partial^\ast\]
by Proposition \ref{prop: fake tower koszul dual derivatives}. The righthand side is symmetric monoidal because $K$ is a symmetric monoidal functor by Proposition \ref{prop: koszul duality right modules}, and we have just demonstrated the same for $\partial^\ast$.
\end{proof}
\subsection{Comparison with the operad  $\partial_\ast\mathrm{Id}_C$}\label{section: comparison}
In \cite{blans2024chainrulegoodwilliecalculus}, it was shown that the Goodwillie derivatives were monoidal with respect to functor composition and the composition product. Let us call this monoidal functor $\partial_\ast^{BB}$. If $\s(C) \simeq L\s$, then a formal consequence of the chain rule is that
$\partial^{BB}_\ast \mathrm{Id}_C$ forms an operad in $L\s$, and there is a functor lifting the Goodwillie derivatives
\[\partial_\ast^{BB}:\mathrm{Fun}_\ast(C^\omega, L\s) \rightarrow \mathrm{RMod}_{\partial^{BB}_\ast \mathrm{Id}_C}.\]

In light of the product rule for the right $K(\mathrm{CoEnd}(\Sigma^\infty_C))$-module structures on derivatives, one conjectures:
\begin{conj} \label{conj: symmetric monoidal derivatives}
    The functor \[\partial^{BB}_\ast: (\mathrm{Fun}_\ast^\omega(C,L\s),\wedge) \rightarrow (\mathrm{RMod}_{\partial_\ast \mathrm{Id}},\circledast)\]
    can be made symmetric monoidal.
\end{conj}
Recently, Blom has announced a proof of this conjecture which works by lifting the classical derivation of the product rule from the chain rule to the level of symmetric monoidal categories:
\[(fg)' = \frac{1}{2}(((f+g)^{2})'-(f^2)'-(g^2)')= \frac{1}{2}((2(f+g)(f'+g')-2ff'-2gg')=fg'+f'g.\]

If $C$ is compactly generated and differentially dualizable, any model of the Goodwillie derivatives
\[\partial_\ast': \mathrm{Fun}^\omega(C,L\s) \rightarrow \mathrm{RMod}_O\]
which satisfies a product rule and provides a reasonable right module structure on $\partial_\ast'L\Sigma^\infty R_c$ for all $c \in C^\omega$ necessarily satisfies \[K(O) \simeq  \mathrm{CoEnd}(\Sigma^\infty_C).\]
We sketch the proof of this for $\partial_\ast^{BB}$.

\begin{prop}\label{prop:product rule implies coend is koszul dual to identity}
    If $C$ is differentially dualizable, there is an equivalence of symmetric sequences
    \[K (\partial^{BB}_\ast \mathrm{Id}_C) \simeq \mathrm{CoEnd}(\Sigma^\infty_C).\]

    Assuming Conjecture \ref{conj: symmetric monoidal derivatives}, this can be made into an equivalence of operads.

\end{prop}
\begin{proof}
    For trivial reasons, there is an equivalence of right modules \[\partial_\ast^{BB}\Sigma^\infty_C \simeq 1 \in \mathrm{RMod}_{\partial_\ast^{BB}\mathrm{Id}_C}.\] The symmetric monoidality of the product rule in Conjecture \ref{conj: symmetric monoidal derivatives} implies that  taking $\partial_\ast^{BB}$ gives a map of operads
    \[\mathrm{CoEnd}(\Sigma^\infty_C) \rightarrow \mathrm{CoEnd}(\partial^{BB}_\ast \Sigma^\infty_C) \simeq \mathrm{CoEnd}(1)=:K(\partial^{BB}_\ast(\mathrm{Id}_C)).\]
    If one does not have symmetric monoidality of $\partial_\ast^{BB}$, the map still exists at the level of symmetric sequences and the following argument will show it is a quasi-equivalence of operads, as defined in Definition \ref{definition: quasiequivalence}.

    It suffices to demonstrate that the unit and counit of the adjunction $(\partial_\ast^{BB},\Theta^{\mathrm{RMod}}_{BB})$, defined similarly to before, are equivalences when restricted to the full subcategories $\langle (\Sigma^\infty_C) ^{\wedge n}\rangle_{n \in \mathbb{N}} $ and $\langle \mathrm{Triv}^\mathrm{RMod}_{\partial_\ast \mathrm{Id}_C} (L\Sigma^\infty_+\Sigma_n)\rangle_{n \in \mathbb{N}}$, respectively. This is because the former is the opposite of the envelope of $\mathrm{CoEnd}(\Sigma^\infty_C)$, while the latter is the opposite of the envelope of $K(\partial^{BB}_\ast \mathrm{Id}_C)$.

 The analysis of units and counits in Proposition \ref{prop: unit and counit} goes through almost entirely unchanged if one instead uses $\partial_\ast^{BB}$ provided one knows that for $c \in C^\omega$ there is an equivalence of symmetric sequences \[K(\partial_\ast^{BB} L\Sigma^\infty R_c) \simeq (\Sigma^\infty_C c) ^{\wedge \ast}.\] This is true as a consequence of the cobar formula for $\partial_\ast^{BB}$ \cite[Corollary 4.4.4]{blans2024chainrulegoodwilliecalculus} and Koszul duality for right modules \[\partial_\ast^{BB}(L\Sigma^\infty R_c) \simeq \Omega(\partial^{BB}_\ast (L\Sigma^\infty R_c \circ \Omega^\infty_{L\s}), \partial_\ast^{BB} \Sigma^\infty_C\Omega^\infty_C, 1) \simeq \Omega (\partial_\ast L\Sigma^\infty R_{\Sigma^\infty_C c} ,\partial_\ast^{BB} \Sigma^\infty_C\Omega^\infty_C, 1)\]  since $\partial_\ast L\Sigma^\infty R_{\Sigma^\infty_C c}= ((\Sigma^\infty_C c )^{\wedge \ast})^\vee$ by Corollary \ref{cor:representables for spectra}.
Thus the analogue of Proposition \ref{prop: unit and counit} applies to show the unit and counit on these subcategories are equivalences since the derivatives of $(\Sigma^\infty_C )^{\wedge n}$ are free, and so have Tate vanishing.

\end{proof}

\begin{remark}
    In order to compare right module structures on the derivatives of arbitrary finitary functors, one needs to show that under the equivalence
    \[K(\partial^{BB}_\ast \mathrm{Id}_C) \simeq \mathrm{CoEnd}(\Sigma^\infty_C c)\]
     the equivalence of symmetric sequences used in the above proof
    \[K (\partial^{BB}_\ast L\Sigma^\infty R_c) \simeq (\Sigma^\infty_C c)^{\wedge \ast}\] lifts to an equivalence of right modules.
\end{remark}

\subsection{Goodwillie calculus of spaces}\label{section: goodwillie of spaces}

\begin{prop}\label{prop:coend in spaces}
    The map of operads \[\mathrm{com} \rightarrow \mathrm{CoEnd}(\Sigma^\infty)\] 
    induced by the diagonal commutative coalgebra structure on $\Sigma^\infty$ is an equivalence.
\end{prop}

\begin{proof}
    Since Lemma \ref{lem: coend is retract of dual derivatives} shows $\mathrm{CoEnd}(\Sigma^\infty)$ is a retract of $(\partial_\ast (\Sigma^\infty \Omega^\infty)) ^\vee \simeq \mathrm{com}$ it suffices to show the map is injective in homology. For any space $X$ there is an evaluation map \[\mathrm{CoEnd}(\Sigma^\infty)(n) \rightarrow \s(\Sigma^\infty X,\Sigma^\infty X^{\wedge n}),\] and so by taking homology it suffices to show that there is some space $X$ such that any multiple of the map $H_\ast(X) \rightarrow H_\ast(X^ {\wedge n} )$ is nonzero. This is of course true, take $X = \mathbb{C}P^\infty$ where this fact is witnessed by the cup product.

\end{proof}

We call $K(\mathrm{com})$ the spectral Lie operad and denote it $\mathrm{lie}$. For finitary functors $F:\T \rightarrow \s$, Section \ref{section: right module construction} combined with the above computation establishes that $\partial_\ast F$ has the structure of a right $\mathrm{lie}$-module. Such structure on the derivatives was first constructed in \cite{ACOperads} using a chain rule.

\begin{ex} \label{ex: derivatives of representable spaces}
    By Corollary \ref{cor: derivatives of compact representable}, if $X$ is compact, i.e. a retract of a finite CW-complex, there is an equivalence of right $\mathrm{lie}$-modules  \[\partial_\ast \Sigma^\infty R_X \simeq K(\Sigma^\infty X^{\wedge \ast}),\] where the commutative operad action on $\Sigma^\infty X^{\wedge \ast}$ is by the diagonal. This exactly agrees with the right module structure computed in \cite[Example 17.28]{ACOperads}. The homotopy type of $K(\Sigma^\infty X^{\wedge \ast})$ is rather understandable. It can be computed as $(X^{\wedge \ast}/\Delta^{\mathrm{fat}})^\vee$, where $\Delta^{\mathrm{fat}}$ denotes the subspace of elements where some coincide, recovering the original calculation in \cite[Section 7]{GoodCalcIII}. Since $(\Sigma^\infty X^{\wedge \ast}/\Delta^{\mathrm{fat}})^\vee$ is a $\Sigma$-finite symmetric sequence, by which we mean it is levelwise equivariantly a finite spectrum, the Goodwillie tower agrees with the upper fake Goodwillie tower. This model of the Goodwillie tower of representables was originally constructed in Arone's thesis \cite{Arone1999}. It also agrees with the lower fake Goodwillie tower when evaluated at a compact space, as noted in \cite[Example 6.28]{ACClassification}.
\end{ex}

\begin{cor}
    The functor
    \[\partial_\ast: \mathrm{Fun}^\omega(\T,\s) \rightarrow \mathrm{RMod}_{\mathrm{lie}}\]
    agrees with the version constructed in \cite{ACOperads}.
\end{cor}
\begin{proof}
    In \cite{envBMT}, it is demonstrated that the theory of Koszul duality of operads and right modules used in \cite{ACOperads} agrees with the Koszul duality of operads and right modules we use in this paper. The above example, together with the commutativity of derivatives with colimits, implies the result since we may express any functor naturally as a coend of representables.
\end{proof}

In unpublished work, recounted in \cite[Pg. 4,5]{AroneChingManifolds} and again in \cite[Proposition 3.15]{crosseffectsclassification}, Dwyer-Rezk demonstrate that there is an equivalence between $\mathrm{Poly}^\omega(\T,\s)$ and the category of bounded commutative comodules\footnote{Unless otherwise noted, all our algebraic objects are non(co)unital.}, defined as the full subcategory of functors \[\mathrm{RComod}^{<\infty}_{\mathrm{com}} \subset \mathrm{Fun}(\mathrm{FinSet}^{\mathrm{sur}}, \s)\] which are trivial on sets with cardinality greater than some arbitrary $N$. The equivalence is given by
\[\adjunction{\int_{(-)} (-)}{\mathrm{RComod}^{<\infty}_{\mathrm{com}}}{\mathrm{Poly}^\omega(\T,\s)}{N[-]}.
\]
where for $X \in \T$, the spectrum $\int_X \bar{C}$ is the coend computing the factorization homology of a commutative comodule $\bar{C}$ over $X$, which is by definition
\[\int_X \bar{C} := \int^{I \in \mathrm{FinSet}^\mathrm{sur}}\bar{C}(I) \wedge \Sigma^\infty X^{\wedge I}, \]
and the right adjoint is given by
\[N[F]:=\nat((\Sigma^\infty)^{\wedge \ast },F).\]

Let $\mathrm{com}^\mathrm{un}$ denote the unital commutative operad. There is a Pirashvili-type equivalence \cite[Theorem 3.78]{crosseffectsclassification} between commutative comodules and augmented unital commutative comodules defined as functors \[\mathrm{Comod}_\mathrm{com^{un}}^\mathrm{aug}:=\mathrm{Fun}(\mathrm{FinSet}_\ast, \s).\] 
We denote this Pirashvili equivalence by
\[\adjunction{\epsilon}{\mathrm{Comod}_\mathrm{com^{un}}^\mathrm{aug}}{\mathrm{RComod}_{\mathrm{com}}}{\epsilon^{-1},}
\]
and remark that the functor $\epsilon$ lifts the cross-effects. 
The composition of the Dwyer-Rezk equivalence and the Pirashvili equivalence is simply \cite[Theorem 6.1.5.1]{HA}
\[\adjunction{\mathrm{LKan}}{\mathrm{Comod}_\mathrm{com^{un}}^\mathrm{aug}}{\mathrm{Poly}^\omega(\T,\s)}{\mathrm{res}}.
\]
Recall from Section \ref{section: right module primitive stuff} that given a comodule $C$ over reduced operad $O$, one can define the indecomposables $B(C)$ which have the structure of a divided power right $K(O)$-module. Dually, for a right module $R$ over $K(O)$, we may define the primitives $\Omega(R)$ which have the structure of a right comodule over $O$.

\begin{prop}\label{prop: derivatives of polynomial}
    If $F\in \mathrm{Poly}^\omega(\T,\s)$, then there is an equivalence of right $\mathrm{lie}$-modules
    \[\partial_\ast F \simeq B(\epsilon (F|_{\mathrm{FinSet}_\ast})).\]
\end{prop}

\begin{proof}
We prove the result by establishing it for $P_n\Sigma^\infty R_X$, for all $n$ and $X \in \T^\omega$. This suffices since both sides of the equivalence commute with colimits, and so we may Kan extend from representables to get the result for a general pointed polynomial functor.

Recall from Example \ref{ex: derivatives of representable spaces}, that for $Y\in \T^\omega$, the $\Sigma$-finiteness of $\partial_\ast \Sigma^\infty R_X \simeq K(\Sigma^\infty X ^{\wedge \ast})$ implies there is an equivalence \[P_n(\Sigma^\infty R_X)(Y) \simeq \mathrm{RMod}_{\mathrm{lie}}(K(\Sigma^\infty Y ^{\wedge \ast}),K(\Sigma^\infty X ^{\wedge \ast})^{\leq n}).\]

We now compute $K(\Sigma^\infty [k]_+ ^{\wedge \ast})$. The set of tuples of points in $[k]_+$ have a standard decomposition in terms of their images which extends to an equivalence of right $\mathrm{com}$-modules
\[\Sigma^\infty [k]_+ ^{\wedge \ast} \simeq \bigvee_{U\in 2^{[k]}} \mathrm{Free}^{\mathrm{RMod}}_\mathrm{com} \Sigma^\infty_+ \Sigma_{U}.\]
After taking Koszul duals there is an equivalence
\[K(\Sigma^\infty [k]_+ ^{\wedge \ast}) \simeq \bigvee_{U\in 2^{[k]}} \mathrm{Triv}^{\mathrm{RMod}}_\mathrm{lie} \Sigma^\infty_+ \Sigma_{U}. \]
Mapping out of a trivial right module is expressed in terms of the primitives, via adjunction, and so the Goodwillie tower is
\[P_n(\Sigma^\infty R_X)([k]_+) \simeq \prod_{U \subset 2^{[k]}} \Omega K(\Sigma^\infty X^{\wedge \ast})^{\leq n}(U). \]

This formula is the inverse of the cross-effects equivalence from augmented commutative comodules to nonunital commutative comodules, and so we have
\[
P_n(\Sigma^\infty R_X)([k]_+) \simeq \epsilon^{-1} (\Omega K(\Sigma^\infty X^{\wedge \ast})^{\leq n}).\]
We invert the outermost functor $\epsilon^{-1}$, and then use that $ K(\Sigma^\infty X^{\wedge \ast})$ is $\Sigma$-finite to conclude by Lemma \ref{lem: counit of bar cobar is equiv on sigma finite} that we may invert $\Omega$:
\[\partial_\ast P_n(\Sigma^\infty R_X )\simeq B(\epsilon(P_n(\Sigma^\infty R_X))|_{\mathrm{FinSet}_\ast})\]
which concludes the proof.
\end{proof}

\begin{cor}
    There is a natural factorization of the derivatives for $ \mathrm{Fun}^\omega(\T,\s)$:
\begin{center}
\[\begin{tikzcd}
	& {\mathrm{RMod}^\mathrm{dp}_{\mathrm{lie}}} \\
	{\mathrm{Fun}^\omega(\T,\s)} & {\mathrm{RMod}_\mathrm{lie}}
	\arrow["{\mathrm{Forget}^\mathrm{RMod}_O}", from=1-2, to=2-2]
	\arrow[from=2-1, to=1-2]
	\arrow["{\partial_\ast}"', from=2-1, to=2-2]
\end{tikzcd}\]
\end{center}
    
\end{cor}
The functor $\partial_\ast:\mathrm{Fun}(\T^\omega,\s) \rightarrow \mathrm{RMod}_\mathrm{lie}^\mathrm{dp}$ has a right adjoint $\Theta^\mathrm{dp}$, and in fact it restricts to an equivalence
\[\adjunction{\partial_\ast}{\mathrm{Poly}^\omega(\T,\s)}{\mathrm{RMod}_\mathrm{lie}^{\mathrm{dp},<\infty}}{\Theta^\mathrm{dp}}.
\]
In light of the Dwyer-Rezk classification of polynomial functors in terms of commutative comodules, the most obvious way to prove this equivalence given Proposition \ref{prop: derivatives of polynomial} is to establish that bar-cobar duality provides an equivalence
\[\adjunction{B}{\mathrm{Comod}^{<\infty}_O}{\mathrm{RMod}_\mathrm{K(O)}^{\mathrm{dp},<\infty}}{\mathrm{\Omega^\mathrm{dp}}}.
\]
While we expect such an equivalence holds generalizing \cite[Proposition 2.7]{AroneChingManifolds}, this approach has little chance of applying to a wider class of categories. An obvious category it fails to address is $\T^{\geq 1}$, the category of simply-connected pointed spaces. This category has the same theory of polynomial functors, but does not contain the category $\mathrm{FinSet}_\ast$. We expect that for a general category $C$, one should instead consider the right $\mathrm{CoEnd}(\Sigma^\infty_C)$-comodule associated to a polynomial $F$ built out of its cross-effects \[\{\nat((\Sigma^\infty_C)^{\wedge n},F)\}_{n \in \mathbb{N}}.\] One can ask whether these cross-effects are bar-cobar dual to $\partial_\ast F$ in the sense of Proposition \ref{prop: derivatives of polynomial}. We expect this is true precisely  when $\mathrm{Poly}^\omega(C,L\s)$ admits a divided power right module classification. However, we will use a more direct approach in Section \ref{section: categories with divided powers}, completely classifying categories that admit divided power classifications.

The appearance of divided powers in the Goodwillie calculus of $\T$ can be understood as a limiting case of certain self duality phenomena which appears in manifold calculus, which we will make precise in \cite{dualityinmfld}. For $R\in \mathrm{RMod}_\mathrm{lie}^{\mathrm{dp}}$, we suggestively define $\int^{X^{\neg}}_\mathrm{dp} R $, the divided power factorization cohomology of $R$ over the formal negation of $X$, as the right adjoint $\Theta^{\mathrm{dp}}$ applied to $R$ and evaluated at $X$
\[\int^{X^{\neg}}_\mathrm{dp} R :=\mathrm{RMod}^\mathrm{dp}_{\mathrm{lie}}(K(\Sigma^\infty X^{\wedge \ast}),R). \]

In light of the divided power classification of $\mathrm{Poly}^\omega(\T,\s)$ that we prove in Section \ref{section: categories with divided powers}, we find that Goodwillie approximation of $F \in \mathrm{Fun}^\omega(\T,\s)$ has the formula
\[P_\infty(F)(X) \xrightarrow{\simeq} \int_\mathrm{dp}^{X^{\neg}} \partial_\ast F.\]

In the case $F(X)= B(\Sigma^\infty X^{\wedge \ast}, \mathrm{com},A)$, the factorization homology of a commutative algebra $A$, this agrees with Amabel's formula \cite[Main Theorem]{generalizedPoincare} which is constructed via Goodwillie calculus in $\mathrm{Alg}_\mathrm{com}$, which we now recall.
\subsection{Goodwillie calculus of algebras over an operad}\label{section: goodwillie calculus algebras}

We fix a reduced operad $O$ in $L\s$ a Bousfield localization of $\s$. Since $O$ is reduced it has an augmentation $O \rightarrow 1$ which has an associated adjunction
\[\adjunction{\mathrm{Indecom}_O}{\mathrm{Alg}_O}{L\s}{\mathrm{Triv}^\mathrm{Alg}_O}.\]

 \begin{definition}
    The category of coalgebras over an operad $O$ in $L\s$ is 
\[ \mathrm{Coalg}_O:=\mathrm{Fun}^\otimes_{L\s}(\mathrm{Env}(O)^\mathrm{op},L\s).\]
 \end{definition}
 We will refer to $C \in \mathrm{Coalg}_O$ as a coalgebra structure on $C(1)$. From this definition and the identification $\mathrm{RMod}_O \simeq \mathrm{Fun}_{L\s}(\mathrm{Env}(O)^\mathrm{op},L\s),$ one is able to produce a forgetful functor
 \[\mathrm{Forget}^\mathrm{RMod}_O: \mathrm{Coalg}_O \rightarrow \mathrm{RMod}_O\]
 \[C \mapsto \{C(1)^{\wedge n}\}_{n \in \mathbb{N}}.\]

\begin{prop}
    If $O$ is an augmented operad, then there is a functor $B$ which lifts $\mathrm{Indecom}_O$ to $K(O)$-coalgebras:
    \begin{center}
\[\begin{tikzcd}
	& {\mathrm{Coalg}_{K(O)}} \\
	{\mathrm{Alg}_O} & {L\s}
	\arrow["{\mathrm{Forget}_O}", from=1-2, to=2-2]
	\arrow["{B}", from=2-1, to=1-2]
	\arrow["{\mathrm{Indecom}_O}", from=2-1, to=2-2]
\end{tikzcd}\]
    \end{center}
\end{prop}
\begin{proof}
    In work in progress \cite{brantnerHeutsUniversal}, Branter-Heuts demonstrate that the functor 
    \[(\mathrm{RMod}_O,\circledast) \rightarrow (L\s,\wedge)\]
    \[R \mapsto B(R,O,A)\] is symmetric monoidal. This implies that
    \[B:\mathrm{Env}(K(O))^\mathrm{op} \rightarrow L\s\]
    \[\mathrm{Triv}_O^\mathrm{RMod}(L\Sigma_+^\infty \Sigma_n) \mapsto B(\mathrm{Triv}_O^\mathrm{RMod}(L\Sigma_+^\infty \Sigma_n),O,A)\]
    defines a $K(O)$-coalgebra structure on \[B(1,O,A) \simeq \mathrm{Indecom}_O(A).\]
\end{proof}

When $O$ is levelwise dualizable, $K(O)$ supplies the universal natural coaction on $\mathrm{Indecom}_O(A)$ by an operad, in the following sense.

\begin{thm}\label{thm: coend in algebras}
    If $O\in \mathrm{Operad}(L\s)$ is a reduced levelwise dualizable operad, then the map of operads \[K(O) \rightarrow \mathrm{CoEnd}(\mathrm{Indecom}_O),\] induced by the natural coaction of $K(O)$ on $\mathrm{Indecom}_O(A)$, is an equivalence.
\end{thm}

\begin{proof}
From Lemma \ref{lem: coend is retract of dual derivatives} and Proposition \ref{prop:product rule implies coend is koszul dual to identity}, we know that the comparison map \[\mathrm{CoEnd}(\mathrm{Indecom}_O)(n) \rightarrow \nat(\mathrm{Indecom}_O \circ \mathrm{Triv}_O^\mathrm{Alg},\mathrm{Id}_{L\s}^{\wedge n}) \] is an equivalence, and so it suffices to show that the composite \[K(O)(n) \rightarrow \mathrm{CoEnd}(\mathrm{Indecom}_O)(n) \rightarrow \nat(\mathrm{Indecom}_O \circ \mathrm{Triv}_O^\mathrm{Alg},\mathrm{Id}_{L\s}^{\wedge n})\] is an equivalence.

We already know that there is an equivalence
\[\nat(\mathrm{Indecom}_O \circ \mathrm{Triv}_O^\mathrm{Alg},\mathrm{Id}_{L\s}^{\wedge n}) \simeq K(O)(n)\]
because of Lemma \ref{lem:dualizable natural trans spec} and the splitting
\[\mathrm{Indecom}_O \circ \mathrm{Triv}_O^\mathrm{Alg} \simeq \bigvee_{i \in \mathbb{N}} B(O)(i) \wedge_{h\Sigma_i} \mathrm{Id}_{L\s}^{\wedge i}. \]
It suffices to show that map induced by the coaction of $K(O)$ on $\mathrm{Indecom}_O(-)$ and the equivalence constructed from computing natural transformations agree.

We claim that this can be checked after composition with the evaluation map 
\[\nat(\mathrm{Indecom}_O \circ \mathrm{Triv}_O^\mathrm{Alg},\mathrm{Id}_{L\s}^{\wedge n}) \rightarrow L\s^{B\Sigma_n}(\mathrm{Indecom}_O \circ \mathrm{Triv}_O^\mathrm{Alg}(L\bb{S}^{\vee n}), (L\bb{S}^{\vee n})^{\wedge n}). \]

This is because the evaluation map includes $\nat(\mathrm{Indecom}_O \circ \mathrm{Triv}_O^\mathrm{Alg},\mathrm{Id}_{L\s}^{\wedge n})$ as a wedge summand. This fact follows from $B(O)(n)$ equivariantly splitting off $B(O)(n) \wedge_{h \Sigma_n} (L\bb{S}^{\vee n})^{\wedge n}$, the weight $n$ part of \[B ( \mathrm{Triv}_O^\mathrm{Alg} (L\bb{S}^{\vee n}))\simeq \mathrm{Indecom}_O \circ \mathrm{Triv}_O^\mathrm{Alg}(L\bb{S}^{\vee n}),\]  and $L\Sigma^\infty_+ \Sigma_n$ equivariantly splitting off $(L\bb{S}^{\vee n})^{\wedge n}$. Since $O(1)= L\bb{S}$, the latter can be identified as the weight $1$ part of  $B(\mathrm{Triv}^{\mathrm{Alg}}_O(L\bb{S}^{\vee n}))$.
The $K(O)$-comultiplication then yields a commutative square
\begin{center}
\[\begin{tikzcd}
	{K(O)(n) \wedge B(\mathrm{Triv}^{\mathrm{Alg}}_O(L\bb{S}^{\vee n}))} & {B(\mathrm{Triv}^{\mathrm{Alg}}_O(L\bb{S}^{\vee n}))^{\wedge n}} \\
	{K(O)(n) \wedge B(O)(n)} & {L\Sigma^\infty_+ \Sigma_n}
	\arrow[from=1-1, to=1-2]
	\arrow[hook', from=2-1, to=1-1]
	\arrow[from=2-1, to=2-2]
	\arrow[hook, from=2-2, to=1-2]
\end{tikzcd}\]
\end{center}
The assertion that our two maps agree is thus equivalent to the assertion that the $\Sigma_n$-equivariant map
\[K(O)(n) \wedge B(O)(n) \rightarrow L \Sigma^\infty_+ \Sigma_n\]
is adjoint to the standard identification 
\[K(O)(n) \simeq  L\s^{B\Sigma_n}(B(O)(n), L\Sigma^\infty_+ \Sigma_n).\]
This can be checked directly since
$B(\mathrm{Triv}^{\mathrm{Alg}}_O(L\bb{S}^{\vee n}))$
is the cofree, conilpotent $K(O)$-coalgebra \cite[3.3.6]{FrancisGaitsgory2012}.
\end{proof}

    If $O\in \mathrm{Operad}(L\s)$ is a reduced levelwise dualizable operad, then our construction of the Goodwillie derivatives is equivalent to a functor
    \[\partial_\ast : \mathrm{Fun}^\omega(\mathrm{Alg}_O,L\s) \rightarrow \mathrm{RMod}_O\]
since we have  \[K(\mathrm{CoEnd}(\Sigma^\infty_C))\simeq K(K(O)) \simeq O.\] 
We implicitly make such identifications throughout the remainder of this section.

\begin{lem}\label{lem: representables in alg}
    Let $O\in \mathrm{Operad}(L\s)$ be a reduced levelwise dualizable operad. If $A\in \mathrm{Alg}_O$ then there is an equivalence of right modules
    \[\partial^\ast L\Sigma^\infty R_A \simeq \mathrm{Forget}_{K(O)}^\mathrm{RMod}(B(A)). \]
    If $A$ is additionally compact, then there is an equivalence of right modules
    \[\partial_\ast L \Sigma^\infty R_A \simeq K^{-1}(\mathrm{Forget}_{K(O)}^\mathrm{RMod}(B(A))).\]
\end{lem}
\begin{proof}
    The follows from Theorem \ref{thm: coend in algebras} and Corollary \ref{cor: derivatives of compact representable}.

\end{proof}

\begin{prop}\label{prop: derivative of bar}
    If $O\in \mathrm{Operad}(L\s)$ is a reduced levelwise dualizable operad and $R$ is a right $O$-module, then  Goodwillie tower of $B(R,O,-)$ is equivalent to the truncation tower
    \[B(R,O,-) \rightarrow B(R^{\leq \ast},O,-).\]
There is an equivalence of right $O$-modules
    \[\partial_\ast B(R,O,-) \simeq R,\]
    and if $R$ is levelwise dualizable, then there is an equivalence of right $K(O)$-modules
    \[\partial^\ast B(R,O,-) \simeq K(R).\]
    
\end{prop}

\begin{proof}
The statement about Goodwillie towers is classical and follows from the same arguments as \cite[Theorem 11.3]{Pereira_2013}. This implies the claim about derivatives as symmetric sequences.

In the case $R = \mathrm{Free}^{\mathrm{RMod}}_O L \Sigma^\infty_+ \Sigma_n$, this is easy to verify either claim on the level of right modules by computing the Koszul dual derivatives to be trivial and exploiting the Koszul duality between derivatives and Koszul derivatives when the latter is levelwise finite.

This implies the assignment
\[R \rightarrow \partial_\ast B(R,O,-) \in \mathrm{RMod}_O\]
preserves colimits and agrees with the identity on the full subcategory generated by the $\mathrm{Free}^{\mathrm{RMod}}_O(L \Sigma^\infty_+ \Sigma_n)$, and hence the result because these free modules generate $\mathrm{RMod}_O$ under colimits. The claim about Koszul dual derivatives for general levelwise dualizable $R$ follows from Koszul duality.

\end{proof}
Consider the forgetful functor
\[\mathrm{Forget}_O:\mathrm{Alg}_O \rightarrow L\s\]
Since the underlying spectrum of an $O$-algebra tautologically has a natural $O$-algebra structure, there is an induced map of operads
 \[O \rightarrow \mathrm{End}(\mathrm{Forget}_O).\] 
Unlike the map
\[K(O) \xrightarrow{\simeq} \mathrm{CoEnd}(\mathrm{Indecom}_O)\]
There are obstructions to this map being an equivalence, see the proof of Proposition \ref{prop: operad is endomorphism of forget locally}. However, as a consequence of the identification 
\[\mathrm{Forget}_O \simeq B(O,O,-)\]
it is the case that this map is an equivalence after differentiating, since as a right $O$-module $\mathrm{End}(O)\simeq O$ \cite{envBMT}.
\begin{cor}\label{cor: derivatives of forgetful recover o}
    For $O \in \mathrm{Operad}(L\s)$, the composite \[O \rightarrow \mathrm{End}(\mathrm{Forget}_O) \xrightarrow{\partial_\ast} \mathrm{End}(\partial_\ast \mathrm{Forget}_O)\]  is an equivalence of operads.
\end{cor}

The bar construction $B(R,O,A)$ is sometimes referred to as \textit{generalized factorization homology} \cite{generalizedPoincare} and denoted $\int_R A$. Dually, given an $O$-coalgebra $D$ and a right $O$-module $R$, one can define the \textit{generalized factorization cohomology},
\[\int^R D:=\mathrm{RMod}_O(R,\mathrm{Forget}_O^{\mathrm{RMod}}(D)).\]
In the special case $R=1$, the factorization cohomology is denoted $\mathrm{Prim}_O(D)$, the primitives of $D$. Factorization cohomology obtains a filtration which arises from the truncation filtration of the right module mapping spectrum. The following is a special case of \cite[Main Theorem]{generalizedPoincare}.

\begin{prop}\label{prop: generalized factorization cohomology}
    Suppose $O \in \mathrm{Operad}$ is reduced. If $O$ is $\Sigma$-finite, then given  $A \in \mathrm{Alg}_O$ there is an equivalence of filtered spectra
    \[P_\infty \mathrm{Id}_{\mathrm{Alg}_O} (A) \simeq \mathrm{Prim}_{K(O)}(B(A)).\]

    More generally, if $O$ is levelwise dualizable and $R$ is $\Sigma$-finite, then there is an equivalence of filtered spectra
    \[P_\infty \int_R (-)(A) \simeq \int^{K(R)} B(A).\]

In ultimate generality, if $O$ and $R$ are levelwise dualizable then such an equivalence holds whenever the Tate construction
    $(R(n) \wedge B(A)^{\wedge n})^{t\Sigma_n} $ vanishes.
\end{prop}

\begin{proof}
    We prove the last statement; the first is implied by the last when $R=1$ because the underlying spectrum of $P_n \mathrm{Id}_{\mathrm{Alg}_O}(A)$ agrees with $P_n \mathrm{Forget}_O(A)$.
    
    By Lemma \ref{lem: representables in alg} and Proposition \ref{prop: derivative of bar}, we have the equivalence
    \[P_\infty^\mathrm{fake}(\int_R)(A) \simeq \int^{K(R)} B(A).\]
    By Proposition \ref{prop: fake tower koszul dual derivatives}, the Tate vanishing hypothesis then guarantees the filtered map
    \[P_\infty (\int_R) (A) \rightarrow \int^{K(R)}B(A)\]
    is an equivalence.
\end{proof}

\subsection{Poincaré/Koszul duality of $E_d$-algebras}\label{section: poincare}

In this brief section, we specify our study of $O$-algebras to $O=\Sigma^\infty_+E_d$, the reduced little $d$-disks operad, in order to give a short proof of the Poincaré/Koszul duality formula of \cite{ayalafrancispoincare} in the setting of framed manifolds and $E_d$-algebras in spectra. The statements extend in an obvious way to $E_d$-algebras in $L\s$. 

 Associated to a framed $d$-manifold $M$ is the right $E_d$-module of disks in $M$ called $E_M$ which has the homotopy type of the symmetric sequence of the ordered configuration spaces of points in $M$. We will assume $M$ is tame meaning that $M$ is abstractly homeomorphic to the interior of a compact topological manifold with, possibly empty, boundary. Given $A \in \mathrm{Alg}_{\Sigma^\infty_+ E_d}$ one can construct the \textit{factorization homology} of $A$ over $M$ as the bar construction
\[\int_M A:= B(\Sigma^\infty_+ E_M, \Sigma^\infty_+ E_d, A).\]

Dually, if $D \in \mathrm{Coalg}_{\Sigma^\infty_+ E_d}$, we can define the factorization cohomology of $D$ over $M$ as
\[\int^M D := \mathrm{RMod}_{\Sigma^\infty_+ E_d}(\Sigma^\infty_+ E_M,\mathrm{Forget}_{\Sigma^\infty_+ E_d}^\mathrm{RMod}(D)).\]
The Weiss filtration \cite[Section 2.1.1]{ayalafrancispoincare} of factorization cohomology is the (always convergent) truncation filtration of the right module mapping spectrum.

It was established in \cite[Theorem 1.2]{malinonepoint} that under the Koszul self duality of $E_d$ \cite[Theorem 1.1]{chingsalvatore}
\[\Sigma^\infty_+ E_d \simeq s_{d} K(\Sigma^\infty_+ E_d),\]
$E_M$ enjoys a version of Koszul self duality with compact support
\[\Sigma^\infty_+ E_M \simeq s_{(d,d)} K(\Sigma^\infty E_{M^+}).\]
Here $s_d$ denotes the $d$-fold operadic suspension of $O$ which can be defined as the pointwise smash product $O \wedge \mathrm{CoEnd}(\bb{S}^n)$. Similarly, $s_{(d,d)}$ denotes the external $d$-fold operadic suspension of a right module followed by the levelwise, categorical $d$-fold suspension. The right module $E_{M^+}$ is the $E_d$-module of disks in the one-point compactification of $M$, see \cite[Definition 8.5]{malinonepoint}. One can define factorization (co)homology over $M^+$ using the same formulas as before. The following was proven by Ayala-Francis \cite{ayalafrancispoincare}.

\begin{thm}[Poincaré/Koszul duality for $E_d$-algebras]\label{thm:poincare duality}
Let $M$ be a framed, tame $d$-manifold and $A\in \mathrm{Alg}_{\Sigma^\infty_+ E_d}$. The Goodwillie filtration and the Weiss filtration are equivalent under bar-cobar duality, and so yield an equivalence of filtered objects
\[P_\infty(\int_M) (A) \xrightarrow{\simeq} \int^{M^{+}} \Sigma^d B(A). \]
\end{thm}

\begin{proof}
Observe that $\Sigma^\infty_+ E_M$ is $\Sigma$-finite since $M$ is the interior of a compact manifold, and so its configuration spaces are $\Sigma$-finite. We deduce
  \begin{equation*} 
\begin{split}
P_\infty(\int_M)(A) &\xrightarrow{\simeq} P_\infty^\mathrm{fake}(\int_M)(A) \quad (\ref{prop: generalized factorization cohomology})\\
 & \simeq \mathrm{RMod}_{K(\Sigma^\infty_+ E_d)}(K(\Sigma^\infty_+ E_M),B(A))\quad (\ref{prop: fake no finiteness}, \ref{prop: derivative of bar})\\
 & \simeq \mathrm{RMod}_{\Sigma^\infty_+ E_d}(\Sigma^\infty E_{M^+},\Sigma^d B(A))\\
 & := \int^{M^+} \Sigma^d B(A).
\end{split}
\end{equation*}
\end{proof}

The statement of Theorem \ref{thm:poincare duality} departs slightly from Ayala-Francis because we use nonunital rather than augmented algebras and our models of bar-cobar duality only agree after $n$-fold (de)suspension.

Ayala-Francis observed \cite[pg. 844]{ayalafrancispoincare} that the right hand side can be identified as the embedding calculus filtration of the Weiss sheaf associated to the $E_d$-coalgebra $\Sigma^d B(A)$, and as such, demonstrated the first instance of \textit{Goodwillie-Weiss} duality.

 The category of framed $d$-manifolds with a single singularity $\mathrm{ZMfld}^\mathrm{fr}_d$ supports an involution $\neg$ which when restricted to the category of disks with a disjoint singularity $\mathrm{ZDisk}^\mathrm{fr}_d$ witnesses the Koszul self duality of $E_d$. This involution is at the heart of Poincaré/Koszul duality. Of course, $\T$ or more generally $C$ does not admit such an involution. However, the study of Goodwillie calculus via the fake Goodwillie tower can be understood as an abstract form of Goodwillie-Weiss duality. Since a general category $C$ does not admit such an involution, one ``embeds'' the objects of $C$ into the category of right $\mathrm{CoEnd}(\Sigma^\infty_C)$-modules and studies their duality with right $K(\mathrm{CoEnd}(\Sigma^\infty_C))$-modules. In forthcoming work \cite{dualityinmfld}, we use this idea to extend Ayala-Francis' Poincaré/Koszul duality from factorization homology of singular manifolds to arbitrary pointed functors $\mathrm{Fun}_\ast(\mathrm{ZMfld}_n^\mathrm{fr},\s)$.

\subsection{Lie algebra models of $v_h$-periodic spaces}\label{section:lie algebra models}
In this section, we describe a construction of the equivalence between $v_h$-periodic spaces $\T_{v_h}$ and $T(h)$-local spectral Lie algebras $\mathrm{Alg}_{\mathrm{lie}_{T(h)}}$. The insights which appear in our construction should be attributed to \cite{heutsannals} where they first appeared in a slightly different form. We extend the construction to include Goodwillie towers of functors $\T_{v_h} \rightarrow \s_{T(h)}$, a result which we expect experts are aware of and which was initially predicted by Arone--Ching. 
 
We recall the strategy of \cite{heutsannals}, which depends crucially on the existence of the Bousfield-Kuhn functor $\Phi:\T_{v_h} \rightarrow \s_{T(h)}$ and its left adjoint $\Theta$\footnote{This $\Theta$ is unrelated to the right adjoint of Goodwillie derivatives.}, constructed in \cite{telescopichomotopybous}: 
\begin{enumerate}
    \item Demonstrate that the $(\Theta, \Phi)$-adjunction is monadic \cite{Eldred_Heuts_Mathew_Meier_2019}, and therefore $\T_{v_h}$ is the category of algebras over some monad.
    \item Demonstrate this monad arises as the cobar construction of the comonad $\Sigma^\infty_{T(h)} \Omega^\infty_{T(h)}$ \cite[Theorem 4.21]{heutsannals}.
    \item Demonstrate that $\Sigma^\infty_{T(h)} \Omega^\infty_{T(h)}$ is coanalytic, in the sense that it converges in dual calculus \cite[Theorem 4.16]{heutsannals}. 
    
    \item Use the coanalyticity of $\Sigma^\infty_{T(h)} \Omega^\infty_{T(h)}$  to show this comonad is equivalent to the comonad associated to conilpotent commutative coalgebras.\footnote{This assertion is implicit in \cite{heutsannals}; we expect it should be a relatively straightforward consequence of the fact that $\Sigma^\infty_{T(n)}$ takes values in commutative coalgebras in combination with the coanalyticity of  $\Sigma^\infty_{T(h)} \Omega^\infty_{T(h)}$.}

    \item Conclude by bar-cobar duality that $(\Theta, \Phi)$ is the free-forgetful adjunction for $\mathrm{lie}_{T(h)}$-algebras.

\end{enumerate} 

Inseparable from this story is the fact that  $\s_{T(h)}$ is $1$-semiadditive, and so for any $X\in \s_{T(h)}^{B\Sigma_n}$ the Tate construction vanishes \cite{kuhn_2004}
\[X^{t\Sigma_n} \simeq \ast.\]

Our description of the equivalence $\T_{v_h} \simeq \mathrm{Alg}_{\mathrm{lie}_{T(h)}}$ will be made by directly constructing an operad action on $\Phi(X)$. Recall from Corollary \ref{cor: derivatives of forgetful recover o} and Theorem \ref{thm: product rule} that for a reduced operad $O$ in $L\s$
\[\mathrm{Forget}_O: \mathrm{Alg}_O \rightarrow L\s\]
can be used to recover $O$ by considering $\partial_\ast(\mathrm{End}(\mathrm{Forget}_O))$. When $L\s= \s_{T(h)}$ it turns out that the passage to derivatives is unnecessary.

 \begin{prop} \label{prop: operad is endomorphism of forget locally}
     Let $O \in \mathrm{Operad}(\s_{T(h)})$ be a reduced operad, not necessarily levelwise dualizable. The map of operads \[O \rightarrow \mathrm{End}(\mathrm{Forget}_O)\]
     induced by the action of $O$ on the underlying spectrum of $A \in \mathrm{Alg}_O$ is an equivalence.
 \end{prop}

\begin{proof}
    Using the resolutions associated to the free-forgetful adjunction of $O$-algebras, an extra degeneracy argument implies we have retraction
    \[\mathrm{RMod}_{\mathrm{Forget}_O \circ \mathrm{Free}_O}((\mathrm{Forget}_O)^{\wedge n} \circ \mathrm{Free}_O,\mathrm{Forget}_O \circ \mathrm{Free}_O) \rightarrow \nat(\mathrm{Forget}_O^{\wedge n}, \mathrm{Forget}_O)\]
The lefthand side is equivalently
\[\mathrm{RMod}_{\mathrm{Forget}_O \circ \mathrm{Free}_O}(\mathrm{Free}^{\mathrm{RMod}}_{\mathrm{Forget}_O \circ \mathrm{Free}_O}(\mathrm{Id}_{\s_{T(h)}}^{\wedge n}),\mathrm{Forget}_O \circ \mathrm{Free}_O), \]
and so this simplifies to $\nat(\mathrm{Id}_{\s_{T(h)}}^{\wedge n}, \mathrm{Forget}_O \circ \mathrm{Free}_O)$

Since $\mathrm{Forget}_O \circ \mathrm{Free}_O$ is a coanalytic functor \cite[Definition 4.3]{heutsannals}, we may compute this natural transformation spectrum using the fact that $T(h)$-locally a natural transformation from a cohomogeneous functor to a wedge of increasingly cohomogeneous functors factors through a finite wedge \cite[Lemma B.1]{heutsannals}; we pick some large $N>n$ and compute \[\nat(\mathrm{Id}_{\s_{T(h)}}^{\wedge n}, \bigvee_{i \leq N} O(i) \wedge_{h\Sigma_i }\mathrm{Id}_{\s_{T(h)}}^{\wedge i}).\]
By the finiteness of the wedge, this is equivalent to
\[\prod_{i \leq N} \nat(\mathrm{Id}_{\s_{T(h)}}^{\wedge n}, O(i) \wedge_{h\Sigma_i} \mathrm{Id}_{\s_{T(h)}}^{\wedge i}). \]
This product is concentrated in a single degree because $T(h)$-locally natural transformations between homogeneous functors of different degrees are contractible, by using dual calculus combined with Tate vanishing. In the single degree it is nonvanishing, it is $O(n)$ by the classification of homogeneous functors.

The conclusion of this argument is that $O(n)$ retracts onto $\nat(\mathrm{Forget}_O^{\wedge n}, \mathrm{Forget}_O)$, but chasing through involved maps one finds that composition in the other order is also the identity.
\end{proof}

\begin{cor}\label{cor: endomorphism of forget is endomorphism of derivatives locally}
    For a reduced, levelwise dualizable operad $O \in \mathrm{Operad}(\s_{T(h)})$, the differentiation map \[\mathrm{End}( \mathrm{Forget}_O) \rightarrow \mathrm{End}(\partial_\ast \mathrm{Forget}_O)\] is an equivalence.
\end{cor}
\begin{proof}
    This follows from Corollary \ref{cor: derivatives of forgetful recover o} and the above result.
\end{proof}

In the $v_h$-periodic setting, $\Sigma^\infty_{T(h)}$ still has a commutative coalgebra structure \cite[Remark 5.11]{heutsannals}, and the analog of Proposition \ref{prop:coend in spaces} asserts that the classifying map witnesses $\mathrm{CoEnd}(\Sigma^\infty_{T(h)})$ as the $T(h)$-local commutative operad $\mathrm{com}_{T(h)}$. Hence, the Goodwillie derivatives assemble into a functor
\[\partial_\ast:\mathrm{Fun}(\T_{v_h},\s_{T(h)}) \rightarrow \mathrm{RMod}_{\mathrm{lie}_{T(h)}}.\]

\begin{thm}\label{thm: classification for vh periodic spaces}
There is an equivalence of operads in $\s_{T(h)}$
\[\mathrm{lie}_{T(h)} \simeq \mathrm{End}(\Phi)\]
under which Heuts's spectral Lie algebra models are given by the tautological action
\begin{center}
   \[\begin{tikzcd}
	{\T_{v_h}} & {\mathrm{Alg}_{\mathrm{End}(\Phi)}} \\
	X & {\Phi(X)}
	\arrow[from=1-1, to=1-2]
	\arrow[maps to, from=2-1, to=2-2]
	\arrow["{\mathrm{End}(\Phi)}"{pos=0.4}, shift left=4, from=2-2, to=2-2, loop, in=325, out=35, distance=10mm]
\end{tikzcd}\]
\end{center} 
The assignment
\[R \mapsto B(R,\mathrm{lie},\Phi(-))\]
induces an equivalence of categories 
    \[\mathrm{RMod}_{\mathrm{lie}_{T(h)}}^{<\infty}\simeq \mathrm{Poly}^\omega(\T_{v_h},\s_{T(h)}) .\]
    In general, for $F \in \mathrm{Fun}^\omega(\T_{v_h},\s_{T(h)})$ the Goodwillie tower takes the form
   \[F(A) \rightarrow P_n(F)=B(\partial_{\leq n} F ,\mathrm{lie}_{T(h)},\Phi(A)).\] 
\end{thm}

\begin{proof}
By the equivalence of $T(h)$-local operads and coanalytic monads \cite[Proposition 4.8]{heutsannals}, the monadicity \cite{Eldred_Heuts_Mathew_Meier_2019} and coanalyticity \cite[Theorem 4.15]{heutsannals} of of $\Phi\circ \Theta$  imply that there is an abstract operad structure on $\partial_\ast \Phi$ such that $\Phi$ is the forgetful functor for algebras over the operad $\partial_\ast \Phi$. Proposition \ref{prop: operad is endomorphism of forget locally} and Corollary \ref{cor: endomorphism of forget is endomorphism of derivatives locally}, assert we can recover this operad structure by computing the endomorphism operad in the category $(\mathrm{RMod}_{\mathrm{lie}_{T(h)}},\circledast)$ given by $\mathrm{End}(\partial_\ast \Phi)$. By \cite[Theorem 6.6]{behrensmalin} \footnote{This is proven using Blans-Blom's right module structure, but it isn't difficult to supply an additional ad hoc argument which adjusts to the models we use here.}, there is an equivalence of right modules $\partial_\ast \Phi \simeq \mathrm{lie}_{T(h)},$
and so this endomorphism operad recovers $\mathrm{lie}_{T(h)}$, finishing the characterization of the equivalence.

The act of taking derivatives induces an equivalence
\[\partial_\ast: \mathrm{Poly}(\T_{v_h}^\omega,\s_{T(h)}) \rightarrow \mathrm{RMod}^{<\infty}_{\mathrm{lie}_{T(h)}},\]
by Proposition \ref{prop: unit and counit} combined with $T(h)$-local Tate vanishing.

By Section \ref{section: goodwillie calculus algebras} the inverse of $\partial_\ast$ when restricted to levelwise dualizable right modules is given by the generalized factorization cohomology $\int^{K(R)}B(\Phi(-)).$
If we apply Tate vanishing and Proposition \ref{prop: generalized factorization cohomology}, this is equivalent to $B(R,O,\Phi(-))$ when restricted to compact objects of $\T_{v_h}$. This formula commutes with colimits in $R$, and so must be true for a general right module $R$ since $P_n$ commutes with colimits.

\end{proof}

\section{Classification of Goodwillie towers} \label{section: classifications}

\subsection{Infinitesimally $1$-semiadditive categories and right module classifications}
Under what circumstances can operadic structures on $\partial_\ast F$ be used to reconstruct the Goodwillie tower of $F$? We first consider what categories have a classification of Goodwillie towers in terms of right modules over some operad $O$. We assume our category $C$ is compactly generated and has an identification $\s(C) \simeq L\s$ for some Bousfield localization of $\s$. We also assume $C$ is differentially dualizable.

Suppose we are provided with a lift of the derivatives to right modules over some operad $O$
\begin{center}
\[\begin{tikzcd}
	& {\mathrm{RMod}_{O}} \\
	{\mathrm{Fun}(C^\omega,L\s)} & {\mathrm{SymSeq}(L\s)}
	\arrow["{\mathrm{Forget}_O}", from=1-2, to=2-2]
	\arrow[from=2-1, to=1-2]
	\arrow["{\partial_\ast}"', from=2-1, to=2-2]
\end{tikzcd}\]
\end{center}
for some $O$ in $\mathrm{Operad}(L\s)$. 

\begin{definition}
    We say that for a differentially dualizable $C$
    \[\partial_\ast: \mathrm{Fun}(C^\omega,L\s) \rightarrow \mathrm{RMod}_O\]
classifies Goodwillie towers if the adjunction
\[
\adjunction{\partial_{\leq n}}{\mathrm{Fun}(C^{\omega},L\s)}{\mathrm{RMod}^{\leq n}_{O}}{\Theta^{\mathrm{RMod}}_{\leq n}}
\]
restricts to an equivalence of categories for all $n<\infty$
\[
\adjunction{\partial_{\leq n}}{\mathrm{Poly}_n(C^{\omega},L\s)}{\mathrm{RMod}^{\leq n}_{O}}{\Theta^{\mathrm{RMod}}_{\leq n}}.
\]
This is equivalent to asking that it restricts to an equivalence
\[
\adjunction{\partial_{\ast}}{\mathrm{Poly}(C^{\omega},L\s)}{\mathrm{RMod}^{<\infty}_{O}}{\Theta^{\mathrm{RMod}}}.
\]
\end{definition}
One might notice that we have abusively referred to both $\partial_\ast$ and its adjoint $\Theta^{\mathrm{RMod}}$ using the same names as in Section \ref{section: right module construction}. This is because, under mild conditions, right module structures on derivatives are essentially unique.

\begin{definition}\label{definition: quasiequivalence}
    A quasi-equivalence of $O,P \in \mathrm{Operad}(L\s)$ is an equivalence of enriched categories
    \[\mathrm{Env}(O)\simeq \mathrm{Env}(P). \]
\end{definition}

If $O$ is quasi-equivalent to $P$, their categories of right modules are canonically equivalent, and the obstruction to upgrading a quasi-equivalence to an equivalence is whether or not the equivalence of right module categories can be made symmetric monoidal.

\begin{prop}\label{prop: right module classification quasi}
    If 
    \[\partial_\ast:\mathrm{Fun}(C^\omega,L\s) \rightarrow \mathrm{RMod}_O\]
    classifies Goodwillie towers, then there is a quasi-equivalence of operads \[O \simeq K(\mathrm{CoEnd}(\Sigma^\infty_C))\] under which the right module structure on $\partial_\ast$ agrees with Definition \ref{definition: right module derivatives}. If $\partial_\ast$ can be made symmetric monoidal, the quasi-equivalence may be lifted to an equivalence of operads.
\end{prop}
\begin{proof}
    From Proposition \ref{prop: dual derivatives of homogeneous}, one may identify $\mathrm{Env}(\mathrm{CoEnd}(\Sigma^\infty_C) )$ with the opposite of the full subcategory of polynomial functors on $(\langle \Sigma^\infty_C )^{\wedge n}\rangle_{n \in \mathbb{N}}$.
     
There is then a quasi-equivalence of operads
    \[\mathrm{CoEnd}(\Sigma^\infty_C) \rightarrow K(O)\]
constructed by applying $\partial_\ast$ to this subcategory and using the assumption that the right $O$-module structure classifies Goodwillie towers. If $\partial_\ast$ satisfies a product rule, this is symmetric monoidal and so induces an equivalence of operads.

In the same way, there is a compatible map
    \[\partial^\ast F \rightarrow K(\partial_\ast F)\]
which is an equivalence by the same reasoning. The Koszul duality between derivatives and Koszul dual derivatives completes the result, with the usual caveat that one should first restrict to representables of compacts, which satisfy the necessary levelwise dualizability, and then Kan extend.
\end{proof}

\begin{definition}
 A stable category $C$ is $1$-semiadditive if for all finite groups $G$ the Tate construction
\[X^{tG}:= \mathrm{cofiber}(X_{G} \rightarrow X^{G})\]
is contractible.
\end{definition}

\begin{lem} \label{lem: semiadditive sigma norm}
   The category $C$ is $1$-semiadditive, if and only if, all Tate constructions vanish for all symmetric groups $\Sigma_n$.

\end{lem}
\begin{proof}
The forward implication is immediate from the definition of $1$-semiadditivity. 

For the latter implication, fix an embedding $G \hookrightarrow \Sigma_n$. By the equivalence between induction and coinduction for finite groups \cite{Fausk2003}, we have an equivalence of cofiber sequences
\begin{center}
\[\begin{tikzcd}
	{(\mathrm{Ind}^{\Sigma_n}_GX)_{h\Sigma_n}} & {(\mathrm{Coind}^{\Sigma_n}_GX)^{h\Sigma_n}} & {(\mathrm{Ind}^{\Sigma_n}_GX)^{t\Sigma_n}} \\
	{X_{hG}} & {X^{hG}} & {X^{tG}}
	\arrow[from=1-1, to=1-2]
	\arrow[from=1-2, to=1-3]
	\arrow["\simeq", from=2-1, to=1-1]
	\arrow[from=2-1, to=2-2]
	\arrow["\simeq", from=2-2, to=1-2]
	\arrow[from=2-2, to=2-3]
	\arrow["\simeq", from=2-3, to=1-3]
\end{tikzcd}\]
\end{center}
which establishes the result.
\end{proof}

\begin{definition}
    A category $C$ is infinitesimally $1$-semiadditive if $\s(C)$ is $1$-semiadditive.
\end{definition}

\begin{lem}\label{lem: stable tate}
The following are equivalent for $L\s$ a compactly generated Bousfield localization of $\s$:
    \begin{itemize}
    \item The category $L\s$ is $1$-semiadditive.
    
        \item  For all $n$ and $X \in L\s^{B\Sigma_n}$, the Tate construction \[(X \wedge \mathrm{Id}_{L\s}^{\wedge n})^{t\Sigma_n}: L\s^\omega \rightarrow L\s\] is contractible. 
        
    \end{itemize}
\end{lem}
\begin{proof}
The forward implication is immediate by definition. In the case $L\bb{S} \in L\s$ is compact, the backwards direction is also immediate since for any $X$ we may evaluate the Tate construction on $L\bb{S}$ to see Tate vanishing for all symmetric group actions. If $L\bb{S}$ is not compact, we instead approach this using Goodwillie calculus.

We first observe that by Proposition \ref{prop: unit and counit} the symmetric sequence of derivatives classifies Goodwillie towers
    \[\partial_\ast: \mathrm{Poly}^\omega(L\s,L\s) \xrightarrow{\simeq} \mathrm{SymSeq}^{< \infty}(L\s).\]
There is always a commuting triangle
\begin{center}
\[\begin{tikzcd}
	& {\mathrm{SymSeq}(L\s)} \\
	{\mathrm{Fun}^\omega(L\s,L\s)} & {\mathrm{Fun}^\omega(\s,L\s)}
	\arrow["{\partial_\ast}", from=2-1, to=1-2]
	\arrow["{(-)\circ L}", from=2-1, to=2-2]
	\arrow["{\partial_\ast}", from=2-2, to=1-2]
\end{tikzcd}\]
\end{center}
where the horizontal map is a full and faithful embedding. As a consequence of the classification of finitary homogeneous functors, the essential image of $\mathrm{Poly}^\omega(L\s,L\s) \subset \mathrm{Fun}^\omega(\s,L\s)$ is $\mathrm{Poly}^\omega(\s,L\s)$. Therefore, we have a commuting triangle
\begin{center}
\[\begin{tikzcd}
	& {\mathrm{SymSeq}^{<\infty}(L\s)} \\
	{\mathrm{Poly}^\omega(L\s,L\s)} & {\mathrm{Poly}^\omega(\s,L\s)}
	\arrow["\simeq", from=2-1, to=1-2]
	\arrow["\simeq", from=2-1, to=2-2]
	\arrow["{\partial_\ast}", from=2-2, to=1-2]
\end{tikzcd}\]
\end{center}
We now describe the inverse of 
\[\partial_\ast:\mathrm{Poly}^\omega(\s,L\s) \xrightarrow{\simeq} \mathrm{SymSeq}^{< \infty}(L\s).\]
We first enlarge to the category of all finitary functors and see the right adjoint of
\[\partial_\ast:\mathrm{Fun}(\s^\omega,L\s) \rightarrow \mathrm{SymSeq}(L\s)\]
is given by the formula
\[R \mapsto (c \in \s^\omega \mapsto \mathrm{SymSeq}(L\s)(\partial_\ast L\Sigma^\infty R_c,R)).\]
An application of the (weak) stable chain rule shows
\[\partial_\ast L\Sigma^\infty R_c \simeq (L c^\vee)^{\wedge \ast}.\]
This adjunction restricts to an adjunction between $\mathrm{Poly}(\s^\omega,L\s)$ and $\mathrm{SymSeq}^{<\infty}(L\s)$, and so the inverse of the Goodwillie derivatives on $\mathrm{Poly}^\omega(\s,L\s)$ is given by
\[R \mapsto (c \in \s^\omega \mapsto \mathrm{SymSeq}(L\s)((Lc^\vee)^{\wedge \ast},R)).\]
This restricts to an equivalence
\[\partial_{\leq n}:\mathrm{Poly}(\s^\omega,L\s) \xrightarrow{\simeq} \mathrm{SymSeq}^{\leq n}(L\s),\]
with inverse given by the same formula
\[R \mapsto (c \in \s^\omega \mapsto \mathrm{SymSeq}^{\leq n}(L\s)((Lc^\vee)^{\wedge \ast},R)).\]
As in the analysis of the lower fake Goodwillie tower, we may pass to the fibers of the filtered unit to see that for any $G$ in $\mathrm{Poly}(\s^\omega,L\s)$ there is an equivalence in $\mathrm{Poly}(\s^\omega,L\s)$
\[\partial_n G \wedge_{h\Sigma_n} L(-)^{\wedge n} \simeq \partial_n G \wedge^{h\Sigma_n} L(-)^{\wedge n}. \]
 Since $\bb{S} \in \s$ is compact, we may evaluate the above formula on $\bb{S}$ to deduce that all norms in $L\s^{B \Sigma_n}$ are equivalences. By Lemma \ref{lem: semiadditive sigma norm}, this implies $L\s$ is $1$-semiadditive.

\end{proof}

\begin{thm}\label{thm: inf semiadd classification}
    The following are equivalent:  
    \begin{enumerate}
        \item $C$ is infinitesimally $1$-semiadditive. 
        \item For all $n$, $X \in L\s^{B\Sigma_n}$, and $c \in C^\omega$
        \[(X \wedge \Sigma^\infty_C c ^{\wedge n})^{t\Sigma_n} \simeq \ast.\]
        \item There is a symmetric monoidal equivalence $(\mathrm{Poly}^\omega(C,L\s),\wedge) \simeq (\mathrm{RMod}^{< \infty}_O,\circledast)$ for some operad $O$.  
        \item The derivatives $\partial_\ast:\mathrm{Fun}^\omega(C,L\s) \rightarrow \mathrm{RMod}_{K(\mathrm{CoEnd}(\Sigma^\infty_C))}$ classify Goodwillie towers.

    \end{enumerate}
\end{thm}

\begin{proof}\noindent
\begin{itemize}
    \item $(1) \implies (4)$: This follows from Proposition \ref{prop: unit and counit} and Tate vanishing for $L\s$.
    \item $(1) \implies (2)$: This is immediate.
    \item $(2) \implies (1)$: By Lemma \ref{lem: stable tate}, it suffices to show that for $Z \in L\s^\omega$ 
    \[(X \wedge Z^{\wedge n})^{t\Sigma_n} \simeq \ast.\]
    This can be checked on derivatives since the Tate construction is the cofiber of $n$-excisive functors. However, precomposition with $\Sigma^\infty_C$ does not change derivatives, and so its derivatives must be $0$.
    \item $(4) \implies (3)$: This follows from the product rule of Theorem \ref{thm: product rule}.
    \item $(3)\implies(4)$: We defer this to Proposition \ref{prop: abstract rmod equiv}.
    \item $(4) \implies (2)$: We prove equivalently that the norm
 \[X \wedge_{h\Sigma_n} (\Sigma^\infty_C(-))^{\wedge n} \rightarrow X \wedge^{h\Sigma_n} (\Sigma^\infty_C(-))^{\wedge n}\]
 is an equivalence when restricted to $C^\omega$.
By assumption, the $n$-truncated right module structure on $\partial_{\leq n}$ classifies $n$-polynomial functors, so we deduce that the comparison map 
\[P_n (F) \rightarrow P^n_\mathrm{fake}(F)\]
is an equivalence of filtered objects. By Proposition \ref{prop: fake no finiteness}, the map on associated graded is the norm in question for $X=\partial_n F$. This implies the norm is an equivalence for $X=\partial_n F$. Since any $X$ occurs as $\partial_n F$, we conclude the norm \[X \wedge_{h\Sigma_n} \mathrm{Id}_{L\s^\omega}^{\wedge n} \rightarrow X \wedge^{h\Sigma_n} \mathrm{Id}_{L\s^\omega}^{\wedge n} \] 
is an equivalence. 
\end{itemize}
\end{proof}

\subsection{Differentially $1$-semiadditive categories and divided power right module classifications}\label{section: categories with divided powers}
The appearance of Tate vanishing hypotheses in operadic results often suggests that a more refined algebraic structure encoding the Tate data through ``divided powers'' exists and is the ``correct'' object to study, in the sense that the result might generally be true for this new class of objects with no hypothesis on Tate vanishing.

We demonstrate that for right modules in Goodwillie calculus this is partially correct, but also somewhat misleading. There is a nontrivial class of categories for which divided power structures on the derivatives do classify Goodwillie towers, but such categories seem rare in practice. In particular, the category still must satisfy some amount of Tate vanishing for this classification to occur. 

We assume our category $C$ is compactly generated, has an identification $\s(C) \simeq L\s$ to some Bousfield localization of $\s$, and is differentially dualizable. Suppose we are provided with a lift of the Goodwillie derivatives to divided power right modules (Definition \ref{definition: dp rmod}).

\begin{center}
\[\begin{tikzcd}
	& {\mathrm{RMod}^\mathrm{dp}_{O}} \\
	{\mathrm{Fun}(C^\omega,L\s)} & {\mathrm{SymSeq}(L\s)}
	\arrow["{\mathrm{Forget}^\mathrm{RMod}_O}", from=1-2, to=2-2]
	\arrow[from=2-1, to=1-2]
	\arrow["{\partial_\ast}"', from=2-1, to=2-2]
\end{tikzcd}\]
\end{center}

Any such assignment automatically admits a right adjoint $\Theta^\mathrm{dp}$ which for $c\in C^\omega$ is given by
\[\Theta^\mathrm{dp}(R)(c)= \mathrm{RMod}_O^\mathrm{dp}(\partial_\ast L \Sigma^\infty R_c , R),\]
and an identical formula for $\Theta^\mathrm{dp}_{\leq n}$, the adjoint of $\partial_{\leq n}$.

\begin{definition}
    We say that
    \[\partial_\ast: \mathrm{Fun}(C^\omega,L\s)\ \rightarrow \mathrm{RMod}^\mathrm{dp}_O\]
classifies Goodwillie towers if the adjunction
\[
\adjunction{\partial_{\leq n}}{\mathrm{Fun}(C^{\omega},L\s)}{\mathrm{RMod}^\mathrm{\leq n,dp}_{O}}{\Theta^\mathrm{dp}_{\leq n}}
\]
restricts to an equivalence of categories for all $n<\infty$
\[
\adjunction{\partial_{\leq n}}{\mathrm{Poly}_n(C^{\omega},L\s)}{\mathrm{RMod}^\mathrm{\leq n,dp}_{O}}{\Theta^\mathrm{dp}_{\leq n}}.
\]
This is equivalent to asking that it restricts to an equivalence
\[
\adjunction{\partial_{\ast}}{\mathrm{Poly}(C^{\omega},L\s)}{\mathrm{RMod}^\mathrm{<\infty,dp}_{O}}{\Theta^\mathrm{dp}}.
\]
\end{definition}

If such a structure divided power structure exists, it must lift the right $K(\mathrm{CoEnd}(\Sigma^\infty_C))$-module structure constructed earlier in the following sense.

\begin{prop}
    If \[\partial_\ast: \mathrm{Fun}(C^\omega,L\s) \rightarrow \mathrm{RMod}^\mathrm{dp}_O\]
    classifies Goodwillie towers, then $O$ is quasi-equivalent to $K(\mathrm{CoEnd}(\Sigma^\infty_C))$, and the underlying right module is that of Definition \ref{definition: right module derivatives}. If 
    \[   \partial_\ast: (\mathrm{Fun}(C^\omega,L\s),\wedge) \rightarrow (\mathrm{RMod}^\mathrm{dp}_O,\circledast)\]
    can be made symmetric monoidal, then the quasi-equivalence lifts to an equivalence of operads.
\end{prop}
\begin{proof}
    Since $\partial_\ast ((\Sigma^\infty_C)^{\wedge n})$ is $\Sigma$-finite, Proposition \ref{prop: divided power mapping spectrum codomain tate vanish} shows that the forgetful map
    \[ \mathrm{RMod}^\mathrm{dp}_O(\partial_\ast F, \partial_\ast ((\Sigma^\infty_C)^{\wedge n}))\rightarrow \mathrm{RMod}_O(\partial_\ast F, \partial_\ast ((\Sigma^\infty_C)^{\wedge n}))  \]
    is an equivalence, and so the same argument as Proposition \ref{prop: right module classification quasi} applies.
\end{proof}

\begin{ex}
    The category $\s$ admits a lift of $\partial_\ast$ to divided power right modules, as divided powers over $O=1$ are just symmetric sequences. Of course, these divided powers do not classify the Goodwillie tower, as many Goodwillie towers do not split.
\end{ex}
\begin{ex}
    Any infinitesimally $1$-semiadditive category $C$  has a lift of $\partial_\ast$ to divided powers, and this lift classifies the Goodwillie tower. This is because because the $1$-semiadditivity of $L\s$ implies right modules agree with divided power right modules, and so one can apply Theorem \ref{thm: inf semiadd classification}.
\end{ex}
\begin{ex}
    The category $\T$ has canonical divided power structures on its derivatives by Proposition \ref{prop: derivatives of polynomial}. The $n$-truncation of this divided power structure is given by the formula
    \[F \mapsto B(\epsilon(P_n(F)|_{\mathrm{FinSet}_\ast})).\]
    A consequence of Theorem \ref{thm: diff semiadditive classification} is that these classify the Goodwillie towers.
\end{ex}

These three examples exhibit drastically different behaviors and show that there is a trichotomoy in Goodwillie calculus:
\begin{enumerate}
    \item Divided powers do not classify the Goodwillie towers.
    \item Divided powers classify the Goodwillie tower for a trivial reason.
    \item Divided powers classify the Goodwillie tower for an interesting reason.
\end{enumerate}

In order to address the problem of when $C$ admits a divided power right module classification, it is useful to compare to a more general classification theorem of Arone-Ching \cite[Theorem 0.2]{ACClassification}. Though stated only for $C=\T,\s$, the proof works for differentiable $C$. The classification is implemented by studying the adjunction
\[
\adjunction{\partial_\ast}{\mathrm{Fun}(C^\omega,L\s)}{\mathrm{SymSeq}(L\s)}{\Theta^\Sigma}
\]
where the right adjoint has the explicit formula
\[\Theta^\Sigma(S)(c):= \prod_{i \in \mathbb{N}} \partial_i (L\Sigma^\infty R_c) ^\vee  \wedge^{h\Sigma_i} S(i)).\]

\begin{thm}[Arone-Ching]
    Suppose $C$ is compactly generated, differentiable, and equipped with an identification $\s(C)\simeq L\s$. The following restricted adjunctions are comonadic:
    \[
\adjunction{\partial_{\leq n}}{\mathrm{Poly}_n(C^{\omega},L\s)}{\mathrm{SymSeq}^{\leq n}}{\Theta^{\Sigma}_{\leq n}},
\]
\[
\adjunction{\partial_{\ast}}{\mathrm{Poly}(C^{\omega},L\s)}{\mathrm{SymSeq}^{<\infty}}{\Theta^{\Sigma}}.
\]
    
\end{thm}

In particular, there is some abstract coalgebraic structure on symmetric sequences that is able to encode precisely the data of a Goodwillie tower.  For the rest of the section we assume our category $C$ is compactly generated, has an identification $\s(C) \simeq L\s$ for some Bousfield localization of $\s$, and is differentially dualizable.

\begin{lem}\label{lem: computation of comonad for diff sa}
 Suppose $S \in \mathrm{SymSeq}^{\leq n}(L\s)$. There is an equivalence of symmetric sequences
 \[\partial_\ast \Theta^{\Sigma} (S)\simeq \mathrm{Cofree}^{\mathrm{RMod^{dp}}}_{K(\mathrm{CoEnd})(\Sigma^\infty_C)}(S),\]
 if and only if $(\partial_{\ast} L\Sigma^\infty R_c)^\vee \wedge S$ has levelwise Tate vanishing.
\end{lem}

\begin{proof}
   We prove this by explicit computation of  $\partial_{\ast}\Theta^{\Sigma} (S)$. The boundedness of $S$ allows us to commute the finite (co)product and $\partial_{\ast}$ to see it is equivalent to
   \[\prod_{i \leq n} \partial_{\ast}( \partial_i (L\Sigma^\infty R_c) ^\vee  \wedge^{h\Sigma_i} S(i))),\]
   where we differentiate with respect to $c\in C^\omega$.  This expression admits a norm map from
   \[\prod_{i \leq n} \partial_{\ast}( \partial_i (L\Sigma^\infty R_c) ^\vee  \wedge_{h\Sigma_i} S(i))),\]
   and this is an equivalence if each $\partial_i (L\Sigma^\infty R_c) ^\vee \wedge S(i)$ has Tate vanishing.
   
By Proposition \ref{cor: derivatives of compact representable} we know
\[\partial_i (L\Sigma^\infty R_c) ^\vee \simeq B((\Sigma^\infty_C c)^{\wedge \ast},\mathrm{CoEnd}(\Sigma^\infty_C),1)(i).\]
The terms of $(\Sigma^\infty_C c)^{\wedge \ast}$ is are homogeneous functors, and so by commuting $\partial_a$ with the homotopy orbits and geometric realization one can easily compute
\[\partial_a(   B((\Sigma^\infty_C c)^{\wedge \ast},\mathrm{CoEnd}(\Sigma^\infty_C),1)(i) )\wedge_{h\Sigma_i} S(i)) \simeq (L \Sigma^\infty_+ \Sigma_a \circ B(\mathrm{CoEnd}(\Sigma^\infty_C))) (i)  \wedge_{h\Sigma_i} S(i).
 \]
 Expanding out the composition product and letting $a$ vary produces the formula found in Proposition \ref{prop: cofree divided power right modules} for the cofree divided power right $K(\mathrm{CoEnd}(\Sigma^\infty_C))$-module on $S$.

Now suppose we do not have Tate vanishing. Observe that the cofree divided power right $O$-module functor
 \[\mathrm{Cofree}^{\mathrm{RMod}^\mathrm{dp}}_{O}(S)(r) \simeq \prod_n \left[ \prod_{n \twoheadrightarrow r} L\s(O(n_1) \wedge \dots \wedge O(n_r), S(n)) \right]_{h\Sigma_n},\]
commutes with filtered colimits in $N$-truncated symmetric sequences. However, $\Theta^\Sigma$ commutes with filtered colimits in $N$-truncated symmetric sequences exactly when \[\partial_i (L\Sigma^\infty R_c) ^\vee \wedge S(i),\] has Tate vanishing. Since $\partial_\ast$ commutes with filtered colimits and detects the $0$-polynomial, $\partial_\ast \circ \Theta^\Sigma$ will commute with the same filtered colimits as the product of all 
\[\partial_i (L\Sigma^\infty R_c) ^\vee \wedge_{h\Sigma_i} S(i).\]
Hence, if any of the above terms does not have Tate vanishing $\partial_\ast \circ \Theta^\Sigma$ will not commute with such filtered colimits, and so cannot be equivalent to \[\mathrm{Cofree}^{\mathrm{RMod}^\mathrm{dp}}_{K(\mathrm{CoEnd}(\Sigma^\infty_C))}(-).\]
\end{proof}

 We say $X \in L\s^{BG}$ has \textit{ideal Tate vanishing} if $X \wedge Z$ has Tate vanishing for any $Z \in  L\s^{BG}$.

\begin{definition}
   A category $C$ is differentially $1$-semiadditive if for all $c \in C^\omega$ and $n \in \mathbb{N}$, the $\Sigma_n$-spectrum \[\partial_n(L\Sigma^\infty R_c)^\vee\] has ideal Tate vanishing.
\end{definition}



    
    

\begin{ex}
    The category $\s$ is not differentially $1$-semiadditive. There is an equivalence \[\partial_\ast \Sigma^\infty R_{\bb{S}} \simeq \mathrm{com},\] and these Tate constructions are the ``extended powers'' and are generally nontrivial.
\end{ex}

\begin{ex}
    If $C$ is infinitesimally $1$-semiadditive, it is differentially $1$-semiadditive because all Tate constructions in $L\s^{B\Sigma_n}$ vanish. 
\end{ex}

\begin{ex}
    The category $\T$ is differentially $1$-semiadditive by Example \ref{ex: derivatives of representable spaces}.
\end{ex}

\begin{remark}\label{remark: dual norms}
If $C$ is differentially $1$-semiadditive, then for $X,Y \in C^\omega$ the Goodwillie tower for $\Sigma^\infty R_X(Y)$ is computed by either of the fake Goodwillie towers. This is because if $A,B \in L\s^{BG}$ are dualizable, then Tate vanishing for $A \wedge B$ occurs, if and only if Tate vanishing for $A^\vee \wedge B^\vee$ occurs. Hence, we can apply Proposition \ref{prop: fake tower koszul dual derivatives} or Proposition \ref{prop: fake no finiteness} since all the spectra involved in the layers are dualizable.
\end{remark}

\begin{prop}\label{prop: lifts of derivatives to dp}
    If $C$ is a differentially $1$-semiadditive category, then the space of lifts
\begin{center}
\[\begin{tikzcd}
	& {\mathrm{RMod}^\mathrm{dp}_{K(\mathrm{CoEnd}(\Sigma^\infty_C))}} \\
	{\mathrm{Fun}(C^\omega,\s)} & {{\mathrm{RMod}_{K(\mathrm{CoEnd}(\Sigma^\infty_C))}}}
	\arrow["{\mathrm{Forget}^\mathrm{RMod}_O}", from=1-2, to=2-2]
	\arrow[from=2-1, to=1-2]
	\arrow["{\partial_\ast}"', from=2-1, to=2-2]
\end{tikzcd}\]
\end{center}
is contractible.
\end{prop}

\begin{proof}
     The category $\mathrm{RMod}_O^\mathrm{dp}$ is defined as a category of coalgebras in symmetric sequences, and so colimits in $\mathrm{RMod}_O^\mathrm{dp}$ are computed on the underlying symmetric sequences. This implies that a lift of $\partial_\ast$ to a divided power right module structure preserves colimits, and so it suffices to study the space of lifts when restricted to the full subcategory of representable functors.
     
    By Lemma \ref{lem: ideal implies o ideal} and Proposition \ref{prop: cofree divided power right modules}, the space of divided power right module structures on $\partial_\ast L\Sigma^\infty R_c$ naturally agrees with the space of right module structures by the ideal Tate vanishing hypothesis. 
\end{proof}

 \begin{thm}\label{thm: diff semiadditive classification}
        
    The following are equivalent:  
    \begin{enumerate}
        \item $C$ is differentially $1$-semiadditive.  
        \item There is a symmetric monoidal equivalence \[(\mathrm{Poly}^\omega(C,L\s),\wedge) \simeq (\mathrm{RMod}^\mathrm{< \infty,dp}_O,\circledast).\]  
        \item The derivatives have a contractible space of lifts \[\partial_\ast:\mathrm{Fun}^\omega(C,L\s) \rightarrow \mathrm{RMod}^\mathrm{dp}_{K(\mathrm{CoEnd}(\Sigma^\infty_C))},\]
        and these classify Goodwillie towers. 
        
    \end{enumerate}

    \end{thm}

\begin{proof}\noindent
\begin{itemize}
\item\noindent $(1)\implies(3)$: In Proposition \ref{prop: lifts of derivatives to dp}, we demonstrated the space of such lifts was contractible, and so it remains to show that the lift induces an equivalence between polynomial functors and bounded divided power right modules.
By Proposition \ref{prop: filtration when domain tate vanish}, the unit of
\[
\adjunction{\partial_{\ast}}{\mathrm{Poly}(C^{\omega},L\s)}{\mathrm{RMod}^\mathrm{<\infty,dp}_{O}}{\Theta^\mathrm{dp}}
\]
will preserve colimits, since it is a finite sequence of extensions of functors which preserve colimits. Therefore, it suffices to check the unit is an equivalence on the polynomial approximations of representables. But this holds by Remark \ref{remark: dual norms}, since the spectrum of right module maps computing the fake Goodwillie tower agrees with the divided power right module maps by Tate vanishing. The counit has a similar analysis and can be reduced to divided power right modules which are $L \Sigma^\infty_+ \Sigma_n$ concentrated in degree $n$ for which it is obviously an equivalence.

    \item $(3)\implies(1)$: If a divided power right module structure on $\partial_\ast$ classifies Goodwillie towers, then because it lifts the underlying symmetric sequence of derivatives (by definition) the comonads on $\mathrm{SymSeq}^{< \infty}$  given by $\mathrm{RM}_{K(\mathrm{CoEnd}(\Sigma^\infty_C))}^{\mathrm{dp}}$ and $\partial_\ast \circ \Theta^\mathrm{dp}$ must agree. By Lemma \ref{lem: computation of comonad for diff sa} this implies $C$ is differentially $1$-semiadditive.
  \item $(2)\implies(1)$: We defer the bulk of this to Proposition \ref{prop: abstract dp rmod equiv} where we show that such an assignment must preserve both the operad and right module structure on derivatives up to smashing with an invertible symmetric sequence. The same argument for $(3) \implies (1)$ then applies with minor changes.
    \item $(3)\implies(2)$: Since $(3) \implies (1)$, Proposition \ref{prop: divided power mapping spectrum codomain tate vanish} shows the symmetric monoidal categories generated by the $\partial_\ast L \Sigma^\infty R_c$ for $c \in C^\omega$ agree regardless of whether one considers divided powers or not. Hence, a symmetric monoidal lift of $\partial_\ast$ to divided power right modules exists when restricted to smash products of representables. The same argument of Theorem \ref{thm: product rule} applies to show the entire lift is canonically symmetric monoidal.
\end{itemize}

\end{proof}

Regardless of whether or not $C$ is differentially $1$-semiadditive, the proof of Lemma \ref{lem: computation of comonad for diff sa} constructs a map \[\mathrm{RM}_{K(\mathrm{CoEnd}(\Sigma^\infty_C))}^\mathrm{dp} \rightarrow \partial_\ast \Theta^\Sigma.\] Although nontrivial to prove, we expect this map is a map of comonads. If this is the case, our progress up until the point can be summarized as lifting coalgebraic data on $\partial_\ast F$ up the tower of comonads

\begin{center}
\[\begin{tikzcd}
	{\mathrm{RM}^\mathrm{dp}_{K(\mathrm{CoEnd}(\Sigma^\infty_C))}} \\
	{\partial_\ast \Theta^\Sigma} \\
	{\mathrm{RM}_{K(\mathrm{CoEnd}(\Sigma^\infty_C))}}
	\arrow["{(a)}", from=1-1, to=2-1]
	\arrow["{(c)}", curve={height=-24pt}, from=1-1, to=3-1]
	\arrow["{(b)}", from=2-1, to=3-1]
\end{tikzcd}\]
\end{center}
By Arone-Ching's classification, such coalgebraic data classifies the Goodwillie tower precisely when the relevant map to the middle term is an equivalence.  Thus $(a)$ is an equivalence, if and only $C$ is differentially $1$-semiadditive, and $(b)$ is an equivalence, if and only if $C$ is infinitesimally $1$-semiadditive. These conditions are quite mysterious, particularly differential $1$-semiadditivity. 

 A consequence of Theorem \ref{thm: diff semiadditive classification} is that differential $1$-semiadditivity is detected by the symmetric monoidal category \[(\mathrm{Poly}(C^\omega,L\s),\wedge).\] This is surprising because the definition of differential $1$-semiadditivity is quantified over the objects of $C$ which cannot be determined from the abstract symmetric monoidal category $(\mathrm{Poly}(C^\omega,L\s),\wedge)$.  Infinitesimal $1$-semiadditivity is easier to understand. It is a purely local condition depending only on arbitrarily small neighborhoods of $\ast \in C$. 
 Nevertheless, there is a connection between differential and infinitesimal $1$-semiadditivity which we understand by way of the composite $(c)$ which is simply a product of norm maps. This connection depends not on the category $C$, but only on the symmetric sequence  \[\partial_\ast \mathrm{Id}_C \simeq K(\mathrm{CoEnd}(\Sigma^\infty_C)).\]

\begin{definition}
    An operad is ideal if the map of comonads
    \[\mathrm{RM}^\mathrm{dp}_O \rightarrow \mathrm{RM}_O\]
    is an equivalence.
\end{definition}

\begin{ex}
    The little $d$-disks operad $\Sigma^\infty_+ E_d$ is ideal for $d<\infty$ because its $\Sigma$-finiteness implies the norms which show up in Proposition \ref{prop: cofree divided power right modules} are equivalences.
\end{ex}

\begin{prop}\label{prop: local to global}
The following are equivalent:
\begin{itemize}
    \item  $C$ is infinitesimally $1$-semiadditive
    \item $C$ is differentially $1$-semiadditive and $\partial_\ast(\mathrm{Id}_C)$ is ideal.
\end{itemize}

\end{prop}
\begin{proof}
    The forwards implication is immediate. The backwards implications follows from Theorem \ref{thm: inf semiadd classification} combined with Theorem \ref{thm: diff semiadditive classification}.
\end{proof}

\begin{cor}\label{cor: diff sa algebra cat}
    If $O$ is a reduced levelwise dualizable operad in $L\s$ such that $\mathrm{Alg}_O$ is differentially $1$-semiadditive, then $L\s$ must be $1$-semiadditive.
\end{cor}
\begin{proof}
    Suppose $\mathrm{Alg}_O$ is differentially $1$-semiadditive, then by Theorem \ref{thm: diff semiadditive classification} and Proposition \ref{prop: derivative of bar} there is a contractible space of natural divided power structures on the category of right $O$-modules. Thus $O$ is ideal, and so Proposition \ref{prop: local to global} applies.
\end{proof}

\begin{question}
    What nonideal operads occur as $K(\mathrm{CoEnd}(\Sigma^\infty_C))$ for differentially $1$-semiadditive $C$?
\end{question}

 \subsection{Morita theory and polynomial equivalences} \label{section: polynomial geometry}

In this section we study some elementary properties of differentiable categories which are preserved under symmetric monoidal equivalences of their polynomial categories. Throughout we assume that $C,D$ are differentiable and have identifications $\s(C) \simeq L\s$ and $\s(D) \simeq L'\s$ with some Bousfield localizations of $\s$.

\begin{definition}
    We say $C$ is polynomially equivalent to $D$
    \[C \simeq_P D\]   
    if there exists a symmetric monoidal equivalence \[(\mathrm{Poly}(C^\omega,L\s),\wedge)) \simeq (\mathrm{Poly}(D^\omega,L'\s),\wedge).\]
\end{definition}
 
\begin{lem}\label{lem: preserves homogeneous}
    If $f:C\simeq_P D$ is a polynomial equivalence, then there is a symmetric monoidal equivalence
    \[(L\s,\wedge) \simeq (L'\s,\wedge),\]
    and $f$ preserves the subcategories of $n$-homogeneous and $n$-polynomial functors.
\end{lem}
\begin{proof}
    It is clear that $f$ preserves the property of being $0$-homogeneous, i.e. being constant, because the unit of $\wedge$ is constant. The constant functors recover the original symmetric monoidal categories, and so the first claim follows.

    We demonstrate the second claim inductively. We first show that $1$-excisive functors are preserved. The claim that a polynomial $F$ is $1$-excisive can be expressed as saying that $H$ cannot be written as a series of extensions of functors of the form
    \[(Z \wedge G^{\wedge n})_{h\Sigma_n}\]
    where $n>1$, $Z \in L\s^{B\Sigma_n}$, and $G$ is an arbitrary polynomial. This proposition is invariant under symmetric monoidal equivalences because $Z \wedge (-)$ is the  tensoring of a stable category by a spectrum, $(-)_{h\Sigma_n}$ is a colimit, and $(-)^{\wedge n}$ is the iterated symmetric monoidal product. Hence, $1$-excisivity is preserved from which $1$-homogeneity can be deduced. The argument obviously extends to higher order polynomial functors.
\end{proof}

Recall that the Picard group $\mathrm{Pic}(\mathscr{C},\otimes)$ of a symmetric monoidal category $(\mathscr{C},\otimes)$ is the group of isomorphism classes of objects which are invertible with respect to $\otimes$. The levelwise smash product of operads
\[(O \wedge P )(n) := O(n) \wedge P(n)\]
makes $(\mathrm{Operad}(L\s),\wedge)$ into a symmetric monoidal category \cite[
Proposition 3.9]{pdoperads}. There is a homomorphism\footnote{It is an open question if this homomorphism is an isomorphism, even in the case of $\s$.}
\[\mathrm{Pic}(L\s,\wedge) \rightarrow \mathrm{Pic}(\mathrm{Operad}(L\s),\wedge)\]
\[X \mapsto \mathrm{CoEnd}(X).\]

The image of this homomorphism controls much of the \textit{Morita theory} we are interested in. By Morita theory, we mean the study of equivalence relations of algebraic objects defined via equivalences of categories. Most classically, an associative algebra $A$ is Morita equivalent to an associative algebra $B$, if and only if there is an equivalence of module categories \[\mathrm{LMod}_A \simeq \mathrm{LMod}_B.\]
The category of associative algebras embeds into the category of operads by setting $O(1)=A$ and $O(n)=\ast$ otherwise. Under this embedding the category $\mathrm{LMod}_A$ is identified with $\mathrm{Alg}_O$. Thus, a \textit{Morita equivalence of operads} from $O$ to $P$
\[\mathrm{Alg}_O \simeq \mathrm{Alg}_P\]
is a generalization of a Morita equivalence of associative algebras. We will study Morita equivalences of reduced operads using Goodwillie calculus.

\begin{prop}\label{prop: polynomial equivalence and coend}
    Suppose $f:C \xrightarrow{\simeq_P} D$ is a polynomial equivalence of two compactly generated, differentially dualizable categories. In that case, there is $X \in \mathrm{Pic}(L\s,\wedge)$ such that there is an equivalence of operads
    \[ K(\mathrm{CoEnd}(\Sigma^\infty_D))\simeq K(\mathrm{CoEnd}(\Sigma^\infty_C) )\wedge \mathrm{CoEnd}(X) .\]
 If $F: C^\omega \rightarrow L\s$ is a polynomial, there is an equivalence of right $K(\mathrm{CoEnd}(\Sigma^\infty_D))$-modules \[\partial_\ast(f(F))\simeq \partial_\ast(F) \wedge X \wedge \mathrm{CoEnd}(X) .\]
 In particular, the derivatives are preserved up to smashing with an invertible spectrum:
 \[\partial_\ast (f(F)) \simeq \partial_\ast (F) \wedge X^{\wedge \ast}.\]
\end{prop}
\begin{proof}
    The $X$ in question is the coefficient of $f(\Sigma^\infty_C) \simeq X \wedge \Sigma^\infty_D$ (Lemma \ref{lem: preserves homogeneous}). Since $f$ is a symmetric monoidal equivalence, we have
    \[\mathrm{CoEnd}(\Sigma^\infty_C) \simeq \mathrm{CoEnd}(X \wedge \Sigma^\infty_D). \]
    The right hand side is readily seen to be the operad
    \[\mathrm{CoEnd}(X) \wedge \mathrm{CoEnd}( \Sigma^\infty_D),\]
    and so smashing with $\mathrm{CoEnd}(X^\vee)$ and applying Koszul duality yields the claimed equivalence of operads. Similarly, there is an equivalence of right modules
\[\partial^\ast F \simeq \nat(f(F), (X \wedge \Sigma^\infty_D )^{\wedge \ast})\]
from which one can deduce the result using the relation between Koszul dual derivatives and derivatives.
\end{proof}

\begin{thm}\label{thm:poly morita equiv}
   If $O,P$ are reduced and levelwise dualizable operads in $L\s$, then the following are equivalent
   \begin{enumerate}
       \item $\mathrm{Alg}_O$ is polynomially equivalent to $\mathrm{Alg}_P$.
       \item $\mathrm{Alg}_O$ is equivalent to $\mathrm{Alg}_P$.
       \item There is $X \in \mathrm{Pic}(L\s,\wedge)$ and an equivalence of operads
       \[O \wedge \mathrm{CoEnd}(X) \simeq P.\]
   \end{enumerate}
\end{thm}
\begin{proof}\noindent
   \begin{itemize}
       \item $(1 \implies 3)$: This follows from Theorem \ref{thm: coend in algebras} and Proposition \ref{prop: polynomial equivalence and coend}.
       \item $(3 \implies 2)$: Invertibility of $X$ implies \[\mathrm{CoEnd}(X) \simeq \mathrm{End}(X^\vee),\] so smashing with the $\mathrm{End}(X^\vee)$ algebra $X^\vee$ gives a functor
       \[\mathrm{Alg}_O \rightarrow \mathrm{Alg}_{O \wedge \mathrm{CoEnd}(X)}\]
       \[A \mapsto A \wedge X^\vee.\]
       This is an equivalence because it has an inverse given by smashing with $X$.
       \item $(2 \implies 1)$ Equivalences of categories induce symmetric monoidal equivalences of polynomials by precomposition.
   \end{itemize}
\end{proof}

The reader might wonder why invertible bimodules have not appeared in the Morita theory of operads, as they are essential to classical Morita theory. This absence is because we have restricted to studying reduced operads. It would be interesting to know whether invertible bimodules and coendomorphism operads of invertible spectra are specializations of some third object which completely controls the Morita theory of an arbitrary operad.

To end the section, we revisit our earlier classifications of polynomial functors without the assumption that our classifications lift the Goodwillie derivatives.

\begin{prop}\label{prop: abstract rmod equiv}
    Suppose $C$ is a differentially dualizable category  and $O$ is a reduced levelwise dualizable operad with \[f:(\mathrm{Poly}(C^\omega,L\s),\wedge) \simeq (\mathrm{RMod}_O^{< \infty},\circledast),\] then for some $X \in \mathrm{Pic}(L\s,\wedge)$ there is an equivalence \[O \simeq K(\mathrm{CoEnd}(\Sigma^\infty_C)) \wedge \mathrm{CoEnd}(X^\vee).\]
    If $F: C^\omega \rightarrow L\s$ is a polynomial, then there is an equivalence of right modules 
    \[f(F) \simeq \partial_\ast F \wedge X^\vee \wedge \mathrm{CoEnd}(X^\vee). \]
     In particular, $f$ lifts the Goodwillie derivatives up to smashing with an invertible spectrum:
 \[f(f) \simeq \partial_\ast (F) \wedge (X^\vee)^{\wedge \ast}.\]
\end{prop}
\begin{proof}
    As in Lemma \ref{lem: preserves homogeneous}, the symmetric monoidality forces homogeneous functors to be sent to right modules concentrated in a single degree. In particular, there is some $X \in L\s$ such that \[f(\Sigma^\infty_C) \simeq \mathrm{Triv}^\mathrm{RMod}_O(X),\]
    where we treat $X$ as a symmetric sequence concentrated in degree $1$. This implies there is an equivalence
    \[\mathrm{CoEnd}(\Sigma^\infty_C) \simeq K(O) \wedge \mathrm{CoEnd}(X)\]
    which implies the result.

    Similarly, one sees
    \[\partial^\ast F \simeq \mathrm{RMod}_O(f(F), X^{\circledast \ast})\]
    and the latter is computed to be \[K(f(F)) \wedge X \wedge \mathrm{CoEnd}(X).\]
    Koszul duality and the relation between $\partial_\ast$ and $\partial^\ast$ implies the claim about $f(F)$.
\end{proof}

\begin{prop}\label{prop: abstract dp rmod equiv}
    Suppose $C$ is a differentially dualizable category  and $O$ is a reduced levelwise dualizable operad with \[f:(\mathrm{Poly}(C^\omega,L\s),\wedge) \simeq (\mathrm{RMod}_O^\mathrm{dp,< \infty},\circledast),\] then for some $X \in \mathrm{Pic}(L\s,\wedge)$ there is an equivalence \[O \simeq K(\mathrm{CoEnd}(\Sigma^\infty_C)) \wedge \mathrm{CoEnd}(X^\vee).\]
    If $F: C^\omega \rightarrow L\s$ is a polynomial, then there is an equivalence of right modules 
    \[\mathrm{Forget}^\mathrm{RMod}_O(f(F)) \simeq \partial_\ast F \wedge X^\vee \wedge \mathrm{CoEnd}(X^\vee) \]
     In particular, on underlying symmetric sequences
 \[f(f) \simeq \partial_\ast (F) \wedge (X^\vee)^{\wedge \ast}.\]
\end{prop}
\begin{proof}
    The proof of this result is practically identical to Proposition \ref{prop: abstract rmod equiv}. In particular, homogeneous functors get sent to divided power right modules concentrated in a single degree. There is an equivalence of mapping spectra 
    \[\mathrm{RMod}_O^\mathrm{dp}(-,R^{\circledast n}) \simeq \mathrm{RMod}_O(-,\mathrm{Forget}^\mathrm{RMod}_O (R^{\circledast n}))\]
when $R$ is $L(\Sigma^\infty_+ \Sigma_n) \wedge X^{\wedge n} $ concentrated in degree $n$ by Proposition \ref{prop: divided power mapping spectrum codomain tate vanish}, and so the same argument applies.
\end{proof}

\subsection{Differentially algebraic categories of spaces}

Throughout this section we assume $C$ is compactly generated, differentiable, has a fixed identification $\s(C)\simeq L\s$ to a Bousfield localization of $\s$, and is differentially dualizable. Throughout this section we will also refer to a Bousfield localization $L'$ of $L\s$.

\begin{definition}
    A category $C$ is differentially algebraic if there is a reduced operad $O$ in $L\s$ together with a polynomial equivalence
    \[\mathrm{Poly}(C^\omega, L\s) \simeq \mathrm{Poly}(\mathrm{Alg}_O,L\s).\]
\end{definition}

The purpose of this section is to study differentially algebraic localizations of spaces. More precisely, we introduce the notion of a differential sublocalization of a category $C$, and we characterize those differential sublocalizations of $\T$ which are differentially algebraic.

A \textit{differential subcategory} $D$ of $C$ is a full reflective subcategory $D\subset C$, compactly generated in its own right, such that the localization induces equivalences \[\s(C)\simeq \s(D),\] \[\mathrm{Poly}^\omega(C,\s(C))\simeq \mathrm{Poly}^\omega(D,\s(D)).\]

A \textit{quasireflective} localization $C \rightarrow C[W^{-1}]$ is a localization equipped with a fully faithful embedding $i:C[W^{-1}] \hookrightarrow C$.

Differential subcategories are quite common because Goodwillie calculus only describes the behavior of functors in an infinitesimal neighborhood of $\ast \in C$. If $C$ has some notion of connectivity, we can produce a collection of neighborhoods of $\ast$ by taking the subcategories of $n$-connective objects. Under completeness hypotheses, these subcategories will be differential subcategories.

\begin{example}
    The category $\T^{\geq n}$ of $n$-connective pointed spaces is a differential subcategory of $\T$ \cite[Proposition 4.1]{glasman2018goodwilliecalculusmackeyfunctors}.
\end{example}

\begin{definition}\label{definition: sublocalization}
A differential sublocalization $D[W^{-1}]$ of $C$ is 
\begin{enumerate}
    \item A differential subcategory $D \subset C$;
    \item A quasireflective localization $f:D \rightarrow D[W^{-1}], i:D \rightarrow D[W^{-1}]$ such that $D[W^{-1}]$ is compactly generated, differentiable, and $f$ preserves filtered colimits;
    \item A Bousfield localization $L':L\s \rightarrow L'(L\s)$ and an identification $\s(D[W^{-1}])\simeq L'L\s$
 such that there is a commuting square
    \begin{center}
\[\begin{tikzcd}
	{{D[W^{-1}]}} & {L\s} \\
	{D[W^{-1}]} & {L'L\s}
	\arrow["{\Sigma_D^\infty \circ i}"', from=1-1, to=1-2]
	\arrow["=", from=1-1, to=2-1]
	\arrow["{L'}"', from=1-2, to=2-2]
	\arrow["{\Sigma^\infty_{D[W^{-1}]}}", from=2-1, to=2-2]
\end{tikzcd}\]
    \end{center}
    \item The induced map of operads
    \[ \mathrm{CoEnd}(\Sigma^\infty_D) \rightarrow \mathrm{CoEnd}(L'(\Sigma^\infty_D)) \]
    is the unit map for $L'$-localization.
\end{enumerate}
\end{definition}

\begin{ex}
    The category of simply-connected rational pointed spaces $(\T^{\geq 2})_\mathbb{Q}$ can be obtained as a differential sublocalization of $\T$ by taking simply-connected spaces and inverting rational homotopy equivalences. This localization is differentially algebraic by classical rational homotopy theory \cite{quillenrational,sullivanrational}.
\end{ex}

\begin{ex}
    The category of $v_h$-periodic spaces is a differential sublocalization of $\T$. The quasireflective localization of $\T$ is constructed in \cite[Definition 3.11]{heutsannals}. By \cite[Proposition 3.18]{heutsannals}, $(3)$ is satisfied is and item $(4)$ follows from Sections \ref{section: goodwillie of spaces} and \ref{section:lie algebra models}.
\end{ex}

\begin{prop}
    If $D[W^{-1}]$ is a differential sublocalization of $C$, then there is an equivalence of operads \[\mathrm{CoEnd}(\Sigma^\infty_{D[W^{-1}]}) \simeq L' \mathrm{CoEnd}(\Sigma^\infty_{C}).  \]
    As a consequence, there is an equivalence of symmetric sequences \[\partial_\ast \mathrm{Id}_{D[W^{-1}]} \simeq L' \partial_\ast \mathrm{Id}_C.\]
\end{prop}
\begin{proof}
   The first claim follows essentially by definition. The second claim follows from the first, the Koszul duality between $\mathrm{CoEnd}(\Sigma^\infty_{D[W^{-1}]})$ and $\partial_\ast \mathrm{Id}_{D[W^{-1}]}$, and the interaction of Koszul duality and localizations of Proposition \ref{prop: localizing koszul dual}.
\end{proof}

\begin{lem}\label{lem: diff sa under localization}
   If $D[W^{-1}]$ is a differential sublocalization of a differentially $1$-semiadditive category $C$, then $D[W^{-1}]$ is differentially $1$-semiadditive.
\end{lem}
\begin{proof}
 Any compact object $f(e)$ of $D[W^{-1}]$ is a retract of $f(d)$ for some $d\in D^\omega$. This is because we may write $e$ as a filtered colimit of compact objects of $D$, apply the localization, and then use compactness of $f(e)$ to find such a retraction.

 Hence, $\Sigma^\infty_{D[W^{-1}]} f(e)$ is a retract of $\Sigma^\infty_{D[W^{-1}]} f(d)$. Since this is a differential sublocalization, the latter is equivalent to $L' \Sigma^\infty_D d$. Taking Koszul duals and applying Proposition \ref{prop: localizing koszul dual}, we see that $K(\Sigma^\infty_{D[W^{-1}]} e)$ is a retract of $L'K(\Sigma^\infty_D d ^{\wedge \ast})$.
 
  The passage from $C$ to its differential subcategory $D$ preserves differential $1$-semiadditivity by Theorem \ref{thm: diff semiadditive classification}. Hence, the spectrum $K(\Sigma^\infty_{D[W^{-1}]} d) ^\vee\simeq (\partial_n L \Sigma^\infty R_d)^\vee$ has ideal Tate vanishing. We argue that this implies $L'(K(\Sigma^\infty_D d ^{\wedge}))^\vee$ has ideal Tate vanishing.
  
Suppose we have dualizable $X \in L\s^{BG}$ and $X^\vee$ satisfies ideal Tate vanishing. This implies $L'(X)^\vee$ satisfies ideal Tate vanishing because for $Z \in L'L\s$:
\[((L'X)^\vee \wedge Z )_{hG} \simeq L'\s(L'X,Z)_{hG} \simeq L\s(X,Z)_{hG} \simeq L\s(X,Z)^{hG}\]\[\simeq L\s^{BG}(X,Z) \simeq L'\s^{BG}(L'X,Z) \simeq L'\s(L'X,Z)^{hG}\simeq 
(L'(X)^\vee \wedge Z)^{hG}.\]

Finally, note that the retract of a spectrum with ideal Tate vanishing satisfies ideal Tate vanishing since the Tate construction preserves retractions.
\end{proof}

\begin{prop}
    Suppose $D[W^{-1}]$ is a differential sublocalization of a differentially $1$-semiadditive category $C$, then $D[W^{-1}]$ is differentially algebraic via \[D[W^{-1}] \simeq_P \mathrm{Alg}_O,\]
    if and only if \[O = L'(K(\mathrm{CoEnd}(\Sigma^\infty_C))) \wedge \mathrm{CoEnd}(X)\] for some $X \in \mathrm{Pic}(L'L\s)$ and $L'L\s$ is $1$-semiadditive.
\end{prop}
\begin{proof}
    By Lemma \ref{lem: diff sa under localization},
 $D[W^{-1}]$ is differentially $1$-semiadditive. If it is polynomially equivalent to an algebra category, then by Corollary \ref{cor: diff sa algebra cat} the stabilization $L'L\s$ is $1$-semiadditive.  The homotopy type of the operad is fixed by Proposition \ref{prop: polynomial equivalence and coend}.
\end{proof}

\begin{cor}
    A differential sublocalization of $\T$ is differentially algebraic, if and only if it is infinitesimally $1$-semiadditive.
\end{cor}

\newpage

\bibliographystyle{plain}
\bibliography{main}
\end{document}